\documentclass[journal]{IEEEtran}

\usepackage{mathtools}
\usepackage{amsmath,amsfonts,enumitem,bbm,amsthm}
\usepackage{graphicx}
\usepackage{pgfplots}
\pgfplotsset{compat=1.7}
\usepackage{epstopdf}
\usepackage{lipsum}
\usepackage{color,url}
\usepackage{tabularx}
\usepackage{hyperref}
\usepackage{algorithm,algorithmic,multirow,hhline}
\usepackage{caption}
\usepackage{xr}
\usepackage{xcite}
\usepackage{subfigure} 
\newtheorem{theorem}{Theorem}

\graphicspath{{../Figures/}}

\input{mysymbol.sty}

\newtheorem{assumption}{\hspace{0pt}\bf Assumption}
\newtheorem{lemma}{\hspace{0pt}\bf Lemma}
\newtheorem{proposition}{\hspace{0pt}\bf Proposition}
\newcommand{\INDSTATE}[1][1]{\STATE\hspace{3mm}}
\newcommand{\INDSTATED}[1][1]{\STATE\hspace{6mm}}

\DeclarePairedDelimiter\ceil{\lceil}{\rceil}

\title{Nonstationary Nonparametric Online Learning: Balancing Dynamic Regret and Model Parsimony}
\author{Amrit~Singh~Bedi,
    Alec~Koppel,~Ketan~Rajawat, and Brian M. Sadler% <-this % stops a space
    \thanks{
        A.S. Bedi and A. Koppel contributed equally to this work. They both are with the U.S. Army Research Laboratory, Adelphi, MD, USA (e-mail: amrit0714@gmail.com, alec.e.koppel.civ@mail.mil). K. Rajawat is with the Department of Electrical Engineering,
        Indian Institute of Technology Kanpur, Kanpur 208016, India (e-mail:
        ketan@iitk.ac.in). B. M. Sadler is a senior scientist with the U.S. Army Research Laboratory, Adelphi, MD, USA (email:brian.m.sadler6.civ@mail.mil). A part of this work is submitted to American Control Conference (ACC), Denver, CO, USA, 2020 \cite{BEdi_ACC}.}
      \vspace{-8mm}}

\begin{document}

\maketitle

%abstract (trailing % is important)

\begin{abstract}
An open challenge in supervised learning is \emph{conceptual drift}: a data point begins as classified according to one label, but over time the notion of that label changes. Beyond linear autoregressive models, transfer and meta learning address drift, but require data that is representative of disparate domains at the outset of training. To relax this requirement, we propose a memory-efficient \emph{online} universal function approximator based on compressed kernel methods. Our approach hinges upon viewing non-stationary learning as online convex optimization with dynamic comparators, for which performance is quantified by dynamic regret.

 Prior works control dynamic regret growth only for linear models. In contrast, we hypothesize actions belong to reproducing kernel Hilbert spaces (RKHS).  We propose a functional variant of online gradient descent (OGD) operating in tandem with greedy subspace projections. Projections are necessary to surmount the fact that RKHS functions have complexity proportional to time.
 
For this scheme, we establish sublinear dynamic regret growth in terms of both loss variation and functional path length, and that the memory of the function sequence remains moderate. Experiments demonstrate the usefulness of the proposed technique for online nonlinear regression and classification problems with non-stationary data.

\end{abstract}
%
%
%!TEX root = Paper.tex
%%%%%%%%%%%%%%%%%%%%%%%%%%%%%%%%%%%%%%%%%%%%%%%%%%%%%%%%%%%%%%%
%%%     S  E  C  T  I  O  N    %%%%%%%%%%%%%%%%%
%%%%%%%%%%%%%%%%%%%%%%%%%%%%%%%%%%%%%%%%%%%%%%%%%%%%%%%%%%%%%%%
\section{Introduction}\label{sec:intro}
%
%{\bf Non-Stationarity} 
A well-known challenge in supervised learning is \emph{conceptual drift}: a data point begins as classified according to one label, but over time the notion of that label changes. For example, an autonomous agent classifies the terrain it traverses as grass, but as the sun sets, the grass darkens. The class label has not changed, but the data distribution has. Mathematically, this situation may be encapsulated by supervised learning with time-series data. 
Classical approaches assume the current estimate depends linearly on its past values, as in autoregressive models \cite{akaike1969fitting}, for which parameter tuning is not difficult \cite{brillinger1981time}. While successful in simple settings, these approaches do not apply to classification, alternate quantifiers of model fitness, or universal statistical models such as deep networks \cite{haykin1994neural} or kernel methods \cite{berlinet2011reproducing}. Such modern tools are essential to learning unknown dynamics when assumptions of linear additive Gaussian noise in system identification are invalid, for instance \cite{aastrom1971system,haykin1997adaptive}.

In the presence of non-stationarity, efforts to train models beyond linear have focused on recurrent networks \cite{jaeger2002tutorial}, but such approaches inherently require the temporal patterns of the past and future to be similar. In contrast, transfer learning seeks to adapt a statistical model trained on one domain to another \cite{pan2010survey}, but requires (1) data to be available in advance of training, and (2) a priori knowledge of when domain shifts happen, typically based on hand-crafted features. Meta-learning overcomes the need for hand-crafted statistics of domain shift by collecting experience over disparate domains and discerning decisions that are good with respect to several environments' training objectives \cite{thrun2012learning}. Combining such approaches with deep networks have yielded compelling results recently \cite{andrychowicz2016learning,finn2017meta}, although they still require (1) offline training. Hence, in domains where a priori data collection is difficult, due to, e.g., lack of cloud access or rapid changes in the environment, transfer and meta-learning do not apply. In these instances, \emph{online training} is required. 

For online training, there are two possible approaches to define learning in the presence of non-stationarity: expected risk minimization \cite{Vapnik1995,friedman2001elements}, and online convex optimization (OCO) \cite{shalev2012online}. The former approach, due to the fact the data distribution is time-varying distribution, requires the development of stochastic algorithms whose convergence is attuned to temporal aspects of the distribution such as mixing rates \cite{borkar2009stochastic,mohri2010stability}. Although mixing rates are difficult to obtain, they substantially impact performance \cite{nagabandi2018deep}. To mitigate these difficulties, we operate within online convex optimization.

{\bf Online convex optimization} OCO formulates supervised learning in a distribution-free manner \cite{shalev2012online}. At each time, a learner selects action $f_t$ after which an arbitrary convex cost $\ell_t:\ccalH
\times \reals^p \rightarrow \reals$ is evaluated as well as parameters $\bbx_t\in\ccalX\subset\reals^p$ of the cost $\ell_t$, i.e., the learner suffers cost  $\ell(f_t(\bbx_t))$. 
Typically, actions $f_t$ are defined by a parameter vector. In contrast, we hypothesize actions $f_t\in\ccalH$ belong to a \emph{function space} $\ccalH$ motivated by nonparametric regression whose details will be deferred to later sections \cite{wasserman2006all}. 
In classic OCO, one compares cost with a single best action in hindsight; however, with non-stationarity, the quintessential quantifier of performance is instead \emph{dynamic regret}, defined as the cost accumulation as compared with a best action at each time:
%
%To do so, we define a more stringent performance metric called 
%
 \begin{align}\label{dynamic}
\textbf{Reg}^D_T=& \sum_{t=1}^{T}L_{t}(f_{t}(\bbS_t))-\sum_{t=1}^{T}L_{t}(f^\star_t(\bbS_t)) \nonumber 
\\
\text{where}\ \ \quad f^\star_t= & \argmin_{f \in \ccalH} L_t(f(\bbS_t)).
\end{align}
%
%\begin{align}
%
OCO concerns the design of methods such that $ \textbf{Reg}_T^D$ grows sublinearly in horizon $T$ for a given sequence $f_t$, i.e., the average regret goes to null with $T$ ( referred to as no-regret \cite{Zinkevich2003}). %Such methods are called no-regret algorithms, and a canonical approach to achieving no-regret is online gradient descent \cite{Zinkevich2003}. 
Observe that $\textbf{Reg}^D_T$, in general, decouples the problem into $T$ time-invariant optimization problems since the minimizer is inside the sum. However, in practice, temporal dependence is intrinsic, as in wireless communications \cite{heath1998simple}, autonomous path planning \cite{vernaza2008online,turchetta2016safe}, or obstacle detection \cite{wurm2010octomap}. Thus, we define \eqref{dynamic} in terms of an augmented cost-data pair $(L_t, \bbS_t)$ which arises from several times, either due to new or previously observed pairs $(\ell_t, \bbx_t)$. Specifications of $L_t$ to time-windowing or batching are discussed in Sec. \ref{sec:problem}.
%
%
%
%{\bf Context} \blue{Here's where related works are explained. We'll make a table of results: contrast static/dynamic regret for convex/strongly convex cases in parametric/nonparametric settings. My initial stab at this is in Table \ref{tab1}.}
%
%
%\red{
%Observe that for $P=1$ this simplifies to 
%%
% \begin{align}\label{dynamic0}
% %
%\textbf{Reg}_T^D=\sum_{t=1}^{T}\ell_{t}(f_{t}(\bbx_t))-\sum_{t=1}^{T}\ell_{t}(f^\star_t(\bbx_t))
%\end{align}
%%
%which is the notion of non-parametric dynamic regret in \cite{shen2019random}. However, since the optimal decouples across time, one may address \eqref{dynamic0} through repeated application of Representer Theorem together with usual parametric OCO approaches, obviating the need for operating in the RKHS. }
%\blue{this will get removed but we'll carefully explain what's going on in \cite{shen2019random}}
%
%%%
%%The definition of dynamic regret in \eqref{dynamic0} do not extend well for the practical setting since the dynamic regret is minimized to achieve a dynamic optimal function $f^\star_t$ which is defined for only one data point $\bbx_t$. 
%%To motivate the dynamic regret for non-linear non-stationary online learning problems, we consider the following loss function
%%
%
%%%%%%%%%%%%%%%%%%%%%%%%%%%%%%%%%%%%%%%%%%%%%%%%%%%%%%%%%%%%%%%%%%%%%%%%%%%%%%%%%%%%%%%%%%%%%%%%%%

\begin{table*}\centering\hfill
\renewcommand{\arraystretch}{1.7}
%\begin{tabular}{llll||}
%\cline{1-5}
%%\footnote{$\tilde{\mathcal{O}}(\cdot)$ hides $\log T$ factors.} 
%%
%\cline{1-5}
\resizebox{\textwidth}{!}{\begin{tabular}{|l|l|l|l|l|}
\hline
Reference & Regret Notion & Loss & Function Class & Regret Bound \\\hline
\cite{Zinkevich2003,hall2015online}      & $\sum_{t=1}^T\ell_t(\bbw_t)-\ell_t(\bbw_t)$          & Convex          & Parametric             & $\mathcal{O}\left(\sqrt{T}(1+{W}_T)\right)$         \\
%
%\cite{hall2015online}       & $\sum_{t=1}^T\ell_t(\bbx_t)-\ell_t(\bbu_t)$           & Convex          & Parametric             & $\mathcal{O}\left(\sqrt{T}(1+{W}_T)\right)$         \\
%
\cite{besbes}              &$\sum_{t=1}^T\mathbb{E}\left[\ell_t(\bbw_t)\right]- \ell_t(\bbw^\ast_t)$  \qquad   & Convex& Parametric  &  $\mathcal{O}\Big(T^{2/3}{(1+W_T)}^{1/3}\Big)$ \\
\cite{besbes}       & $\sum_{t=1}^T\mathbb{E}\left[\ell_t(\bbw_t)\right]- \ell_t(\bbw^\ast_t)$          & Strongly convex          & Parametric             & $\mathcal{O}(\sqrt{T{(1+W_T)}})$         \\
\cite{jadbabaie2015online}      & $\sum_{t=1}^T\ell_t(\bbw_t)- \ell_t(\bbw^\ast_t)$          & Convex          & Parametric             &  ${\mathcal{O}}\Big(\sqrt{D_T+1}+ \min\Big\{\sqrt{({D}_{T}+1) {V}_{T}} ,  [({D}_{T}+1)W_{T} T]^{1/3}\Big\}\Big)$  \\
\cite{mokhtari2016online,bedi}      & $\sum_{t=1}^T\ell_t(\bbw_t)- \ell_t(\bbw^\ast_t)$          & Strongly convex          & Parametric             & $\mathcal{O}(1+{{W}_T})$         \\      
\cite{shen2019random} & $\sum_{t=1}^{T}\ell_{t}(f_{t}(\bbw_t))-\sum_{t=1}^{T}\ell_{t}(f^\star_t(\bbw_t))$          & Convex          & Nonparametric             & $\mathcal{O}\left(T^{\frac{2}{3}}V_T^{1/3}\right)$        \\      
\textbf{This Work }& $\sum_{t=1}^{T}L_{t}(f_{t}(\bbS_t))-\sum_{t=1}^{T}L_{t}(f^\star_t(\bbS_t))$          & Convex          & Nonparametric             &   \red{$\mathcal{O}\left(T^{2/3}V_T^{1/3}+\epsilon T^{2/3}V_T^{-1/3}\right)$}\\
\textbf{This Work }& $\sum_{t=1}^{T}L_{t}(f_{t}(\bbS_t))-\sum_{t=1}^{T}L_{t}(f^\star_t(\bbS_t))$          & Convex          & Nonparametric             & \textbf{\red{${\mathcal{O}\left(1+ T\sqrt{\epsilon}+W_T\right)}$}}  \\
\textbf{This Work }& $\sum_{t=1}^{T}L_{t}(f_{t}(\bbS_t))-\sum_{t=1}^{T}L_{t}(f^\star_t(\bbS_t))$          & Strongly convex          & Nonparametric             & \textbf{\red{${o\left(1+ T\sqrt{\epsilon} +W_T\right)}$}} \\
\hline        
\end{tabular}}
%%
%%

%This work \qquad                                &$\sum_{t=1}^Tf_t(\bbx_t)- f_t(\bbx^\ast_t)$         &  Strongly convex\qquad &Convex  & $\mathcal{O}\Big(1+{{C}_T}\Big)$ \\
%%
%\cline{1-5}
%%Reference & Eqns.~\eqref{result1}-\eqref{result2}  & Eqn.~\eqref{result11}-\eqref{result22}&  Eqn.~\eqref{result_37}]  & Eqn.~\eqref{result_371} \\
%\end{tabular}
\caption{\small Summary of related works on dynamic online learning. {In this work, we have derived the dynamic regret both in terms of $V_T$ and $W_T$ with an additional compression parameter $\epsilon$ to control complexity of nonparametric functions, which permits sublinear regret growth for dynamic regret in terms of $W_T$  under selection $\epsilon=\mathcal{O}\left(T^{-\alpha}\right)$ with $\alpha\in(0,\frac{1}{p}]$, where $p$ is the parameter dimension. {Note that for the strongly convex case  with $\epsilon=0$, we obtain $o(1+W_T)$ which is better than its parametric counterpart obtained in \cite{mokhtari2016online}}.{ In particular, we just need the compression budget to be $\epsilon<\mathcal{O}\left(\left(\frac{W_T}{T}\right)^2\right)$ to achieve $\mathcal{O}(1+W_T)$ dynamic regret.}}\normalsize}
%
%\blue{Three things: (1) Is the reference to the dynamic regret bound in \cite{shen2019random} correct? Didn't we clarify that it should be independent of $T$? Just curious. (2) Are there no other works that establish bounds in terms of $V_T$? (3) Some remark about the improvements when we have strong convexity vs. weak convexity are needed to explain our contributions to this table... } }
\label{tab1}
\end{table*}

\subsection{Related Work and Contributions}
OCO seeks to develop algorithms whose regret grows sublinearly in time horizon $T$. In the static case, the simplest approach is online gradient descent (OGD), which selects the next action to descend along the gradient of the loss at the current time. OGD attains static regret growth $\ccalO(T^{1/2})$ when losses are convex \cite{Zinkevich2003} and $\ccalO(\log{T})$ strongly convex \cite{hazan2007logarithmic}, respectively. See Table \ref{tab1} for a summary of related works.

The plot thickens when we shift focus to dynamic regret: in particular, \cite{besbes} establishes the impossibility of attaining sublinear dynamic regret, meaning that one \emph{cannot} track an optimizer varying arbitrarily across time, a fact discerned from an optimization perspective in \cite{simonetto2016class}. 
Moreover, \cite{besbes} shows that dynamic regret to be an irreducible function of quantifiers of the problem dynamics called the cost function variation $V_T$ and variable variation $W_T$ (definitions in Sec. \ref{sec:problem}). Thus, several works establish sublinear growth of dynamic regret up to factors depending on $V_T$ and $W_T$, i.e., $\ccalO(T^{1/2}(1+W_T ) )$ for OGD or mirror descent with convex losses \cite{Zinkevich2003,hall2015online}, more complicated expressions that depend on $D_T$, the variation of instantaneous gradients \cite{jadbabaie2015online}, and $\mathcal{O}(1+{{W}_T})$ for strongly convex losses \cite{mokhtari2016online}.

The aforementioned works entirely focus on the case where decisions define a linear model $\bbw_t\in\ccalW\subset\reals^p$, which, by the estimation-approximation error tradeoff \cite{friedman2001elements}, yield small dynamic regret at the cost of large approximation error. Hypothetically, one would like actions to be chosen from a universal function class such as a deep neural network (DNN) \cite{tikhomirov1991representation,scarselli1998universal} or RKHS \cite{park1991universal} while attaining no-regret. It's well-understood that no-regret algorithms often prescribe convexity of the loss with respect to actions as a prerequisite \cite{shalev2012online}, thus precluding the majority of DNN parameterizations.
 %\red{Mengdi mentioned some new Sham Kakade paper about no-regret algorithms with function approximation at the UW Data Science workshop -- We need to dig that up and check the assumptions. Should mention here.}
While exceptions to this statement exist \cite{amos2017input}, instead we focus on parameterizations defined in nonparametric statistics \cite{wasserman2006all}, namely, RKHS \cite{berlinet2011reproducing}, due to the fact they yield universality \emph{and} convexity. Doing so allows us to attain methods that are \emph{both} no-regret and universal in the non-stationary setting. We note that \cite{shen2019random} considers a similar setting based on random features \cite{rahimi2008random}, but its design cannot be tuned to the learning dynamics; and yields faster regret growth.

{\bf Contributions}
We propose a variant of OGD adapted to RKHS. A challenge for this setting is that the function parameterization stores all observations from the past \cite{Kivinen2004}, via the Representer Theorem \cite{scholkopfgeneralized}. To surmount this hurdle, we greedily project the functional OGD iterates onto subspaces constructed from subsets of points observed thus far which are $\epsilon$-close in RKHS norm  (Algorithm \ref{alg:soldd}), as in \cite{koppel2019parsimonious,koppelconsistent}, which allows us to explicitly tune the sub-optimality caused by function approximation, in contrast to random feature expansions \cite{rahimi2008random}.
 Doing so allows us to establish sublinear dynamic regret in terms of both the loss function variation (Theorem \ref{theorem:dynamic_cost}) and function space path length (Theorem \ref{theorem:dynamic_path}). Moreover, the learned functions yield finite memory (Lemma \ref{theorem_model_order}). In short, we derive a tunable tradeoff between memory and dynamic regret, establishing for the first time global convergence for a universal function class in the non-stationary regime (up to metrics of non-stationarity \cite{besbes}). 
These results translate into experiments in which one may gracefully address online nonlinear regression and classification problems with non-stationary data, contrasting alternative kernel methods and other state of the art online learning methods. 
%or efforts to train off-the-shelf deep networks online. 

%!TEX root = Paper.tex
%%%%%%%%%%%%%%%%%%%%%%%%%%%%%%%%%%%%%%%%%%%%%%%%%%%%%%%%%%%%%%%
%%%   S  U  B  S  E  C  T  I  O  N    %%%%%%%%%%%%%%%%%
%%%%%%%%%%%%%%%%%%%%%%%%%%%%%%%%%%%%%%%%%%%%%%%%%%%%%%%%%%%%%%%
\section{Non-Stationary Learning}\label{sec:problem}
In this section, we clarify details of the loss, metrics of non-stationarity, and RKHS representations that give rise to the derivation of our algorithms in Sec. \ref{sec:algorithm}.
To begin, we assume Tikhonov regularization, i.e.,  $\ell_{t}(f(\bbx)):=\check{\ell}_t(f(\bbx))+(\lambda'/{2})\|f \|^2_{\ccalH}$ for some convex function $\check{\ell}_t : \ccalH \times \ccalX \rightarrow \reals$, which links these methods to follow the \emph{regularized leader} in \cite{shalev2012online}.

{\bf Time-Windowing and Mini-Batching} To address when the solutions $f^\star_t$ are correlated across time or allow for multiple samples per time slot, we define several augmentations of loss data-pairs $(\ell_t, \bbx_t)$.
%
%Examples of cost functions arising in various settings include:

%\begin{itemize}
%
(i) Classical loss: $L_t = \ell_t$ and $\bbS_t = \bbx_t$, and the minimization may be performed over a single datum. In other words, the action taken depends only on the present, as in fading wireless communication channel estimation.
%
  %: the instantaneous cost is the average of the current and previous $P$ losses:% where the loss function is defined as 
%
\begin{align}\label{eq:loss}
&\text{$(ii)$  $H$-Window}: L_t(f(\bbS_t))\!=\sum\limits_{\tau=t-H+1}^{t}\!\!\!\!\!\!\!\ell_\tau(f(\bbx_{\tau})) \; , \nonumber \\ \
&\text{ $(iii)$  Mini-batch} :  L_t(f(\bbS_t))= \sum_{i=1}^B\! \ell_t (f(\{\bbx_{t}^i\}_{i=1}^B)). %\label{batch}
\end{align}
The first cost $L_t(f(\bbS_t))$ in \eqref{eq:loss}{\it (ii)} for each time index $t$ consists $H-1$ previous cost-data pairs $\{\ell_\tau, \bbx_\tau\}_{\tau=t-P+1}^{t-1}$ and new cost-data pair $(\ell_t,\bbx_t)$, where we denote samples $\{\bbx_\tau\}$ in this time window as $\bbS_t$. $H=1$ simplifies to dynamic regret as in \cite{shen2019random}. \eqref{eq:loss} is useful for, e.g., obstacle avoidance, where obstacle is  correlated with time.  
%
%
%%
%\begin{align}\label{batch}
%%
%L_t(f(\bbS_t)) = \sum_{i=1}^B \ell_t 
%(f(\{\bbx_{t}^i\}_{i=1}^B)).
%%
%\end{align}
%
Typically, we distinguish between the sampling rate of a system and the rate at which model updates occur. If one takes $B$ samples per update, then mini-batching is appropriate, as in \eqref{eq:loss}{\it (iii)} . 
%
%$4)$ Hybrid loss function, combining the windowed and mini-batch loss 
%function in \eqref{eq:loss} and \eqref{batch}. 
%
In this work, we focus windowing in \eqref{eq:loss}{\it (ii)}, i.e., $H>1$. Further, instead of one point at $t$ given by $\bbx_t$, one may allow $B$ points $\{\bbx_i\}_{i=1}^B$, yielding a hybrid of \eqref{eq:loss}{\it (ii) - {\it (iii)}}. Our approach naturally extends to mini-batching. For simplicity, we focus on $B=1$. We denote $\check{L}_t$ as the component of \eqref{eq:loss} without regularization.

{\bf Metrics of Non-Stationarity} With the loss specified, we shift focus to illuminating the challenges of non-stationarity.
%
%% is generalization of the definition in \eqref{dynamic0}. For $P=1$ and $B=1$, the definition in \eqref{dynamic} becomes same as that in \eqref{dynamic0}. Further, from the definition of loss function we can write  $L_t(f(\bbS_t))=\check{L}_t(f(\bbS_t))+\frac{\lambda'}{2}PB\|f\|_{\mathcal{H}}^2$ with notation $\check{L}_t(f(\bbS_t)):=\sum\limits_{\tau=t-P+1}^{t}\sum\limits_{i=1}^{B}\check\ell_\tau(f(\bbx_{\tau}^i))$.
%%   %
%%   
%The statistical interpretation of dynamic regret, as compared with the static regret, is that the learner is being evaluated on \emph{non-stationary} but independent data, as compared with the classical independent static/identically distributed setting (i.d. but \emph{not} i.i.d.). In non-stationary stochastic programming, one must tie convergence criteria to the mixing rates of the distribution \cite{borkar2006stochastic,mohri2010stability}; in the dynamic regret setting, by contrast, these information-theoretic technicalities are mitigated, and one may tie no-regret to the rate of change of the loss function and the optimizer at each time slot.
% 
As mentioned in Sec. \ref{sec:intro}, \cite{besbes} establishes that designing no-regret [cf. \eqref{dynamic}] algorithms against dynamic comparators when cost functions change arbitrarily is impossible. Moreover, dynamic regret is shown to be an irreducible function of fundamental quantifiers of the problem dynamics called cost function variation and variable variation, which we now define. Specifically, the cost function variation $\text{Var}(L_1,L_2,\cdots,L_T)$ tracks the largest loss drift across time:
%
% conditions are required to get some meaningful results and proceed with the analysis, 
%   
%
%We introduce the following measures for expressing the variations and bounds in dynamic regret as in \cite{besbes} and \cite{mokhtari2016online}. We restrict the class of admissible cost function  similar to the one proposed in \cite{besbes}. This restriction controls the behavior of loss functions changing from one time instant to other. To proceed, let us first define the notion of variation as
%
\begin{align}\label{variation}
& \text{Var}(L_1,L_2,\cdots,L_T):=\sum\limits_{t=2}^{T}|L_t-L_{t-1}| \; , \nonumber \\
& \quad \mathcal{V}:=\Big\{\{L_t\}_{t=1}^T\ ; \ \sum_{t=2}^{T}|L_{t}-L_{t-1}|\leq V_T\Big\},
\end{align}
where $|L_t-L_{t-1}|:=\sup_{f\in\mathcal{H}}|L_t(f(\bbS))-L_{t-1}(f(\bbS))|$ for all $\bbS {\in} \boldsymbol{\mathcal{X}}$ and denote $\mathcal{V}$ as the class of convex losses bounded by $V_T$ for any set of points $\bbS\in\boldsymbol{\mathcal{X}}$. 
%%
%\begin{align}\label{variation}
%%
%\mathcal{V}:=\Big\{\{L_t\}_{t=1}^T\ ; \ \sum_{t=2}^{T}|L_{t}-L_{t-1}|\leq V_T\Big\}
%%
%\end{align}
%
%Observe that $\{V_t\}$ is a non-decreasing sequence of real numbers
%
% defines the variation budget over $T$ iterations. Here  $V_T$ quantifies the variation of loss functions $L_t$ over $t=1$ to $T$
% 
%\red{We can provide different examples of the possible $V_T$ later for better clarification. For instance, it allows continuous, discrete and non-constant changes over time. }
Further define the variable variation $W_T$ as
\begin{align}\label{vtdy}
W_T := \sum_{t=1}^{T} \|f_{t+1}^\star-f_{t}^\star\|_{\ccalH}
\end{align}
which quantifies the drift of the optimal function $f_t^\star$ over time $t$. One may interpret \eqref{variation} and \eqref{vtdy} as the distribution-free analogue of mixing conditions in stochastic approximation with dependent noise in \cite{borkar2006stochastic} and reinforcement learning \cite{karmakar2017two}. Then, our goal is to design algorithms whose growth in dynamic regret \eqref{dynamic} is sub-linear, up to constant factors depending on the fundamental quantities \eqref{variation}-\eqref{vtdy}. %\blue{I think it's fine to use $\bbx$ in the definition of  $V_t$, provided $\ccalX$ is bounded. This is due to the fact that $\langle f, \kappa(\bbx,\cdot) \rangle= f(\bbx)$, so if $\ccalX$ is bounded, then Cauchy Schartz provides a link to the sup very naturally as being modified by a multiplicative constant $X:=\sup \bbx$.} 

\section{Algorithm Definition}
\label{sec:algorithm}
%
%\blue{Throughout this section I corrected $\check{L}_t$ to $L_t$ since this is the manner in which the proofs/theorems are stated. Please make sure I didn't screw anything up in making that switch. \red{I think that substitution is not correct because of the defintiion $\ell_{t}(f(\bbx)):=\check{\ell}_t(f(\bbx))+(\lambda'/{2})\|f \|^2_{\ccalH}$. This discrepancy is due to the regularized and un regularized loss. Since regularized loss is used to define the regret, that is why $\check{L}$} was used to represent the gradient component.  But we can try to swap the things, but then in the current form, it is not correct in this section. }
%\blue{Note the definitions \eqref{dynamic} and \eqref{eq:loss} use $L_t$, not $\check{L}_t$. So we need to be consistent with the definition of regret. That's why I changed it. We only really need to }
{\bf Reproducing Kernel Hilbert Space} With the metrics and motivation clear, we detail the function class $\ccalH$ that defines how decisions $f_t$ are made. As mentioned in Sec. \ref{sec:intro}, we would like one that satisfies universal approximation theorems \cite{park1991universal}, i.e., the hypothesis class containing the Bayes optimal \cite{friedman2001elements}, while also permitting the derivation of no-regret algorithms through links to convex analysis. RKHSs \cite{berlinet2011reproducing} meet these specifications, and hence we shift to explaining their properties. A RKHS is a Hilbert space equipped with an inner product-like map called a kernel $\kappa: \ccalX \times \ccalX \rightarrow \reals$ which satisfies
\begin{align} \label{eq:rkhs_def}
\textrm{(i)} \  \langle f , \kappa(\bbx, \cdot) \rangle _{\ccalH} = f(\bbx) \;,\qquad
\textrm{(ii)} \ \ccalH = \overline{\text{span}\{ \kappa(\bbx , \cdot) \}}
\end{align}
 %
 % %
 {for all } $\bbx \in \ccalX$. Common choices $\kappa$ include the polynomial kernel and the radial basis kernel, i.e., $\kappa(\bbx,\bbx') = \left(\bbx^T\bbx'+b\right)^c $ and $\kappa(\bbx,\bbx') = e^{ - ({\lVert \bbx - \bbx' \rVert_2^2})/{2c^2} }$, respectively, where $\bbx, \bbx' \in \ccalX$.
For such spaces, the function $f^\star(\bbx)$ that minimizes the sum, ${R}(f ; \{\bbx_{{t}} \}_{t=1}^T )= \frac{1}{T}\sum_{t=1}^T \ell_t(f ; (\bbx_{{t}}))$, over $T$ losses satisfies the Representer Theorem \cite{kimeldorf1971some,scholkopfgeneralized}. Specifically, the optimal $f$ may be written as a weighted sum of kernels evaluated \emph{only} at training examples as
%\begin{align}\label{eq:kernel_expansion}
$f(\bbx) = \sum_{t=1}^T w_{{t}} \kappa(\bbx_{{t}}, \bbx)$,
%\end{align}
%%
where $\bbw = [w_1, \cdots, w_T]^T \in \reals^T$ denotes a set of weights. We define the upper index $T$ as the \emph{model order}. 

One may substitute this expression into the minimization of $R(f)$ to glean two observations from the use of RKHS in online learning: the latest action is a weighted combination of kernel evaluations at previous points, e.g., a mixture of Gaussians or polynomials centered at previous data $\{\bbx_u\}_{u\leq T}$; and that the function's complexity becomes unwieldy as time progresses, since its evaluation involves all past points. Hence, in the sequel, we must control both the growth of regret \emph{and} function complexity.

{\bf Functional Online Gradient Descent} %\label{subsec:sgd}
Begin with functional online gradient method, akin to \cite{Kivinen2004}: %i.e., \emph{functional online gradient descent} (FOGD). %Given data $\{\bbx_\tau\}_{\tau =t-P+1 }^t $ at time $t$, differentiate $L_t(f)$ as:
%
%\begin{align}\label{eq:stochastic_grad}
%\nabla_f L_t(f(\bbS_t))(\cdot) 
%= \sum\limits_{\tau=t-P+1}^{t}\frac{\partial \ell_t(f(\bbx_\tau))}{\partial f(\bbx_\tau)}\frac{\partial f(\bbx_\tau)}{\partial f}(\cdot)
%\end{align}
%%
%where we have applied the chain rule. Now, define the short-hand notation {$\ell_\tau'(f(\bbx_\tau^i)): ={\partial \ell_\tau(f(\bbx_\tau))}/{\partial f(\bbx_\tau)} $} for the derivative of {$\ell_\tau(f(\bbx_\tau))$} with respect to its scalar argument $f(\bbx_\tau)$ evaluated at $\bbx_\tau$. To evaluate the second term on the right-hand side of \eqref{eq:stochastic_grad}, differentiate both sides of the expression defining the reproducing property of the kernel [cf. \eqref{eq:rkhs_def}(i)] with respect to $f$ to obtain
%%
%\begin{align}\label{eq:stochastic_grad2}
%\frac{\partial  f(\bbx_\tau)}{\partial f} = \frac{\partial \langle f , \kappa(\bbx_\tau, \cdot) \rangle _{\ccalH}}{\partial f}
%= \kappa(\bbx_\tau,\cdot)
%\end{align}
%%
%Based on \eqref{eq:stochastic_grad2}, functional online gradient descent (FOGD) for the kernelized $\lambda'$-regularized minimization problem in \eqref{static} takes the form
%
%
\begin{align}\label{eq:sgd_hilbert}
f_{t+1} =&(1-\eta H\lambda' ) f_{t} - \eta \nabla_f \check{L}_t (f_{t}(\bbS_t)) \nonumber 
\\ 
=&(1-\eta H\lambda' ) f_{t} - \eta  \sum\limits_{\tau=t-P+1}^{t}\check{\ell}_\tau'(f_t(\bbx_\tau)) \kappa(\bbx_\tau,\cdot) \; ,
\end{align}
where the later equality makes use of the definition of ${L}_t (f_{t}(\bbS_t))$ [cf. \eqref{eq:loss}], the chain rule, and the reproducing property of the kernel \eqref{eq:rkhs_def} -- see \cite{Kivinen2004}.
%\blue{I think it makes sense to define the time-windowed loss so that we don't accumulate a factor of $P$ in the regularizer. That seems unhelpful/unnecessary, and would require us to use really small step-sizes.}
%
We define $\lambda=\lambda'H$. Step-size $\eta> 0$ is chosen as a small constant -- see Section \ref{sec:convergence}. We require that, given $\lambda > 0$, the step-size satisfies $\eta < 1/\lambda$ and initialization $f_0 = 0 \in \ccalH$. Given this initialization, one may apply induction and Representer Theorem \cite{scholkopfgeneralized} to write the function $f_t$ at time $t$ as a weighted kernel expansion over past data $\bbx_t$ as
\begin{align}\label{eq:kernel_expansion_t}
f_t(\bbx) 
= \sum_{u=1}^{t-1} w_u \kappa(\bbx_u, \bbx)
= \bbw_t^T\boldsymbol{\kappa}_{\bbX_t}(\bbx) \; .
\end{align}
On the right-hand side of \eqref{eq:kernel_expansion_t} we have introduced the notation $\bbX_t = [\bbx_1, \cdots, \bbx_{t-1}]\in \reals^{p\times (t-1)}$, $\boldsymbol{\kappa}_{\bbX_t}(\cdot) = [\kappa(\bbx_1^1,\cdot),\cdots,\kappa(\bbx_{t-1},\cdot)]^T$, and $\bbw_t=[w_1 ; \cdots ; w_{t-1}]$. We may glean from  \eqref{eq:kernel_expansion_t}, that the functional update \eqref{eq:sgd_hilbert} amounts to updates on the data matrix $\bbX$ and coefficient $w_{t+1}$:
%
%\blue{Please double check. I changed some indexing}
%\blue{Please double check \eqref{eq:param_update} - \eqref{eq:coefficient_horizon}. It looked weird to me. I tried to correct it but am not completely sure. We also need to be careful to only use bold for vectors, and not for scalars.}
%
\begin{align}\label{eq:param_update} 
\bbX_{t+1} = [\bbX_t, \;\; \bbx_t],\;\;\;\; w_{t+1} = -\eta\check{\ell}_t'(f_t(\bbx_t)) \; , 
\end{align}
%}
In addition, we need to update the last $H-1$ weights over range $\tau=t-H+1$ to $t-1$:%. Thus, the overall past coefficients are updated as follows
\begin{align}\label{eq:coefficient_horizon}
w_\tau \!=\! \begin{cases*}\!\!
%
%w_{\tau}=(1 - \eta \lambda) w_{\tau-1}-\eta\sum\limits_{u=\tau}^{t-1}\check{\ell}_k'(f_t(\bbx_u).
  (1\!-\!\eta \lambda) w_\tau \!-\! \eta\check{\ell}_\tau'(f_t(\bbx_\tau))  \text{ for } \tau \in \{t\!-\!H\!+\!1,\dots,t\!-\!1\} \\
  \!\!(1\!-\!\eta \lambda) w_\tau \qquad\quad\qquad \ \ \  \text{ for } \tau < t-H+1.
\end{cases*}
\end{align}
Observe that \eqref{eq:param_update} causes $\bbX_{t+1}$ to have one more column than $\bbX_t$. Define the \emph{model order} as number of points (columns) $M_t$ in the data matrix at time $t$. $M_t=t-1$ for OGD, growing unbounded.

%
%%%%%%%%%%%%%%%%%%%%%%%%%%%%%%%%%%%%%%%%%%%%%%%%%%%%%%%%%%%%%%%
%%% A  L  G  O  R  I  T  H  M  %%%%%%%%%%%%%%%%%
%%%%%%%%%%%%%%%%%%%%%%%%%%%%%%%%%%%%%%%%%%%%%%%%%%%%%%%%%%%%%%%
\begin{algorithm}[t]
\caption{Dynamic Parsimonious Online Learning with Kernels (DynaPOLK)}
\begin{algorithmic}
\label{alg:soldd}
\REQUIRE $\{\bbx_t,\eta,\epsilon \}_{t=0,1,2,...}$
\STATE \textbf{initialize} ${f}_0(\cdot) = 0, \bbD_0 = [], \bbw_0 = []$, i.e. initial dictionary, coefficient vectors are empty
\FOR{$t=0,1,2,\ldots$}
	\STATE Obtain independent data realization $(\bbx_t)$ and loss $\ell_t(\cdot)$
	\STATE Compute unconstrained functional online gradient step \vspace{-0mm}% [cf. \eqref{eq:sgd_tilde}]
	{$$\tilde{f}_{t+1}(\cdot) = (1-\eta \lambda){f}_t - \eta\nabla_f\check{L}_t({f}_t(\bbS_t))$$} \vspace{-4mm}
	\STATE Revise dict. $\tbD_{t+1} = [\bbD_t,\;\;\bbx_t]$, weights $\bbw_{t+1}$ via \eqref{eq:param_tilde}-\eqref{eq:param_tilde_horizon}\vspace{-0mm}
	\STATE Compress function via KOMP \cite{Vincent2002} with budget $\epsilon$\vspace{-2mm}
	$$({f}_{t+1},\bbD_{t+1},\bbw_{t+1}) = \textbf{KOMP}(\tilde{f}_{t+1},\tbD_{t+1},\tbw_{t+1},\epsilon)$$
\ENDFOR
\end{algorithmic}
\end{algorithm}

%%%%%%%%%%%%%%%%%%%%%%%%%%%%%%%%%%%%%%%%%%%%%%%%%%%%%%%%%%%%%%%
%%%   S  U  B  S  E  C  T  I  O  N    %%%%%%%%%%%%%%%%%
%%%%%%%%%%%%%%%%%%%%%%%%%%%%%%%%%%%%%%%%%%%%%%%%%%%%%%%%%%%%%%%
%\subsection{Model Order Control via Subspace Projection}\label{subsec:proj}
%
{\bf Model Order Control via Subspace Projection}
To overcome the aforementioned bottleneck, we propose projecting the OGD sequence \eqref{eq:sgd_hilbert} onto subspaces $\ccalH_\bbD \subseteq \ccalH$ defined by some dictionary $\bbD = [\bbd_1,\ \ldots,\ \bbd_M] \in \reals^{p \times M}$, i.e., $\ccalH_\bbD = \{f\ :\ f(\cdot) = \sum_{t=1}^M w_t\kappa(\bbd_t,\cdot) = \bbw^T\boldsymbol{\kappa}_{\bbD}(\cdot) \}=\text{span}\{\kappa(\bbd_t, \cdot) \}_{t=1}^M$, inspired by \cite{koppel2019parsimonious}. For convenience we have defined $[\boldsymbol{\kappa}_{\bbD}(\cdot)=\kappa(\bbd_1,\cdot) \ldots \kappa(\bbd_M,\cdot)]$, and $\bbK_{\bbD,\bbD}$ as the resulting kernel matrix from this dictionary. We  ensure parsimony by ensuring $M_t \ll t$.
%
%We first show that, by selecting $\bbD = \bbX_{t+1}$ at each iteration, the sequence \eqref{eq:sgd_hilbert} derived in Section \ref{subsec:sgd} may be interpreted as carrying out a sequence of orthogonal projections.  To see this, rewrite \eqref{eq:sgd_hilbert} as the quadratic minimization
%%
%{\begin{align}\label{eq:proximal_hilbert_dictionary}
%f_{t+1} &= \argmin_{f \in \ccalH}
%\Big\lVert f - \Big((1-\eta H\lambda) f_t 
%- \eta \nabla_f\ell_t(f_{t}(\bbx_t)) \Big) \Big\rVert_{\ccalH}^2 \nonumber \\
%%
%&= \argmin_{f \in \ccalH_{\bbX_{t+1}}} 
% \Big\lVert f - \Big( (1-\eta H\lambda)f_t - 
% \eta \nabla_f\ell_t(f_{t}(\bbx_t)) \Big) \Big\rVert_{\ccalH}^2, 
%\end{align}}
%%
%where the first equality in \eqref{eq:proximal_hilbert_dictionary} comes from ignoring constant terms which vanish upon differentiation with respect to $f$, and the second comes from observing that $f_{t+1}$ can be represented using only the points $\bbX_{t+1}$, using \eqref{eq:param_update}.  Notice now that \eqref{eq:proximal_hilbert_dictionary} expresses $f_{t+1}$ as the orthogonal projection of the update {$(1-\eta H\lambda) f_t - \eta\nabla_f\ell_t(f_{t}(\bbx_t))$} onto the subspace defined by dictionary $\bbX_{t+1}$.

Rather than allowing model order of $f$ to grow in perpetuity [cf. \eqref{eq:param_update}], we project $f$ onto subspaces defined by dictionaries $\bbD=\bbD_{t+1}$ extracted from past data.  Deferring the selection of $\bbD_{t+1}$ for now, we note it has dimension $p \times {M}_{t+1}$, with ${M}_{t+1} \ll t$.  Begin by considering function ${f}_{t+1}$ is parameterized by dictionary $\bbD_{t+1}$ and weight vector $\bbw_{t+1}$. Moreover, we denote columns of $\bbD_{t+1}$ as $\bbd_t$ for $t=1,\dots,{M}_{t+1}$. We propose a projected variant of OGD:%where the time index is dropped for notational clarity but may be inferred from the context.
%
%To be specific, we propose replacing the update \eqref{eq:proximal_hilbert_dictionary} in which the dictionary grows at each iteration by the projection of the functional online gradient sequence onto the subspace $\ccalH_{\bbD_{t+1}}=\text{span}\{ \kappa(\bbd_t, \cdot) \}_{t=1}^{M_{t+1}}$ as
\begin{align}\label{eq:projection_hat}
\!\!\!{f}_{t+1}\! =&\! \argmin_{f \in \ccalH_{\bbD_{t+1}}}  \!\Big\lVert f\! - \!
\Big(\!(\!1-\eta \lambda) f_t 
\!-\! \eta \nabla_f \check{L}_t(f_{t}(\bbS_t)) \!\Big)\Big\rVert_{\ccalH}^2 \nonumber \\
&:=\ccalP_{\ccalH_{\bbD_{t+1}}}\! \!\Big[ \!
(1\!-\!\eta \lambda) f_t 
- \eta \nabla_f\check{L}_t(f_{t}(\bbS_t)) \Big] \! 
\end{align}
where we define the projection operator $\ccalP$ onto subspace $\ccalH_{\bbD_{t+1}}\subset \ccalH$ by the update \eqref{eq:projection_hat}.  %%%%%%%%%%%%%%%%%%%%%%%%%%%%%%%%%%%%%%%%%%%%%%%%%%%%%%%%%%%%%%%
%%%   S  U  B  S  E  C  T  I  O  N    %%%%%%%%%%%%%%%%%
%%%%%%%%%%%%%%%%%%%%%%%%%%%%%%%%%%%%%%%%%%%%%%%%%%%%%%%%%%%%%%%

{\bf Coefficient update} The update \eqref{eq:projection_hat}, for a fixed dictionary $\bbD_{t+1} \in \reals^{p\times M_{t+1}}$, implies an update only on coefficients. To illustrate this point, define the online gradient update without projection, given function ${f}_t$ parameterized by dictionary $\bbD_t$ and coefficients $\bbw_t$, as
%
%\begin{align}\label{eq:sgd_tilde}
$\tilde{f}_{t+1} =(1 - \eta H\lambda ) {f}_t - \eta\nabla_f\check{L}_t({f}_t(\bbS_t)).$
%\end{align}
%
This update may be represented using dictionary and weight vector as %\blue{\eqref{eq:param_tilde}-\eqref{eq:param_tilde_horizon} needed to change to reflect the windowing over $P$ points. I have done that. I think the sum over past losses was incorrect, as each individual loss component only contributes to the $\kappa(\bbx_\tau,\cdot)$ term. Please check this.}
\begin{align}\label{eq:param_tilde}
\tbD_{t+1} = [\bbD_t,\;\;\bbx_t], \;\;\;\; w_{t+1} = -\eta\check{\ell}_t'({f}_t(\bbx_t))\; .
\end{align}
and revising last $H-1$ weights with $\tau=t-H+1$ to $t-1$, yielding the update for coefficients as  
\begin{align}\label{eq:param_tilde_horizon}
w_{\tau}=
\begin{cases}
\!(1\! - \!\eta \lambda) w_\tau\!-\!\eta\check{\ell}_\tau'(f(\bbx_\tau)) \ \  \text{ for } \tau\!=\!t\!-\!H\!+\!1,\dots,t\!-\!1 
\\
(1 - \eta \lambda) w_\tau \qquad \qquad \qquad \text{ for } \tau < t-H+1.
\end{cases}
\end{align}
%
%Denote the columns of $\tbD_{t+1}$ by $\tbd_{n}$ for $t=1,\dots,\hat{M}_t+1$, where again we have dropped the time index for notational clarity but note that it can be easily inferred from the context. However, the $\hat{f}_{t+1}$ we compute will be distinct from $\tilde{f}_{t+1}$ in that its model order may be smaller, i.e., $\hat{M}_{t+1} \leq \hat{M}_t +1$.  That is, we use different process to select a $\hbD_{t+1}$, and we then compute the orthogonal projection of $\tilde{f}_{t+1}$ onto $\ccalH_{\hbD_{t+1}}$, i.e.,
%
%Observe that $\tbD_{t+1}$ has $\tilde{M}=M_t + 1$ columns, which is also the length of $\tbw_{t+1}$. We note that the minimum dictionary size is always at least $P$ in order to evaluate the online gradient.  
%
For fixed dictionary $\bbD_{t+1}$, the projection \eqref{eq:projection_hat} is a least-squares problem on coefficients  $\bbw_{t+1}$ \cite{williams2001using}: % To see this, make use of the Representer Theorem to rewrite \eqref{eq:projection_hat} in terms of kernel expansions, and that the coefficient vector is the only free parameter to write
\begin{align} \label{eq:hatparam_update}
\bbw_{t+1}=  \bbK_{\bbD_{t+1} \bbD_{t+1}}^{-1} \bbK_{\bbD_{t+1} \tbD_{t+1}} \tbw_{t+1} \;.
\end{align}
Given that projection of $\tilde{f}_{t+1}$ onto subspace $\ccalH_{\bbD_{t+1}}$ for a fixed dictionary $\bbD_{t+1}$ is a simple least-squares multiplication, we turn to explaining the selection of the kernel dictionary $\bbD_{t+1}$ from past data $\{\bbx_u\}_{u \leq t}$.

{\bf Dictionary Update} One way to obtain the dictionary $\bbD_{t+1}$ from $\tbD_{t+1}$, as well as the coefficient $\bbw_{t+1}$, is to apply a destructive variant of \emph{kernel orthogonal matching pursuit} (KOMP) with pre-fitting \cite{Vincent2002}[Sec. 2.3]  as in \cite{koppel2019parsimonious}. KOMP operates by beginning with full dictionary $\tbD_{t+1}$ and sequentially removing columns while the condition $\|\tilde{f}_{t+1} - f_{t+1}\|_{\ccalH}\leq \epsilon$ holds. The projected FOGD is defined as:
\begin{equation}\label{equ:KOMP}
({f}_{t+1},\bbD_{t+1},\bbw_{t+1}) = \textbf{KOMP}(\tilde{f}_{t+1},\tbD_{t+1},\tbw_{t+1},\epsilon),
\end{equation}
where $\epsilon$ is the compression budget which dictates how many model points are thrown away during model order reduction. By design,  we have $\|{f}_{t+1}-\tilde{f}_{t+1}\|_{\ccalH}\leq \epsilon$, which allows us tune $\epsilon$ to only keep dictionary elements critical for online descent directions. These details allow one to implement Dynamic Parsimonious Online Learning with Kernels (DynaPOLK) (Algorithm \ref{alg:soldd}) efficiently.  Subsequently, we discuss its theoretical and experimental performance.% of DynaPOLK in terms of regret and complexity.

%!TEX root = Paper.tex
%%%%%%%%%%%%%%%%%%%%%%%%%%%%%%%%%%%%%%%%%%%%%%%%%%%%%%%%%%%%%%%
%%%   S  E  C  T  I  O  N  %%%%%%%%%%%%%%%%%
%%%%%%%%%%%%%%%%%%%%%%%%%%%%%%%%%%%%%%%%%%%%%%%%%%%%%%%%%%%%%%%
\section{Balancing Regret and Model Parsimony}\label{sec:convergence}
In this section, we establish the sublinear growth of dynamic regret of Algorithm \ref{alg:soldd} up to factors depending on \eqref{vtdy} and the compression budget parameter that parameterizes the algorithm. To do so, some conditions on the loss, its gradient, and the data domain are required which we subsequently state.
%%%%%%%%%%%%%%%%%%%%%%%%%%%%%%%%%%%%%%%%%%%%%%%%%%%%%%%%%%%%%%%
%%%   A   S   S   U   M   P   T   I   O   N   %%%%%%%%%%%%%%%%%
%%%%%%%%%%%%%%%%%%%%%%%%%%%%%%%%%%%%%%%%%%%%%%%%%%%%%%%%%%%%%%%
%compact data domain
\begin{assumption}\label{as:first}
The feature space $\ccalX\subset\reals^p$ is compact, and the reproducing kernel is bounded:
\begin{align}\label{eq:bounded_kernel}
\sup_{\bbx\in\ccalX} \sqrt{\kappa(\bbx, \bbx )} = X < \infty.
\end{align}
\end{assumption}
%
%loss function boundedness
\begin{assumption}\label{as:2}
{The loss $\check{\ell}_t: \ccalH \times \ccalX  \rightarrow \reals$ is uniformly $C$-Lipschitz continuous for all $z \in \reals $: }
\begin{align}\label{eq:lipschitz}
| \check{\ell}_t(z) - \check{\ell}_t( z') | \leq C |z - z'|.
\end{align}
\end{assumption}
%
%loss convexity and differentiability
%
\begin{assumption}\label{as:3}
The loss $\check{\ell}_t(f(\bbx))$ is convex and differentiable w.r.t. $f(\bbx)$ on $\reals$ for all $\bbx\in\ccalX$.  
\end{assumption}
\begin{assumption}\label{as:4}
The gradient of the loss $\nabla {\ell}_t(f(\bbx))$ is Lipschitz continuous with parameter $\tilde L>0$:
\begin{align}\label{gradient_assumption}
\|{\nabla}_f{\ell}_t(f(\bbS_t))-{\nabla}_g{\ell}_t(g(\bbS_t)) \|_{\ccalH}\leq \tilde L \|f-g\|_{\ccalH} 
\end{align} 
for all $t$ and $f,g\in\mathcal{H}$. 
\end{assumption}
%%%%%%%%%%%%%%%%%%%%%%%%%%%%%%%%%%%%%
Assumption \ref{as:first} and Assumption \ref{as:3} are standard \cite{Kivinen2004,dai2014scalable}. Assumptions \ref{as:2} and \ref{as:4} ensures the instantaneous loss $\check{\ell}_t(\cdot)$ and its derivative are smooth, which is usual for gradient-based optimization \cite{Bertsekas1999}, and holds, for instance, for the square, squared-hinge, or logistic losses. Because we are operating under the windowing framework over last $P$ losses \eqref{eq:loss}, we define the Lipschitz constant of  $L_t(\cdot)$ as $CP$ and that of its gradient as $L=H \tilde{L}$. Doing so is valid, as the sum of Lipschitz functions is Lipschitz \cite{rudin1964principles}.
%
%. We remark that the assumption in \eqref{gradient_assumption} can be derived by considering the Lipschitz continuity of the gradient $\check{\ell}'_t(\cdot)$

%\red{Before heading to discuss the regret analysis of the proposed algorithm, we first describe the effect of comrpession technique proposed in Sec. \ref{subsec:proj}. Recall that the number of elements in the dictionary is called model order. Firstly, we try to characterize the model order growth with respect to iteration index $t$. We proposed a compression technique in Sec. \ref{subsec:proj} to control the model order from increasing linearly for Algorithm \ref{alg:soldd}. We present the main result in Lemma \ref{theorem_model_order} with the proof provided in Appendix \ref{proof_thm_4}.
Before analyzing the regret of Alg. \ref{alg:soldd}, we discern the influence of the learning rate, compression budget, and problem parameters on the model complexity of the function. In particular, we provide a minimax characterization of the number of points in the kernel dictionary in the following lemma, which determines the required complexity for sublinear dynamic regret growth in different contexts.
%                
%%%%%%%%%%%%%%%%%%%%%%%%%%%%%%%%%%%%%%
%%%%%%%%%%%%%%%%%%%%%%%%%%%%%%%%%%%%%
%%%%%%%%%   T  H  E  O  R  E  M     %%%%%%%%%%%%%%%
%%%%%%%%%%%%%%%%%%%%%%%%%%%%%%%%%%%%%
%%%%%%%%%%%%%%%%%%%%%%%%%%%%%%%%%%%%%
\begin{lemma}\label{theorem_model_order}
Let $f_t$ be the function sequence of Algorithm \ref{alg:soldd} with step-size $\eta<\min\{1/\lambda,1/L\}$ and compression $\eps$. Denote $M_t$ as the model order (no. of columns in dictionary $\bbD_t$) of $f_t$. For a Lipschitz Mercer kernel $\kappa$ on compact set $\mathcal{X}\subseteq\mathbb{R}^p$, there exists a constant $Y$ s.t. for data $\{\bbx_t\}_{t=1}^\infty$, $M_t$ satisfies
 \begin{align}\label{model_order_0}
   H\leq M_t\leq Y({CH})^p\left(\frac{\eta}{\epsilon}\right)^p.
   \end{align}
\end{lemma}
Lemma \ref{theorem_model_order} (proof in Appendix \ref{proof_thm_4}) establishes that the model order of the learned function is lower bounded by the time-horizon $H$ and its upper bound depends on the ratio of the step-size to the compression budget, as well as the Lipschitz constant [cf. \eqref{eq:lipschitz}].
 Next, we shift to characterizing the dynamic regret of Algorithm \ref{alg:soldd}. 
 %\blue{Please check that I got the idea/definition of $\triangle_T$ correct. I read through Besbes's paper and this is what made sense to me...}
%
%%%%%%%%%%%%%%%%%%%%%%%%%%%%%%%%%%%%%%
%%%%%%%%%%%%%%%%%%%%%%%%%%%%%%%%%%%%%
%%%%%%%%%   T  H  E  O  R  E  M     %%%%%%%%%%%%%%%
%%%%%%%%%%%%%%%%%%%%%%%%%%%%%%%%%%%%%
%%%%%%%%%%%%%%%%%%%%%%%%%%%%%%%%%%%%%
%\subsection{Dynamic Regret Analysis in Terms of the Loss Function Variation}
Our first result establishes that the dynamic regret, under appropriate step-size and compression budget selection, grows sublinearly up to a factor that depends on a batch parameter and the cost function variation \eqref{variation}, and that the model complexity also remains moderate. This result extends \cite{besbes}[Proposition 2] to nonparametric settings. 

\begin{theorem}\label{theorem:dynamic_cost} 
Denote as $\{f_t\}$ the sequence generated by Algorithm \ref{alg:soldd} run for $T$ total iterations partitioned into $m=\ceil{\frac{T}{\triangle_T}}$ mini-horizons of length $\triangle_T$. Over mini-horizons, Algorithm \ref{alg:soldd} is run for $\triangle_T$ steps. Under Assumptions \ref{as:first}-\ref{as:4}, the {dynamic} regret \eqref{dynamic} grows with horizon $T$ and loss variation \eqref{variation} as:
\begin{align}\label{dynamic_regret_vT}
\textbf{Reg}^D_T = \ceil{\frac{T}{\triangle_T}}{\mathcal{O}\left(\frac{1+(\epsilon+\eta^2) \triangle_T}{\eta}\right)}+2\triangle_TV_T \; ,
\end{align}
which is sublinear for %$\eta=\mathcal{O}(T^{-a})$ and $ \eps=\mathcal{O}(T^{-b})$ such that $p(a-b)\in\(0,1)$
$\eta=\mathcal{O}(\triangle_T^{-a})$ and $\eps=\mathcal{O}(\triangle_T^{-b})$ with mini-horizon $\triangle_T=o(T)$, provided $p(a-b)\in(0,1)$. That is, with $\eta=\triangle_T^{-1/2}$ and $\eps=\triangle_T^{-(p-1)/2p}$, \eqref{dynamic_regret_vT} grows sublinearly in $T$ and $V_T$.
\end{theorem}
%
%
%Consult Appendix \ref{proof_thm_2} for proof. 
\begin{proof}
 %
 %\begin{proof}
 %
 %
 Consider the expression for the dynamic regret is given by
 \begin{align}\label{eq:dynamic_regret}
 \textbf{Reg}^D_T=&\sum_{t=1}^{T}L_{t}(f_{t}(\bbS_t))-\sum_{t=1}^{T}L_{t}(f^\star_t(\bbS_t)).
 \end{align}
 Add subtract the term $\sum_{t=1}^{T}\ell_{t}(f^\star(\bbx_t))$ to the right hand side of \eqref{eq:dynamic_regret}, we obtain
 \begin{align}\label{eq:dynamic_regret2}
 \textbf{Reg}^D_T=&\sum_{t=1}^{T}L_{t}(f_{t}(\bbS_t))-\sum_{t=1}^{T}L_{t}(f^\star(\bbx_t))\nonumber \\
 &+\sum_{t=1}^{T}L_{t}(f^\star(\bbx_t))-\sum_{t=1}^{T}L_{t}(f^\star_t(\bbS_t))\nonumber
 \\
& \quad =\textbf{Reg}^S_T+\sum_{t=1}^{T}[L_{t}(f^\star(\bbx_t))-L_{t}(f^\star_t(\bbS_t))].
 \end{align}
 We have utilized the definition of static regret in \eqref{static} to obtain \eqref{eq:dynamic_regret2}. Note that the behavior in terms of static regret of Algorithm \ref{alg:soldd} is characterized in Theorem \ref{theorem:static}. To analyze the dynamic regret in terms of $V_T$, we need to study the different between the static optimal and dynamic optimal given by the second term on the right hand side of \eqref{eq:dynamic_regret2}. The difference between the two benchmarks (static and dynamic) is determined by the size of $T$ and fundamental quantifiers of non-stationarity defined in Section \ref{sec:problem} . To connect \eqref{eq:dynamic_regret2} with the loss function variation, following \cite{besbes}, we split the interval $T$ into equal size $m$ batches with each of size $\triangle_T$ except the last batch given by $\ccalT_j = \{t\ ;\ (j-1)\triangle_T +1\leq t\leq \min\{j \triangle_T, T\} $  for $j=1,\dots,T$ where  $m=\ceil{\frac{T}{\triangle_T}}$. We can rewrite the expression in \eqref{eq:dynamic_regret2} as follows
 \begin{align}\label{eq:dynamic_regret_3}
 \textbf{Reg}^D_T
 =&\sum_{s=1}^{\ceil{\frac{T}{\triangle_T}}}\sum_{t\in\mathcal{T}_s}[L_{t}(f_{t}(\bbS_t))-L_{t}(f^\star_s(\bbx_t))]\nonumber
 \\
 &+\sum_{s=1}^{\ceil{\frac{T}{\triangle_T}}}\sum_{t\in\mathcal{T}_s}[L_{t}(f^\star_s(\bbx_t))-L_{t}(f^\star_t(\bbS_t))]
 \end{align}
 where we define $f^\star_s= \argmin_{f \in \ccalH}\sum\limits_{t\in\mathcal{T}_s} L_t(f(\bbS_t))$ for all $s=1, 2, \cdots, m$, and note that the outer sum over $s$ indexes the batch number, whereas inner one indexes elements of a particular batch $\ccalT_s$. The expression for the dynamic regret in \eqref{eq:dynamic_regret_3} is decomposed into two sums. Note that the first sum represents the sum of the regrets against a single batch action for each batch $\ccalT_s$. 
 The second term in \eqref{eq:dynamic_regret_3} quantifies the non-stationarity of the optimizer: it is a sum over differences between the best action over batch $s$ and corresponding dynamic optimal actions.  Next, we bound the each term on the right hand side of \eqref{eq:dynamic_regret_3} separately.  From the static regret in \eqref{static_Regret}, it holds that
 \begin{align}\label{eq:dynamic_regret_4}
 \sum_{t\in\mathcal{T}_s}\![L_{t}(f_{t}(\bbS_t))\!-\!L_{t}(f^\star_s(\bbx_t))]
 %
% \!=\!\mathcal{O}\left(\frac{1+\epsilon^2 \triangle_T}{\eta}\right)
\!=\! \mathcal{O}\!\left(\frac{1\!+\!( \epsilon \!+\!\epsilon^2) \triangle_T}{\eta}\!+\!\eta \triangle_T\right)
 \end{align}
  for all $s=1, 2, \cdots, m$. To upper bound the term in \eqref{eq:dynamic_regret_3} associated with non-stationarity, i.e., the second term on the right-hand side, by definition of the minimum, we have
 \begin{align}\label{eq:cost_variation_bound}
 \sum_{t\in\mathcal{T}_s}\![L_{t}(f^\star_s(\bbx_t))\!-\!L_{t}(f^\star_t(\bbS_t))]\!\leq\!& \sum_{t\in\mathcal{T}_s}\![L_{t}(f_{k}^\star(\bbx_t))\!-\!L_{t}(f^\star_t(\bbS_t))]
 \end{align}
 where $k$ denotes the first epoch of batch $\mathcal{T}_s$ and the inequality in \eqref{eq:cost_variation_bound} holds from the optimality of $f^\star(\bbS_t)$. Further taking maximum over batch, we obtain the upper bound for \eqref{eq:cost_variation_bound} as 
 \begin{align}\label{eq:cost_variation_bound2}
 %
 %\sum_{t\in\mathcal{T}_s}&[L_{t}(f^\star(\bbS_t))-L_{t}(f^\star_t(\bbS_t))]\non %umber\\
% \leq&
  \triangle_T \max_{t\in\mathcal{T}_s}[L_{t}(f_{k}^\star(\bbS_t))-L_{t}(f^\star_t(\bbS_t))].
 \end{align}
 Next, we need to upper bound the right hand side of \eqref{eq:cost_variation_bound2} in terms if the loss function variation budget $V_T$. To do that, let us first define the loss function variation over each batch $\mathcal{T}_s$ as follows
 \begin{align}
 V_s:=\sum_{t\in\mathcal{T}_s}|L_{t}-L_{t-1}|
 \end{align}
 and note that $V_T=\sum\limits_{s=1}^{m}V_s$. With this definition, we now  show that   
 \begin{align}\label{condition}
 \max_{t\in\mathcal{T}_s}[L_{t}(f_{k}^\star(\bbS_t))-L_{t}(f^\star_t(\bbS_t))]\leq 2V_s
 \end{align}
  by contradiction. Let us assume that the inequality in \eqref{condition} in not  true  which means that there is at least one epoch, say $m\in\mathcal{T}_s$, for which the following property is valid:
  \begin{align}\label{condition2}
  L_{m}(f_{k}^\star(\bbS_m))-L_{m}(f^\star_m (\bbS_m))>2V_s.
  \end{align}
  Since $V_j$ is the maximal variation for batch $\mathcal{T}_s$, it holds that
  \begin{align}\label{condition3}
  L_{t}(f_{m}^\star(\bbS_t))\leq L_{m}(f_{m}^\star(\bbS_m))+V_s.
  \end{align}
  Substituting the upper bound for $L_{m}(f_{m}^\star(\bbS_m))$ from \eqref{condition2} into \eqref{condition3}, we get
  \begin{align}\label{condition4}
  L_{t}(f_{m}^\star(\bbS_t))<& L_{m}(f_{k}^\star(\bbS_m))-V_s\nonumber\\
  \leq& L_{m}(f_{k}^\star(\bbS_m)).
  \end{align}
  for all $t\in\mathcal{T}_s$. The second inequality in \eqref{condition4} holds by dropping the negative terms. We note that the inequality in \eqref{condition4} is a contradiction for $t=m$, since a positive number cannot be less than itself. Therefore, the hypothesis in \eqref{condition2}  is invalid, which implies that  \eqref{condition} holds true. Next, we utilize the upper bound in \eqref{condition} to the right hand side of \eqref{eq:cost_variation_bound2}, we get
 \begin{align}\label{eq:cost_variation_bound22}
 \sum_{t\in\mathcal{T}_s}[L_{t}(f^\star(\bbS_t))-L_{t}(f^\star_t(\bbS_t))]\leq& 2\triangle_T V_s.
 \end{align}
 Now, we return to the aggregation of static regret and the drift of the costs over time in \eqref{eq:dynamic_regret2}, applying \eqref{eq:dynamic_regret_4} and \eqref{eq:cost_variation_bound22} into \eqref{eq:dynamic_regret_3} to obtain final expression for the dynamic regret as
 \begin{align}\label{dynamic_regret_vT22}
 \textbf{Reg}^D_T
 \leq& \ceil{\frac{T}{\triangle_T}} \mathcal{O}\left(\!\!\frac{1\!+\!( \epsilon \!+\!\epsilon^2) \triangle_T}{\eta}\!+\!\eta \triangle_T\right)\!\!+\!\!2\triangle_TV_T.
 \end{align}
 Suppose we make the parameter selections 
 \begin{align}\label{selection1}
  &\eta=\mathcal{O}(\triangle_T^{-a}) \ \ \ \text{and}\ \ \  \eps=\mathcal{O}(\triangle_T^{-b})
 \end{align}
 with $\triangle_T<\ccalO(T)$. Then the right-hand side of \eqref{dynamic_regret_vT22} takes the form
  \begin{align}\label{dynamic_regret_vT22}
 \textbf{Reg}^D_T
 \leq%& \ceil{\frac{T}{\triangle_T}} \mathcal{O}\left(\frac{1+( \triangle_T^{-b} +\triangle_T^{-2b}) \triangle_T}{\triangle_T^{-a}}+\triangle_T^{-a} \triangle_T\right)+2\triangle_TV_T. \nonumber \\
&  \ceil{\frac{T}{\triangle_T}} \mathcal{O}\left(\triangle_T^{a}+( \triangle_T^{-b} +\triangle_T^{-2b}) \triangle_T^{(1+a)}+ \triangle_T^{1-a}\right)\nonumber
\\
&+2\triangle_TV_T.
 \end{align}
 %
%$ \mathcal{O}\left(T^b+T^{(1-(a-b))}+T^{1-b}\right)$
with model order $M=  \mathcal{O}(\triangle_T^{p(a-b)})$ by substituting \eqref{selection1} into the result of Lemma \ref{theorem_model_order}.
 For the dynamic regret to be sublinear, we need $b\in(0,1)$ and $a\in(b,b+\frac{1}{p})$. As long as the dimension  $p$ is not too large, we always have a range for $a$. This implies that $p(a-b)\in(0,1)$ and hence $M$ is sublinear. One specification of that satisfies this range is $a=1/2$ and $b=(p-1)/2p$, as stated in Theorem \ref{theorem:dynamic_cost}. We obtain the result presented in Table \ref{tab1} for the selection $\eta=\mathcal{O}{(1/\sqrt{\triangle_T})}$ and $\triangle_T=\left(T/V_T\right)^{2/3}$. 
 %\hfill $\qed$
\end{proof}
The batch parameter $\triangle_T$ tunes static versus non-stationary performance: for large $\triangle_T$, then the algorithm attains smaller regret with respect to the static oracle, i.e., the first terms on the right-hand side of \eqref{dynamic_regret_vT}, but worse in terms of the non-stationarity as quantified by function variation $V_T$, the last term. On the other hand, if the batch size is smaller, we do worse in terms of static regret terms but better in terms of non-stationarity. This contrasts with the parametric setting as well \cite{besbes}: the $\mathcal{O}(\epsilon)$ term appears due to the compression-induced error. %\blue{we need to be a bit more specific about how the batching works in practice. How would choice of  $\triangle T$ actually affect how Algorithm \ref{alg:soldd} works? 
%
%Also missing here is an explicit characterization of how to select the step-size for sublinear growth... Rather than breaking up the discussion of how to select the step-size into a separate subsection after all the results, we should present those with each theorem, and than collate them at the end. \red{Need to read besbas paper and then do it.}}

%\subsection{Dynamic Regret in Terms of Path Length}
%
Up to now, we quantified algorithm performance by loss variation \eqref{variation}; however, this is only a surrogate for the performance of the sequence of \emph{time-varying} optimizers \eqref{vtdy}, which is fundamental in time-varying optimization \cite{simonetto2016class,simonetto2017time}, and may be traced to functions of bounded variation in real analysis \cite{rudin1964principles}. Thus, we shift focus to analyzing Algorithm \ref{alg:soldd} in terms of this fundamental performance metric.

First, we note that the path length \eqref{vtdy} is unique when losses are strongly convex. On the other hand, when costs are non-strongly convex, then \eqref{vtdy} defines a set of optimizers. Thus, these cases must be treated separately. First, we introduce an assumption used in the second part of the upcoming theorem.
%
%To do so, an additional condition regarding the strong convexity of the instantaneous losses is required, stated next.
%Now we present the assumption of strong convexity which is specific to this subsection only and not used anywhere else in the paper. 
%%%%%%%%%%%%%%%%%%%%%%%%%%%%%%%%%%%%%%%
%%%%%%%%%%%%%%%%%%%%%%%%%%%%%%%%%%%%%
%%%%%%%%%  A  S  S  U  M  P  T  I  O  N    %%%%%%%%%%%%%%%
%%%%%%%%%%%%%%%%%%%%%%%%%%%%%%%%%%%%%
%%%%%%%%%%%%%%%%%%%%%%%%%%%%%%%%%%%%%
\begin{assumption}\label{as:5}
The instantaneous loss $L_t: \ccalH \times \ccalX  \rightarrow \reals$ is strongly convex with parameter $\mu$:
\begin{align}\label{eq:strong_convexity}
 L_t(f) - L_t(\tilde f)  \geq \mu \| f - \tilde f \|_{\ccalH}^2  
\end{align}
 for all $t$ and any functions $f,\tilde f \in \ccalH$.
\end{assumption}

With the technical setting clarified, we may now present the main theorem regarding dynamic regret in terms of path length \eqref{vtdy}.
%%%%%%%%%%%%%%%%%%%%%%%%%%%%%%%%%%%%%%
%%%%%%%%%%%%%%%%%%%%%%%%%%%%%%%%%%%%%
%%%%%%%%%   T  H  E  O  R  E  M     %%%%%%%%%%%%%%%
%%%%%%%%%%%%%%%%%%%%%%%%%%%%%%%%%%%%%
%%%%%%%%%%%%%%%%%%%%%%%%%%%%%%%%%%%%%
\begin{theorem}\label{theorem:dynamic_path}
Denote $\{f_t\}$ as the function sequence generated by Algorithm \ref{alg:soldd} run for $T$ iterations. Under Assumptions \ref{as:first}-\ref{as:4}, with regularization $\lambda>0$ the following dynamic regret bounds hold in terms of path length \eqref{vtdy} and compression budget $\epsilon$:
\begin{enumerate}[label=(\roman*)]
\item \label{theorem:dynamic_path_cvx} when costs $\ell_t$ are convex, regret is sublinear {with $\eta<\min\{\frac{1}{\lambda},\frac{1}{L}\}$} and for any {$\epsilon=\mathcal{O}\left(T^{-\alpha}\right)$ with $\alpha\in(0,\frac{1}{p}]$}, we have
\begin{align}\label{eq:dynamic_path_thm}
\textbf{Reg}^D_T &= \ccalO\left( \frac{1 + {T\sqrt{\epsilon} } + W_T }{\eta}  \right)\nonumber
\\
&= \ccalO\left( {1 + {T\sqrt{\epsilon} } + W_T }  \right).    	   	 	   	 	 	 
                \end{align}
%
%
%which is sublinear in $T$ up to factors depending on path length $W_T$ for $\eta=\ccalO(T^{-1/2})$ and $\epsilon\in[0,T^{-1/2})$.     
%
\item \label{theorem:dynamic_path_strong_cvx} Alternatively, if the cost functions $\ell_t$ are strongly convex, i.e., Assumption \ref{as:5} holds, {with $\eta<\min\{\frac{1}{\lambda},\frac{\mu}{L^2}\}$} and for any {$\epsilon=\mathcal{O}\left(T^{-\alpha}\right)$ with $\alpha\in(0,\frac{1}{p}]$}, we have %then we have the following dynamic regret bound
\begin{align}\label{eq:dynamic_path_thm_strongcvx}
\textbf{Reg}^D_T =& \ccalO\left(\frac{1 + {T\sqrt{\epsilon} } + W_T}{1-\rho} \right)\nonumber\\
=& {o\left({1 + {T\sqrt{\epsilon} } + W_T}\right)},
\end{align}
where $\rho:=\sqrt{(1-2\eta(\mu-\eta L^2))}\in(0,1)$ is a contraction constant for a given $\eta$.
%
%and \eqref{eq:dynamic_path_thm_strongcvx} is sublinear in $T$ up to terms depending on the path length for any step-size $\eta$ such that $1-\rho = T^{-1/2}$, for compression constant $\epsilon\in[0,T^{-1/2})$. This implies that the step size $\eta$ satisfies $\eta\geq \mathcal{O}(T^{-1/2})$. 
\end{enumerate}
\end{theorem}
%
%%%%%%%%%%%%%%%%%%%%%%%%%%%%%%%%%%%%%%%%%%%%%%%%%%%%%%%%%%%%%%%
%%%  P  R  O  O F  %%%%%%%%%%%%%%%%%
%%%%%%%%%%%%%%%%%%%%%%%%%%%%%%%%%%%%%%%%%%%%%%%%%%%%%%%%%%%%%%%
%The proof of Theorem \ref{theorem:dynamic_path} is given in Appendix \ref{proof_thm_3}.
%\begin{proof}
 
\hspace{-5mm}   \emph{ Proof of Theorem \ref{theorem:dynamic_path}\ref{theorem:dynamic_path_cvx}}
% \begin{proof}
 %
 Begin by noting that the descent relation in Lemma \ref{lemma1} also holds for time-varying optimizers $f_t^\star$, which allows us to write
 \begin{align}\label{eq:iterate_prop1_subst0}
 \| f_{t+1}\!\! -\!\! f^\star_t \|_{\ccalH}^2 
 	 \leq&\| f_t - f^\star_t \|_{\ccalH}^2
 	 - 2 \eta [L_t(f_t(\bbS_t)) -L_t(f_t^\star(\bbS_t))] \nonumber\\
 	 &\!\!+ 2 \eps \| f_t\! - \!f^\star_t \|_{\ccalH} 
 	\!+\! \eta^2 \| \tilde{\nabla}_fL_{t}(f_{t}(\bbS_t)) \|_{\ccalH}^2. 
 \end{align}
 From the inequality in \eqref{proof_Strong_111}, we have 
 \begin{align}\label{proof_Strong_1112}
 \| \tilde{\nabla}_f L_t(f_{t}(\bbS_t)) \|_{\ccalH}^2
 \leq &\frac{2\epsilon^2}{\eta^2}+2\|{\nabla}_f L_t(f_{t}(\bbS_t))\|^2.
 \end{align}
 For a Lipschitz continuous gradient function [Assumption \ref{as:4}] with ${\nabla}_f L_t(f_{t}^\star(\bbS_t))=0$, we have 
 \begin{align}
 \|{\nabla}_fL_t(f_{t}(\bbS_t))\|^2\leq 2L[L_t(f_t(\bbS_t)) -L_t(f_t^\star(\bbS_t))],
 \end{align}
 which implies that
 \begin{align}\label{proof_Strong_1121}
\!\!\!\!\!\!\!\! \| \tilde{\nabla}_fL_t(f_{t}(\bbS_t)) \|_{\ccalH}^2\!\leq\!\frac{2\epsilon^2}{\eta^2}+2L[L_t(f_t(\bbS_t)) -L_t(f_t^\star(\bbS_t))].
 \end{align}
 Next, substitute the upper bound in \eqref{proof_Strong_1121}  for the last term on the right hand side of \eqref{eq:iterate_prop1_subst0}, we obtain
 \begin{align}\label{eq:iterate_prop1_subst32}
 \|& f_{t+1} - f^\star_t \|_{\ccalH}^2 
 \\ 
&\leq\| f_t \!-\! f^\star_t \|_{\ccalH}^2
 	 \!-\! 2 \eta [L_t(f_t(\bbS_t)) \!-\!L_t(f_t^\star(\bbS_t))]+ 2 \eps \| f_t \!-\! f^\star_t \|_{\ccalH}\nonumber\\
 		&\quad 
 		 	+ 2\epsilon^2 +2\eta^2L[L_t(f_t(\bbS_t)) -L_t(f_t^\star(\bbS_t))] \nonumber	\\
 	&=\| f_t - f^\star_t \|_{\ccalH}^2
 		 - 2 \eta(1-\eta L) [L_t(f_t(\bbS_t)) -L_t(f_t^\star(\bbS_t))]\nonumber 
 		  	 \\
 		  	  	 &\quad+ 2 \eps \| f_t - f^\star_t \|_{\ccalH} 
 		+ 2\epsilon^2 \nonumber	\\
 	&\leq\| f_t - f^\star_t \|_{\ccalH}^2
 				 - 2 \eta(1-\eta L) [L_t(f_t(\bbS_t)) -L_t(f_t^\star(\bbS_t))]\nonumber 
 				 \\
 				 &\quad + \frac{4 \eps CX}{\lambda} 
 				+ 2\epsilon^2. \nonumber
 \end{align}
 %prop_bounded
 The second inequality in \eqref{eq:iterate_prop1_subst32} is obtained by using the  upper bound derived in Proposition \ref{prop_bounded}. 
 %\blue{We remark that the result in Proposition \ref{prop_bounded} utilizes Assumption \ref{as:first} to get the upper bound but it is possible to relax the Assumption \ref{as:first} if there exists a a finite constant $\sigma>0$ such that  $\|f_{t+1}^\star-f_t^\star\|_{\mathcal{\H}}\leq\sigma$ for all $t$, along the similar lines of \cite{bedi}.}
 To proceed further, we will use the following inequality.  For positive scalars $u$, $v$, and $w$ that satisfy $2u^2 > v$, by simple manipulation of the quadratic formula over the positive reals, it holds that
 \begin{align}\label{main_inequality}
 \sqrt{u^2-v+w^2} &\leq \sqrt{u^2 - v + v^2/4u^2 +  w^2}  \\
 &= \sqrt{u^2\left(1\!-\!\frac{v}{2u^2}\right)^2 + w^2} 
 \leq u\left(1\!-\!\frac{v}{2u^2}\right) + w\nonumber\\
 & \hspace{4cm}= u - \frac{v}{2u} + w.\nonumber
 \end{align} 
 The first inequality in \eqref{main_inequality} holds since we add a positive quantity $\frac{\nu^2}{4u^2}$ inside the square root. After rearranging the terms, we get the second equality of \eqref{main_inequality}. With the condition $2u^2>\nu$ in hand, we used the inequality $\sqrt{a+b}\leq \sqrt{a}+\sqrt{b}$ for any non-negative $a$ and $b$. Again rearranging the terms, we obtain the final equality in \eqref{main_inequality}. 
 
 We can use \eqref{main_inequality} to upper-estimate the right-hand side of \eqref{eq:iterate_prop1_subst32} with the following identifications: $u=\| f_t - f_t^\star \|_{\ccalH}$, $v=2 \eta (1-\eta L) [L_t(f_t(\bbS_t)) -L_t(f_t^\star(\bbS_t))]$, and $w=\sqrt{\frac{4\epsilon CX}{\lambda}+ 
 	 2\epsilon^2}$, such that 
  \begin{align}\label{eq:expectation_convexity24}
  \|f_{t+1} \!\!-\!\! f_t^\star \|_{\ccalH}
  	 \!\leq&\| f_t \!-\! f_t^\star \|_{\ccalH}
  	 	 \!-\!  \eta(1\!-\!\eta L) \frac{L_t(f_t(\bbS_t)) \!-\!L_t(f_t^\star(\bbS_t))}{\| f_t - f_t^\star \|_{\ccalH}}\nonumber\\
  	 	 &  +\sqrt{\frac{4\epsilon CX}{\lambda}+ 
  	 	 	 2\epsilon^2}. \; 
  \end{align} 
  The inequality in \eqref{eq:expectation_convexity24} holds since for a Lipschitz gradient convex loss  function (c.f. Assumption \ref{as:4}), we have
  \begin{align}\label{another}
  L_t(f_t(\bbS_t)) -L_t(f_t^\star(\bbS_t))\leq\frac{L}{2}\| f_t - f_t^\star \|_{\ccalH}^2.
  \end{align}   
 Note to satisfy the condition $u^2>\frac{\nu}{2}$, it is sufficient to show that $u^2>{\nu}$ holds. Note that from \eqref{another}, it holds that
 \begin{align}
 u^2\geq \frac{\nu}{\eta L(1-\eta L)}
 \end{align}
 from the definitions of $\nu$ and $u$. The required condition of $u^2>v$ holds if we select $\eta<\frac{1}{L}$. Next, in order to derive the dynamic regret, from triangle's inequality, it holds that
 \begin{align}\label{temp1}
  \| f_{t+1} - f_{t+1}^\star \|_{\ccalH}&=\| f_{t+1} -f_{t}^\star+f_{t}^\star- f_{t+1}^\star \|_{\ccalH}\nonumber  \\
  &\leq \| f_{t+1} - f_t^\star \|_{\ccalH}+ \| f_{t+1}^\star - f_t^\star \|_{\ccalH}.
  \end{align}
  Substitute the upper bound in \eqref{eq:expectation_convexity24} for the first term on the right hand side of \eqref{temp1}, we get
  \begin{align}
  \!\!\!\| f_{t+1} \!-\! f_{t+1}^\star \|_{\ccalH}\!\leq &\| f_t \!-\! f_t^\star \|_{\ccalH}
   	 	 \!-\!  \eta(1\!\!-\!\!\eta L) \frac{L_t(f_t(\bbS_t)) \!-\!L_t(f_t^\star(\bbS_t))}{\| f_t - f_t^\star \|_{\ccalH}}\nonumber 
   	 	 \\
   	 	 &  +\sqrt{\frac{4\epsilon CX}{\lambda}+ 
   	 	 	 2\epsilon^2}+\| f_{t+1}^\star - f_t^\star \|_{\ccalH}.\label{eq:expectation_convexity27} 
 \end{align}
  Next, rearranging the terms in \eqref{eq:expectation_convexity27}, and %we get
%   \begin{align}\label{eq:expectation_convexity25}
%   	 \eta (1-\eta L) &\frac{L_t(f_t(\bbS_t)) -L_t(f_t^\star(\bbS_t))}{\| f_t - f_t^\star \|_{\ccalH}}\leq\| f_t - f_t^\star \|_{\ccalH}
%   	 	 -\| f_{t+1} - f_{t+1}^\star \|_{\ccalH}    
%   	 	\nonumber\\& +  	 	\sqrt{\frac{4\epsilon CX}{\lambda}+ 
%   	 	   	 	 	 2\epsilon^2}+\| f_{t+1}^\star - f_t^\star \|_{\ccalH}. \; 
%   \end{align} 
%   
   utilizing  the upper bound in Proposition \ref{prop_bounded}, it holds that $\frac{1}{\| f_t - f_t^\star \|_{\ccalH}}>\frac{\lambda}{2CX}$ and we obtain
   \begin{align}\label{eq:expectation_convexity26}
&     	 \frac{\eta\lambda(1\!\!-\!\!\eta L)}{2CX}  L_t(f_t(\bbS_t)) \!-\!L_t(f_t^\star(\bbS_t))\!
     	   	 	 \nonumber
     	   	 	 \\
     	   	 	 & \qquad \leq\!\| f_t \!\!-\!\! f_t^\star \|_{\ccalH}
     	   	 	 \!-\!\| f_{t+1} \!\!-\!\! f_{t+1}^\star \|_{\ccalH}   \nonumber \\
			   &\qquad\quad +  	 	\sqrt{\frac{4\epsilon CX}{\lambda}+ 
     	   	 	   	 	 	 2\epsilon^2}+\| f_{t+1}^\star - f_t^\star \|_{\ccalH}. \; 
     \end{align}
     %
       %       \green{Still the aesthetic/layout of the equations could be improved. I modified the preceding expression to what I think is a publishable standard. Please see if you can change some of the others so they are easily interpretable.} 
     Take the summation from $t=1$ to $T$, we get
     \begin{align}\label{eq:expectation_convexity28}
         	& \frac{\eta\lambda(1-\eta L)}{2CX} \sum\limits_{t=1}^{T}[L_t(f_t(\bbS_t)) -L_t(f_t^\star(\bbS_t))]\\
	 &\quad \leq\| f_1 - f_1^\star \|_{\ccalH} 
         + T\sqrt{\frac{4\epsilon CX}{\lambda}+ 
         	   	 	   	 	 	 2\epsilon^2}\nonumber 
         	   	 	   	 	 	 \\
         	   	 	   	 	 	 &\qquad+\sum\limits_{t=1}^{T}\| f_{t+1}^\star - f_t^\star \|_{\ccalH}.\nonumber
         \end{align}
         We have dropped the negative terms on right hand side of \eqref{eq:expectation_convexity28}. Next, multiplying both sides  by $\frac{2CX}{\eta\lambda(1-\eta L)}$ and utilizing the definition of path length from \eqref{vtdy}, we get
         \begin{align}\label{eq:expectation_convexity29}
          \!\textbf{Reg}^D_T\!&\!\leq\!\!\frac{2CX\| f_1 \!\!-\!\! f_1^\star \|_{\ccalH}}{\eta\lambda(1\!\!-\!\!\eta L)}   
                 	   	 	\! +\!  	 	\frac{2CX}{\eta\lambda(1\!\!-\!\eta L)}\!\!\left(\!\!\!\!\sqrt{\frac{4T^2\epsilon CX}{\lambda}\!\!+\!\! 
                 	   	 	   	 	 	 2\epsilon^2T^2}\!\!+\!\!W_T\!\!\!\right).          \nonumber \\
 						& \leq \ccalO\left({ \frac{1 + T\sqrt{\epsilon}  + W_T }{\eta}}  \right)    	   	 	   	 	 	 
                 \end{align}
 which is sublinear in $T$ up to factors depending on path length $W_T$  for {$\epsilon=\mathcal{O}\left(T^{-\alpha}\right)$ with $\alpha\in(0,\frac{1}{p}]$} as stated in Theorem \ref{theorem:dynamic_path}\ref{theorem:dynamic_path_cvx}. \hfill $\qed$
 %Utilizing the in equality $\sqrt{a+b}\leq \sqrt{a}+\sqrt{b}$ again, we get 
 %   %
 %        \begin{align}\label{eq:expectation_convexity30}
 %        %
 %         \!\!\!\!\!\!       	  \sum\limits_{t=1}^{T}[L_t(f_t(\bbS_t)) -L_t(f_t^\star(\bbS_t))]\leq&\frac{2CX\| f_1 \!-\! f_1^\star \|_{\ccalH}}{\eta\lambda(1-\eta L)}   
 %                	   	 	\! +\!  	 	\frac{2CX}{\eta\lambda(1-\eta L)}\left(\!\!\!\sqrt{\frac{4CX}{\lambda}}\sqrt{\epsilon}T+ 
 %                	   	 	                	   	 	   	 	 	 \sqrt{2}\epsilon T+W_T\!\!\right).                	   	 	   	 	 	 
 %%						 %
 %%                \end{align}
 %%                %
 %                For a constant step size $\eta$, the final expression fo the dynamic regret is given by 
 %                %
 %                        \begin{align}\label{eq:expectation_convexity31}
 %                        %
 %                       	  \textbf{Reg}^D_T\leq&\mathcal{O}(1+\sqrt{\epsilon T}+W_T).                	   	 	   	 	 	 
 %                						 %
 %                                \end{align}
 %                                
 %                \red{                The expression in \eqref{eq:expectation_convexity31} is the final inequality for the dynamic regret for the convex loss function in terms of path length variation. We can optimization the value of $\epsilon$ to get the minimum possible value of dynamic regret.}
 % with the condition that
 % \begin{align}
 % \eta\leq\frac{2CX}{\lambda L}.
 % \end{align}
 % \end{proof}

\hspace{-5mm}  \emph{Proof of Theorem\ref{theorem:dynamic_path}\ref{theorem:dynamic_path_strong_cvx}}
 %%%%%%%%%%%%%%%%%%%%%%%%%%%%%%%%%%%%%%%%%%%%%%%%%%%%%%%%%%%%%%%
 %%%  P  R  O  O F  %%%%%%%%%%%%%%%%%
 %%%%%%%%%%%%%%%%%%%%%%%%%%%%%%%%%%%%%%%%%%%%%%%%%%%%%%%%%%%%%%%
 %\begin{proof}
 %
 Again, we begin with the descent related stated in Lemma \ref{lemma1} for time-varying optimizer $f_t^\star$:
 \begin{align}\label{eq:iterate_prop1_subst2}
 \| f_{t+1} - f^\star_t \|_{\ccalH}^2 
 	 &\leq\| f_t - f^\star_t \|_{\ccalH}^2
 	 - 2 \eta [L_t(f_t(\bbS_t)) -L_t(f_t^\star(\bbS_t))]\nonumber \\
 	 	 &\quad\!\!\!\!+\!\! 2 \eps \| f_t \!-\! f^\star_t \|_{\ccalH}\!+\! \eta^2 \| \tilde{\nabla}_fL_{t}(f_{t}(\bbS_t)) \|_{\ccalH}^2. \end{align}
 Consider the last term in \eqref{eq:iterate_prop1_subst2} as follows
 \begin{align}\label{iterate_prop1_subst001}
 \| &\tilde{\nabla}_fL_{t}(f_{t}(\bbS_t)) \|_{\ccalH}
 \\&\!\!\!=\!\!\|\! \tilde{\nabla}_{\!\!f}\!L_{t}(\!f_{t}(\bbS_t\!)\!)\!-\!\!{\nabla}_{\!\!f}L_{t}(\!f_{t}(\bbS_t))\!+\!\!{\nabla}_{\!\!f}\!L_{t}(f_{t}(\!\bbS_t))\!\!-\!{\nabla}_{\!\!f}L_{t}(f_{t}^\star(\bbS_t)) \|_{\ccalH}\nonumber
 \end{align}
 where we add and subtract the term $ {\nabla}_fL_{t}(f_{t}(\bbS_t))$ and utilize the optimality condition that ${\nabla}_fL_{t}(f_{t}^\star(\bbS_t))=0$. Using Cauchy-Schwartz inequality and $(a+b)^2\leq(2a^2+2b^2)$ in \eqref{iterate_prop1_subst001}, we get 
 \begin{align}\label{proof_Strong_1}
 \!\!\!\!\!\!\| \tilde{\nabla}_fL_{t}(f_{t}(\bbS_t)) \|_{\ccalH}^2\leq &\Big(\| \tilde{\nabla}_fL_{t}(f_{t}(\bbS_t))-{\nabla}_fL_{t}(f_{t}(\bbS_t))\|_{\ccalH}\nonumber 
 \\ &+\|{\nabla}_fL_{t}(f_{t}(\bbS_t))-{\nabla}_fL_{t}(f_{t}^\star(\bbS_t)) \|_{\ccalH}\Big)^2\nonumber 
 \\
 & \hspace{-5mm}\leq  2\| \tilde{\nabla}_fL_{t}(f_{t}(\bbS_t))-{\nabla}_fL_{t}(f_{t}(\bbS_t))\|_{\ccalH}^2\nonumber 
  \\ &\hspace{-5mm}\!\!\!+2\|{\nabla}_fL_{t}(f_{t}(\bbS_t))-{\nabla}_fL_{t}(f_{t}^\star(\bbS_t)) \|_{\ccalH}^2.
 \end{align}
 Next, utilizing the result of Proposition \ref{prop1} and Assumption \ref{as:4} into \eqref{proof_Strong_1}, we obtain
  \begin{align}\label{proof_Strong_12}
  \| \tilde{\nabla}_fL_{t}(f_{t}(\bbS_t)) \|_{\ccalH}^2\!\leq\! &\frac{2\epsilon^2}{\eta^2}\!+\!2\|{\nabla}_fL_{t}(f_{t}(\bbS_t))\!-\!{\nabla}_fL_{t}(f_{t}^\star(\bbS_t)) \|_{\ccalH}^2\nonumber\\
  \leq &\frac{2\epsilon^2}{\eta^2}+2L^2\|f_t-f_t^\star\|_{\mathcal{H}}^2.
  \end{align}
 The last inequality in \eqref{proof_Strong_1} holds from Assumption \ref{as:4}. Next, substitute the upper bound in \eqref{proof_Strong_1}  into \eqref{eq:iterate_prop1_subst2}, we obtain
 \begin{align}\label{eq:iterate_prop1_subst3}
 \| f_{t+1}\! - \!f^\star_t \|_{\ccalH}^2 
 	 &\leq\| f_t \!-\! f^\star_t \|_{\ccalH}^2
 	 - 2 \eta [L_t(f_t(\bbS_t)) -L_t(f_t^\star(\bbS_t))] \nonumber \\
 	  	 	 &\quad\!\!\!\!+ 2 \eps \| f_t \!\!-\!\! f^\star_t \|_{\ccalH}	\!\!+\!\! 2\epsilon^2\!\!+\!\!2\eta^2L^2\|f_t\!-\!f_t^\star\|_{\mathcal{H}}^2.
 \end{align}
 %
 %%%%%%%%%%%%%%%%%%%%%%%%%%%%%%%%%%%%%%%%%%%%%%%%%%%%%%%%%%%%%%%%%%%%%%%%%%%%%%%%%%%%%%%%%%%%%%%%%%%%%%%%%%%%%%%%%%%%%%%%%%%%%%%%%%%%%%%%%%%%%%%%%%%%%%%%%%%%%%%%%%%%%%%%%%%%%%%%%%%%%%%%%%%%%%%%%%%%%%%%%%%%%%%%%%%%%%%%%%%%%%%%%%%%%%%%%%%%%%%%%%%%%%%%%%%%%%%%%%%%%%%%%%%%%%%%%%%%%%%%%%%%%%%%%%%%%%%%%%%%%%%%%%%%%%%%
\begin{figure*}[t]
	\centering 
	\subfigure[Dynamic regret]{\includegraphics[width=0.33\linewidth,height=4cm]{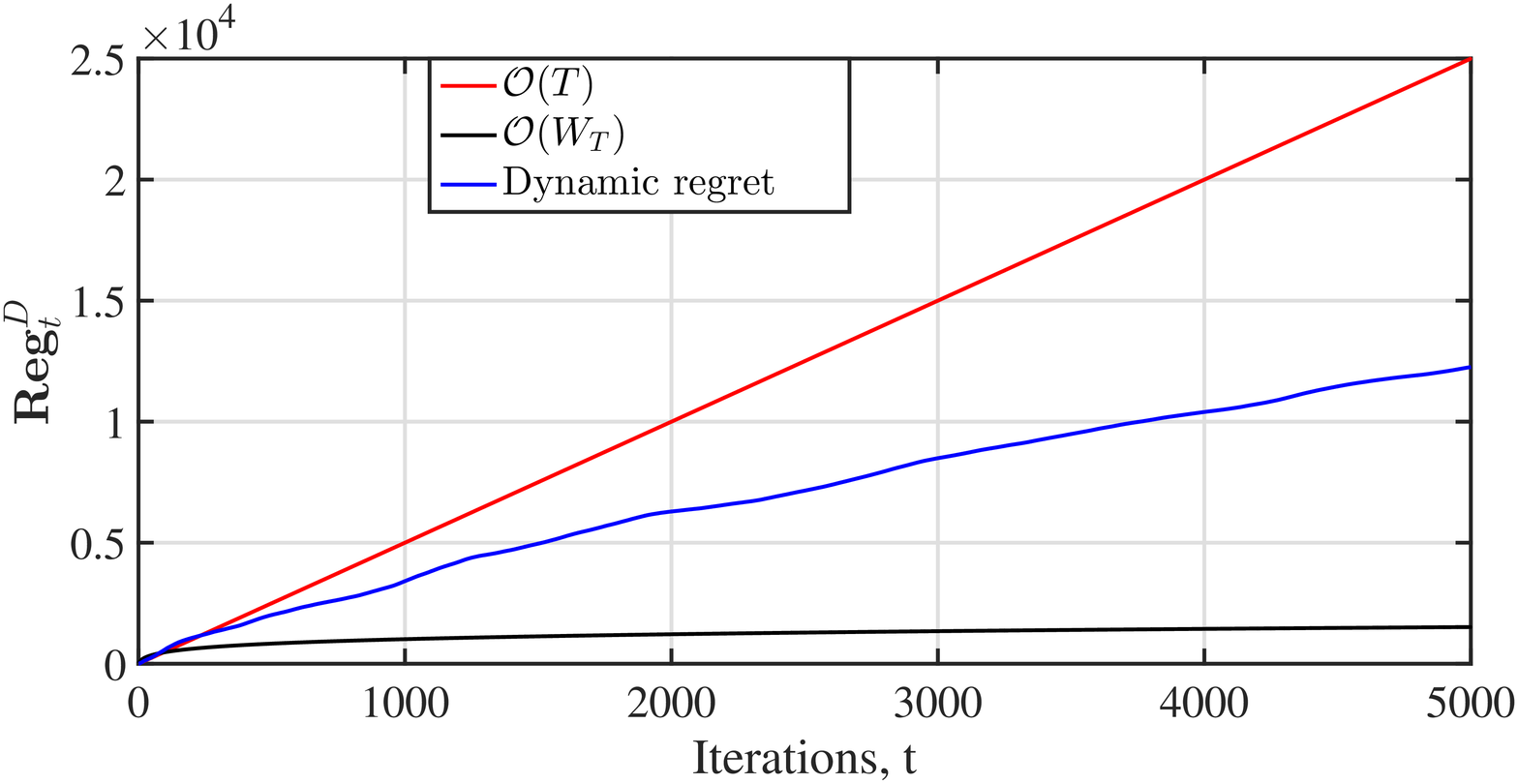}\label{regression1}}	
	\subfigure[Model order $M_t$]{\includegraphics[width=0.325\linewidth,height=4cm]{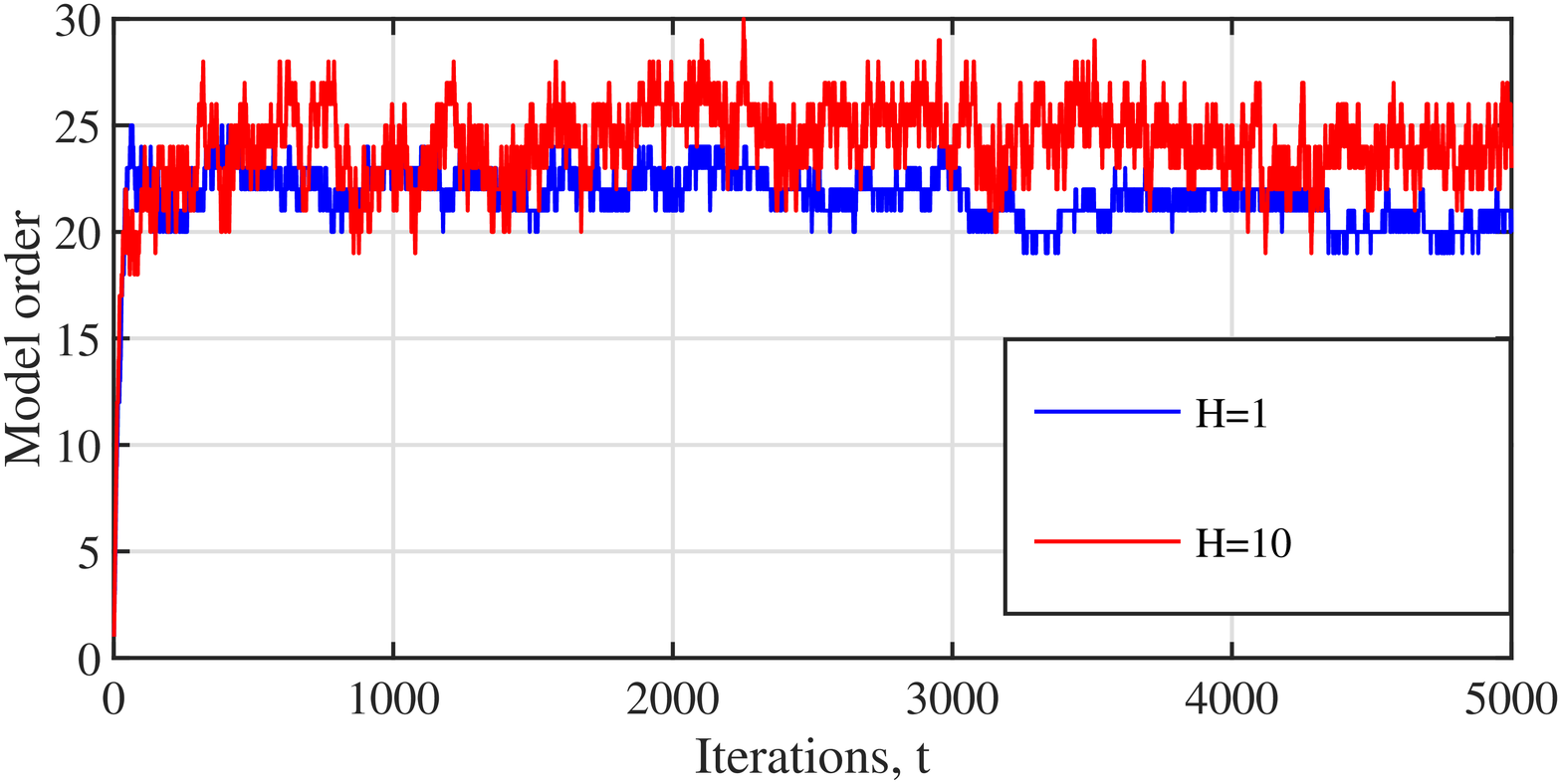}\label{regression2}}
	\subfigure[Non-Stationary Tracking]{\includegraphics[width=0.33\linewidth,height=4cm]{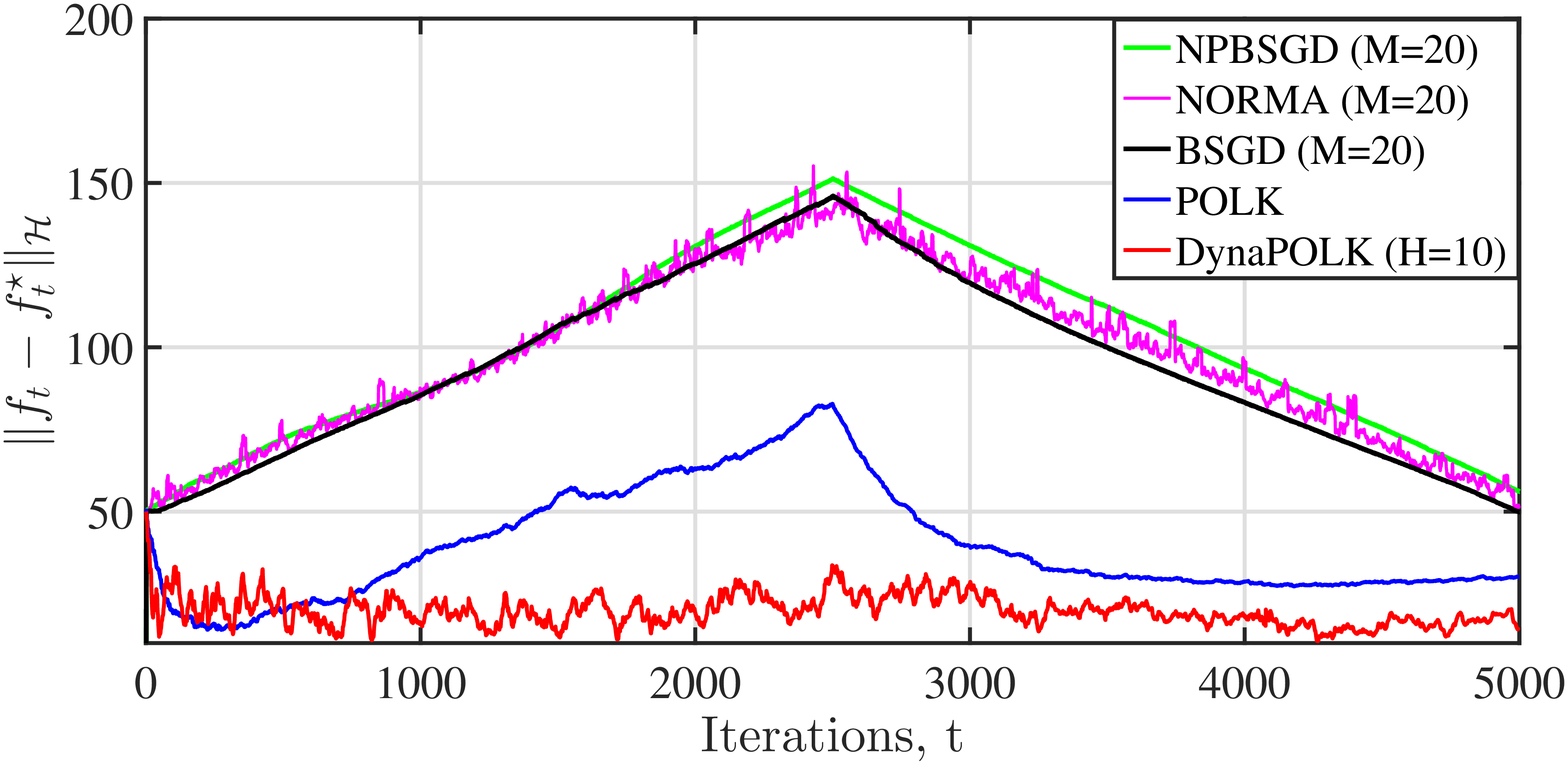}\label{comparisons}}\caption{Experiments with non-stationary nonlinear regression common to phase retrieval: scalar targets are $y_t=a_t\text{sin}(b_t\bbx_t+c_t)+\eta_t$, which one would like to predict via sequentially observed  $\bbx_t$, where $\eta_t$ is additive Gaussian noise. DynaPOLK attains sublinear regret, and is able to track a shifting nonlinearity with low model complexity. In contrast, alternatives are unable to adapt to drift.  \vspace{-0mm}}
		\end{figure*}
 From the strong convexity of the objective function (Assumption \ref{as:5}), we have \eqref{eq:strong_convexity}, which we may substitute in for the second term on right hand side of \eqref{eq:iterate_prop1_subst3} to obtain
 \begin{align}\label{eq:iterate_prop1_subst4}
 \| f_{t+1} - f^\star_t \|_{\ccalH}^2 
 	 &\leq\| f_t - f^\star_t \|_{\ccalH}^2
 	 - 2 \eta\mu \| f_t - f^\star_t \|_{\ccalH}^2\\
 	  		 &\quad + 2 \eps \| f_t - f^\star_t \|_{\ccalH} 
 	+ 2\epsilon^2+2\eta^2L^2\|f_t-f_t^\star\|_{\mathcal{H}}^2. \nonumber\\
 	%\
 	&\leq\!(1\!-\!2\eta\mu\!+\!2\eta^2L^2)\| f_t \!-\! f^\star_t \|_{\ccalH}^2\!+\!\epsilon\frac{8CX}{\lambda}\!+\!2\epsilon^2\nonumber
 \end{align}
 where for the second inequality we have used the statement of Proposition \eqref{prop_bounded} for the third term on the right-hand side of the first inequality. 
  Take square root on both sides of \eqref{eq:iterate_prop1_subst4}, we get
 \begin{align}\label{eq:iterate_prop1_subst7}
 \| f_{t+1} - f^\star_t \|_{\ccalH} 
 	&\leq\rho\| f_t - f^\star_t \|_{\ccalH}+\sqrt{\epsilon\frac{8CX}{\lambda}+2\epsilon^2},
 \end{align}
 where $\rho:=\sqrt{(1-2\eta(\mu-\eta L^2))}$. The value of $\rho\in (0,1)$ defines a contraction mapping provided $\eta$ satisfies $0<\eta<\frac{\mu}{L^2}$. With the help of triangle inequality, we can write the difference $\| f_{t+1} - f_{t+1}^\star \|_{\ccalH}$ as
 \begin{align}\label{last}
 \| f_{t+1} - f_{t+1}^\star \|_{\ccalH}\leq\| f_{t+1} - f_{t}^\star \|_{\ccalH}+\| f_{t+1}^\star - f_{t}^\star \|_{\ccalH}.
 \end{align}
 Utilize the upper bound in \eqref{eq:iterate_prop1_subst7} into \eqref{last}, and taking the summation over $t$ on the both sides, we get 
 %we get
 %
% \begin{align}\label{eq:iterate_prop1_subst9}
% %
% \| f_{t+1} \!\!-\!\! f_{t+1}^\star \|_{\ccalH}\!\leq\!\rho\| f_t \!-\! f^\star_t \|_{\ccalH}\!\!+\!\!\sqrt{\epsilon\frac{8CX}{\lambda}\!+\!2\epsilon^2}\!+\!\| f_{t}^\star \!-\! f_{t-1}^\star \|_{\ccalH}.
% %
% \end{align}
% %
% Take the summation over $t$ on both sides,  we get
 %
 \begin{align}\label{eq:iterate_prop1_subst10}
 \sum\limits_{t=1}^{T}\| f_{t} - f_{t}^\star \|_{\ccalH}\leq&\| f_{1} - f_{1}^\star \|_{\ccalH}+\rho\sum\limits_{t=1}^{T}\| f_t - f^\star_t \|_{\ccalH} 
 \\
 &+\!\!\sqrt{\epsilon T^2\frac{8CX}{\lambda}\!+\!2\epsilon^2T^2}\!+\!\sum\limits_{t=1}^{T}\| f_{t}^\star \!-\! f_{t-1}^\star \|_{\ccalH}.\nonumber
 \end{align}
 After rearranging and dividing the both sides by $1-\rho$, we get
 \begin{align}\label{eq:iterate_prop1_subst11}
 \sum\limits_{t=1}^{T}\| f_{t} \!\!-\!\! f_{t}^\star \|_{\ccalH}
 &\!\!\leq\!\!\frac{\| f_{1} \!\!-\!\! f_{1}^\star \|_{\ccalH}}{1-\rho}\!+\!\frac{1}{1\!\!-\!\!\rho}\left(\!\!\!\sqrt{\epsilon T^2\frac{8CX}{\lambda}\!+\!2\epsilon^2T^2}\!+\!W_T\!\!\right)\nonumber \\
 &\qquad \leq \ccalO\left(\frac{1 + {T \sqrt{\epsilon} } + W_T}{1-\rho} \right)
 \end{align}
 where we have used the  definition of path length $W_T$ \eqref{vtdy} on the right-hand side of \eqref{eq:iterate_prop1_subst10}. From the first order convexity condition, we can write 
 \begin{align}
 \!\!\sum\limits_{t=1}^{T}[L_t(f_t(\bbS_t)) \!\!-\!\!L_t(f_t^\star(\bbS_t))]\leq\!\!& \sum\limits_{t=1}^{T}\langle \nabla_f L_t(f_t(\bbS_t)), f_t\!-\!\!f_t^\star\rangle_{\mathcal{H}}\nonumber
 \\
 &\hspace{-2cm}\leq\sum\limits_{t=1}^{T}\|\nabla_f L_t(f_t(\bbS_t))\|_{\mathcal{H}} \|f_t-f_t^\star\|_{\mathcal{H}}\label{gradient}
 \end{align}
 where the second inequality in \eqref{gradient} holds sue to Cauchy-Schwartz inequality. Next, since the space $\mathcal{X}$ is compact, the gradient norm $\|\nabla L_t(f_t(\bbS_t))\|_{\mathcal{H}}$ evaluated for any $\bbS_t$ will be upped bounded by some finite constant $G$, which implies that $\|\nabla L_t(f_t(\bbS_t))\|_{\mathcal{H}}\leq G$. Using the gradient upper bound on the right hand side of \eqref{gradient}, we obtain
 \begin{align}
  \sum\limits_{t=1}^{T}[L_t(f_t(\bbS_t)) -L_t(f_t^\star(\bbS_t))]
  \leq & G\sum\limits_{t=1}^{T}\|f_t-f_t^\star\|_{\mathcal{H}}.\label{gradient2}
  \end{align}
  Next, utilizing the upper bound in \eqref{eq:iterate_prop1_subst11} into the right hand side of \eqref{gradient2}, we obtain the final regret result as
  \begin{align}\label{final}
    \textbf{Reg}_T^D= & \ccalO\left(\frac{1 + {T\sqrt{\epsilon}  }+ W_T}{1-\rho} \right).
    \end{align}
    Observe that \eqref{final} is sublinear in $T$ up to terms depending on the path length for any step-size $\eta$ and for compression constant \red{$\epsilon=\mathcal{O}\left(T^{-\alpha}\right)$ with $\alpha\in(0,\infty]$}.
    % By maximizing $\rho$ with respect to the step-size, we obtain $\eta\geq \ccalO(T^{-1/2})$. 
 %\end{proof}
% \blue{
 The expression in \eqref{eq:iterate_prop1_subst11} is similar to the one on \eqref{eq:dynamic_path_thm} except for the term $(1-\rho)$ in the denominator. If we choose $\eta$ such that $(1-\rho)>\eta$, the results for strongly convex functions is improved. Rearrange this expression to obtain
                 $$ (1 - \eta)^2 >\rho^2 = 1 - 2 \eta (\mu - \eta L^2)$$
                 which, upon solving for a condition on $\eta$, simplifies to 
                 $$\eta < \frac{2 (\mu - 1)}{2 L^2 -1}.$$\hfill $\qed$
                 
 %                Of course, any $\eta=\ccalO(T^{-a})$ satisfies the preceding expression for sufficiently large $T$.%}
                 %
                 %
%\end{proof}
%
%
Theorem \ref{theorem:dynamic_path} generalizes existing dynamic regret bounds of \cite{Zinkevich2003,hall2015online,besbes,mokhtari2016online} to the case where decisions are defined by functions $f_t$ belonging to RKHS $\mathcal{H}$. To facilitate this generalization, gradient projections are employed to control function complexity, which appears as an additional term depending on compression budget $\epsilon$ in the dynamic regret bounds, in particular, the product $T\sqrt{\epsilon}$ in the expressions \eqref{eq:dynamic_path_thm} and \eqref{eq:dynamic_path_thm_strongcvx}. For smaller $\epsilon$, the regret is smaller, but the model complexity increases, and vice versa. Overall, this compression induced error in the gradient is a version of inexact functional gradient descent algorithm with a tunable tradeoff between convergence accuracy and memory. Note that for $\epsilon=0$, these results becomes of the order of $\mathcal{O}{(1+W_T)}$ which matches \cite{mokhtari2016online} and improves upon existing results \cite{Zinkevich2003,hall2015online,besbes}. Even for the strongly convex case  with $\epsilon=0$, we obtain $o(1+W_T)$ which is better than its parametric counterpart obtained in \cite{mokhtari2016online}. 

Regarding the complexity reduction technique for kernel methods, we note that dynamic regret bounds for random feature approximations in the looser sense of \eqref{dynamic_regret_vT22} have been recently established \cite{shen2019random}. These results hinge upon tuning the random feature incurred error to gradient bias. However, in practice, the number of random features required to ensure a specific directional bias is unknown, which experimentally dictates one using a large enough number of random features to hope the bias is small. However, this error is in the \emph{function representation} itself, not the gradient direction. This issue could be mitigated through double kernel sampling \cite{dai2014scalable}, a technique whose use in non-stationary settings remains a direction for future research.
 %\red{Alec, can we write a stong point here in comparison to the method and analysis proposed in \cite{shen2019random}. One point for sure is that the results presented in \cite{shen2019random} are not in terms of $W_T$. They also not have results for the strongly convex case.}

{\bf Parameter Selection} For step-size \red{\red{$\eta<\min\{\frac{1}{\lambda},\frac{1}{L}\}$} and compression budget $\eps=\mathcal{O}(T^{-\alpha})$, substituted into Lemma \ref{theorem_model_order} yields model complexity $   M= \mathcal{O}(T^{\alpha p})$. To obtain sublinear regret (up to factors depending on $W_T$) \emph{and} model complexity in the non-strongly convex case, we require \red{$\alpha\in(0,\frac{1}{p}]$} and $\alpha p\in(0,1)$, which holds, for instance, if  $\eps=T^{-1/(p+1)}$.} {Note that the dynamic regret result in \eqref{eq:dynamic_path_thm} and the model order, using Lemma \ref{theorem_model_order}, becomes  
                 \begin{align}\label{last_11111}
                  \textbf{Reg}_T^D=\mathcal{O}\left(1+T^{{(1-\frac{\alpha}{2})}}+W_T\right)\;, \quad
%                 \end{align}
                 %and
%                  \begin{align}
                                 M=  \mathcal{O}(T^{\alpha p}).
                                 \end{align}
                                 
                                 For the regret to be sublinear, we need  $\red{{\alpha}\in(0,\frac{1}{p}]}$. As long as the dimension  $p$ is not too large, we always have a range for $\alpha$. This implies that $\red{\alpha p\in(0,1)}$ and hence $M$ is sublinear.

                                   \begin{table}[t]
                                                 	\centering
                                                 	\resizebox{\columnwidth}{!}{\begin{tabular}{|c|c|c|c|}
                                                 	                                                 		\hline
                                                 	                                                 		$\alpha$& Regret & M & Comments	 \\
                                                 	                                                 		\hline
                                                 	                                                 		 \red{$\alpha=0$}& $\mathcal{O}(T)+W_T$ & $\mathcal{O}(1)$                  	& \text{Linear regret} \\
                                                 	                                                 		\hline
                                                 	                                                 		$\red{{\alpha}=\frac{1}{p}}$& $\mathcal{O}\left(T^{\frac{(2p-1)}{2p}}+W_T\right)$ & $\mathcal{O}(T)$                  	& \text{Linear } $M$ \\
                                                 	                                                 		       		\hline
                                                 	                                                		\red{$\alpha=\frac{1}{p+1}$}& $\mathcal{O}\left(T^{\frac{2p+1}{2p+2}}+W_T\right)$ & $\mathcal{O}(T^{p/(1+p)})$                  	& \text{Sublinear } $M$ \\
                                                 	                                                 		                		                		                		\hline
                                                 	                                                 	\end{tabular}}\vspace{2mm}
                                                 	
                                                 	\caption{Summary of dynamic regret rates for convex loss function. {Note that the same rates are obtained for the strongly convex loss function but $\mathcal{O}$ is replaced by small $o$.}}
                                                 	\label{tab:my_label21}
                                                 \end{table}}

 Observe that the rate for the strongly convex case \eqref{eq:dynamic_path_thm_strongcvx} is strictly better the non-strongly convex counterpart \eqref{eq:dynamic_path_thm_strongcvx} whenever  $\eta$ satisfies $(1-\rho)>\eta$. This holds, provided 
 %
 %, the results for strongly convex functions is improved. Rearrange this expression to obtain
 %                $$ (1 - \eta)^2 >\rho^2 = 1 - 2 \eta (\mu - \eta L^2)$$
  %               %
  %               which, upon solving for a condition on $\eta$, simplifies to 
 %                %
                 $\eta < ({2 (\mu - 1)})/({2 L^2 -1})$.
                 %
                 %Of course, any $\eta=\ccalO(T^{-a})$ satisfies this condition for sufficiently large $T$. 
                 Taken together, Theorems \ref{theorem:dynamic_cost} - \ref{theorem:dynamic_path} establish that Algorithm \ref{alg:soldd} is effective for non-stationary learning problems. In the next section, we experimentally  benchmark these results on representative tasks. %and compare them to alternatives.

	\begin{figure*}
 		\centering \vspace{-2mm}\hspace{-7mm} 
		\subfigure[No. of misclassifications (MC)]{\includegraphics[width=.36\linewidth,height=4cm]{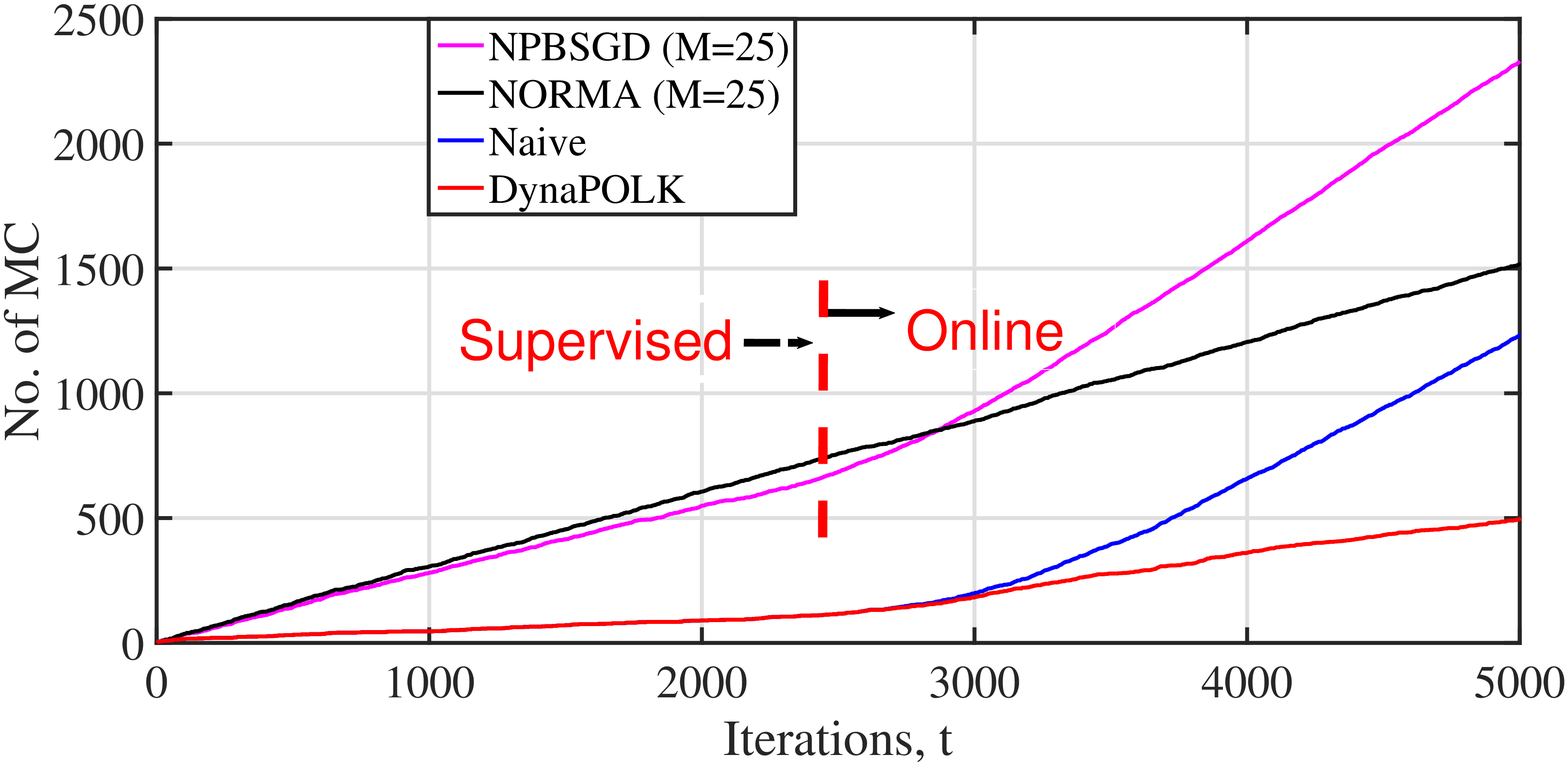}\label{classification1} \hspace{-7mm}}
 		\subfigure[Model order $M_t$]{\includegraphics[width=.36\linewidth,height=4cm]{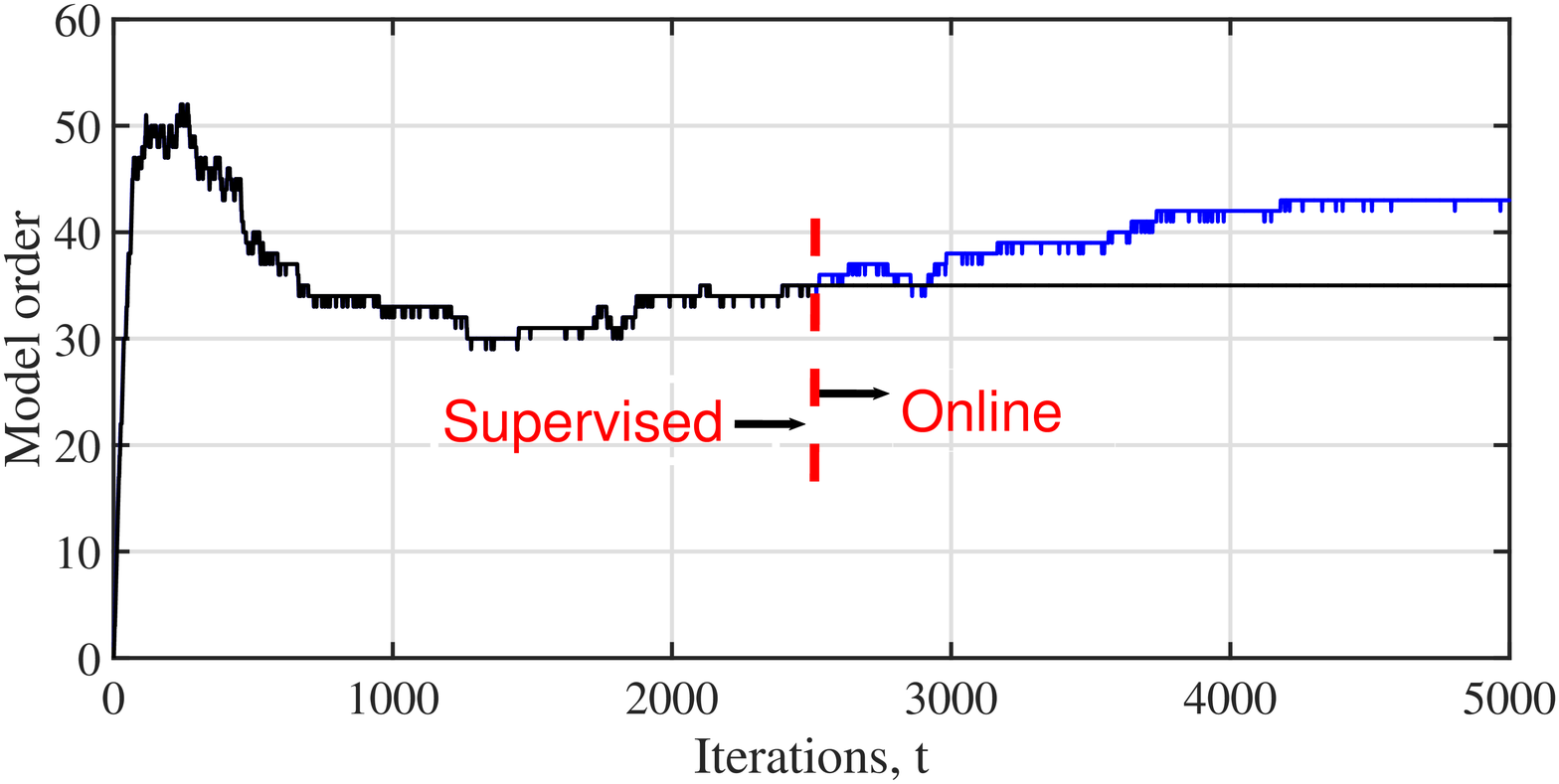}\label{classification2}\hspace{-7mm}}	
 		\subfigure[MSE for Classification problem]{\includegraphics[width=.36\linewidth,height=4cm]{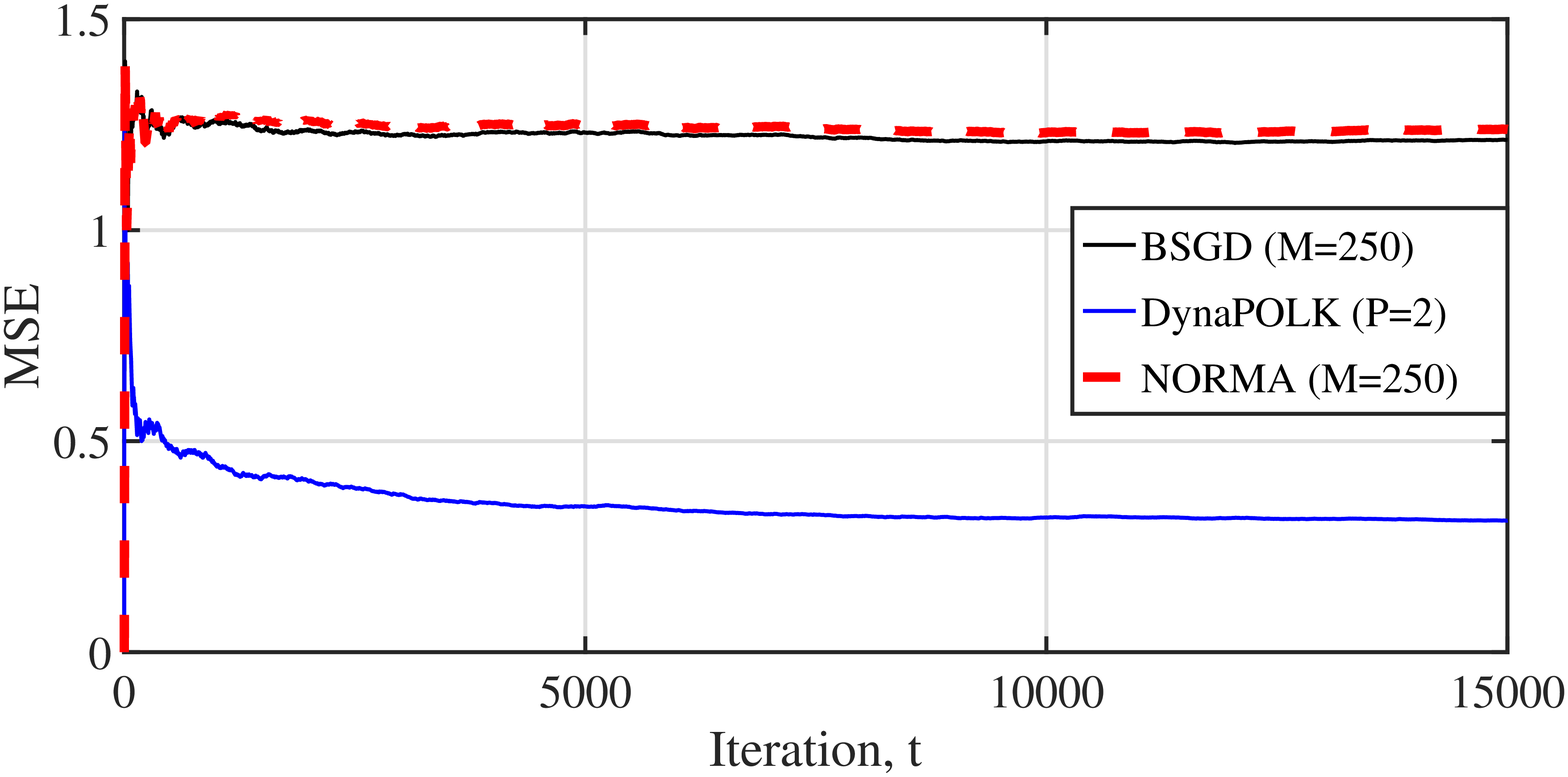}
		\label{mse_classification}}
%		
%	%	\includegraphics[scale=0.23]{mse_sea_comparisons.eps}
%		
%		
%		\hspace{-6mm}\label{classification113}}
 		\caption{Comparison of DynaPOLK to other kernel methods (left) for an online non-stationary classification on Gaussian Mixtures data \cite{Zhu2005} with dynamic class means. Alternative methods experience nearly linear regret, and their mean-square error on the time-series classification problem defined in \cite{street2001streaming} is relatively uncontrolled (right). 
 		%\green{Does it make sense to have a plot that is disassociated with the table? I'd expect to see Raker in the plot, if it's in the table... }\vspace{0mm} 
 		}
 		\label{classification}	
 	\end{figure*}

\section{Experiments} \label{sec:experiments}
%\green{Please try to maintain the aesthetic of three side-by-side figures at the top of each of the these three pages. It's good design.}
In this section, we evaluate the ability of Algorithm \ref{alg:soldd} to address online regression and classification in non-stationary regimes and compare it with some alternatives.

%\subsection{Online Regression}
{\bf Online Regression} We first consider a simple online regression to illustrate performance: target variables are of the form $y_t=a_t\text{sin}(b_t\bbx_t+c_t)+\eta_t$, which one would like to predict upon the basis of sequentially observed values of $\bbx_t$. Here $\eta\sim\mathcal{N}(\mu_t,\sigma^2)$ is Gaussian noise. Such models arise in phase retrieval, as in medical imaging, acoustics, or communications. Non-stationarity comes from parameters $(a_t,b_t,c_t)$ changing with $t$: $a_t$ and $c_t$ increase from $0$ to $3$ and then decrease to $1$, both linearly, while $b_t$ is increased from $0$ to $1$ linearly.
We consider a square loss function given by $\ell_t(f(\bbx))=(f(\bbx)-y_t)^2$ and run the simulations for $T=5000$ iterations. For experiments, we select Gaussian kernels of bandwidth $\sigma=0.252$, step-size $\eta={T^{-0.4}}$, and compression parameter $\epsilon={T^{-0.1}}$. The dynamic regret for $H=1$ is shown in Fig. \ref{regression1} -- observe that it grows sublinearly with time. Path length $W_T$ is shown for reference.
 Fig. \ref{regression2} shows the model order relative to time for window lengths $H=1$ and $H=10$, which remains moderate.  Observe that Algorithm \ref{alg:soldd} is able to track shifting data more gracefully with larger $H$ as clear from Fig. \ref{classification111}. {This figure shows the true function at the first and last time, i.e., $f_1(x)$ at iteration $1$ to $f_T(x)$ at iteration $T$. The red curve shows the learned function via DynaPOLK, which better adheres to the target for $H=10$. An animation of online nonlinear regression in the presence of non-stationarity is appended to this submission.} and the supplementary regression video. We further compare DynaPOLK against the alternative methods, namely, NPBSGD \cite{le2016nonparametric}, NORMA \cite{Kivinen2004}, BSGD \cite{wang2012breaking}, and POLK \cite{koppel2019parsimonious}. We plot the distance from the optimal $\|f_t-f_t^\star\|_{\mathcal{H}}$ in Fig. \ref{comparisons}. Fig. \ref{comparisons} we observe DynaPOLK with $H=10$ is able to track the time-varying nonlinearity, whereas the others experience nearly linear regret during the non-stationary phase. {We remark that a recent algorithm AdaRaker is proposed in \cite{shen2019random} to solve the nonparametric online learning problems. The authors in \cite{shen2019random} shows that AdaRaker performs better than all the other available techniques in the literature. Hence, in this work, we compare the proposed DynaPOLK algorithm mainly with the algorithms of \cite{shen2019random} and show the improvement as provided in Table \ref{comparisos_literature} (see \cite{shen2019random} for the datasets description). }
 %\green{See comment in figure above. It's strange to mention Raker for comparison in some points but not others...} }
 %
 \begin{table}[t]
 \centering
\resizebox{\columnwidth}{!}{ \begin{tabular}{|l|l|l|l|l|}
 \hline
 \textbf{Algorithms/Dataset} & \textbf{Twitter} & \textbf{Tom} & \textbf{Energy} & \textbf{Air} \\ \hline
 %Raker                       & 3.0              & 3.4          & 19.3            & 2.0          \\ \hline
 AdaRaker                    & 2.6              & 1.9          & 13.8            & 1.3          \\ \hline
 \textbf{DyanPOLK}           & \textbf{0.06}                 & \textbf{0.68}         &    \textbf{0.0052}             &   \textbf{0.14}           \\ \hline
 Model order (DyanPOLK)           & 50                & 24         &    31             &   33           \\ \hline
 \end{tabular}}
 \caption{MSE ($10^{-3}$) performance of the different algorithms with $B = D = 50$ (as in \cite{shen2019random}). }
 \label{comparisos_literature}
 \end{table}
%%%%%%%% %%%%%%%% %%%%%%%% %%%%%%%% %%%%%%%% %%%%%%%% %%%%%%%% %%%%%%%% %%%%%%%% %%%%%%%% %%%%%%%% %%%%%%%% %%%%%%%% %%%%%%%% %%%%%%%% %%%%%%%% %%%%%%%% %%%%%%%% %%%%%%%% %%%%%%%% %%%%%%%% %%%%%%%% %%%%%%%% %%%%%%%% %%%%%%%%
%\begin{figure}[t]
%%	\centering
%	\subfigure[No. of misclassifications (MC)]{\includegraphics[scale=0.23]{comparisons_classification.eps}\label{classification1}}	\\
%	\subfigure[Model order $M_t$]{\includegraphics[scale=0.23]{classificaition_model.eps}\label{classification2}}
%%	\subfigure[No. of misclassifications (MC)]{\includegraphics[width=0.32\linewidth,height=3.5cm]{com_dnn.eps}	\label{comp_dnn}}
%	\caption{\blue{These figures need to be side by side.}Comparison of DynaPOLK to other kernel methods (left) for an online non-stationary classification on Gaussian Mixtures data \cite{Zhu2005} with dynamic class means. Alternative methods experience nearly linear regret. }\vspace{-0mm}
%	\label{classification}	
%\end{figure}

%%%%%%%% %%%%%%%% %%%%%%%% %%%%%%%% %%%%%%%% %%%%%%%% %%%%%%%% %%%%%%%% %%%%%%%% %%%%%%%% %%%%%%%% %%%%%%%% %%%%%%%% %%%%%%%% %%%%%%%% %%%%%%%% %%%%%%%% %%%%%%%% %%%%%%%% %%%%%%%% %%%%%%%% %%%%%%%% %%%%%%%% %%%%%%%% %%%%%%%%
	\begin{figure*}
 		\centering \vspace{-2mm}\hspace{-7mm} 
		\subfigure[Initial $and$ final nonlinearity]{\includegraphics[width=.35\linewidth,height=4cm]{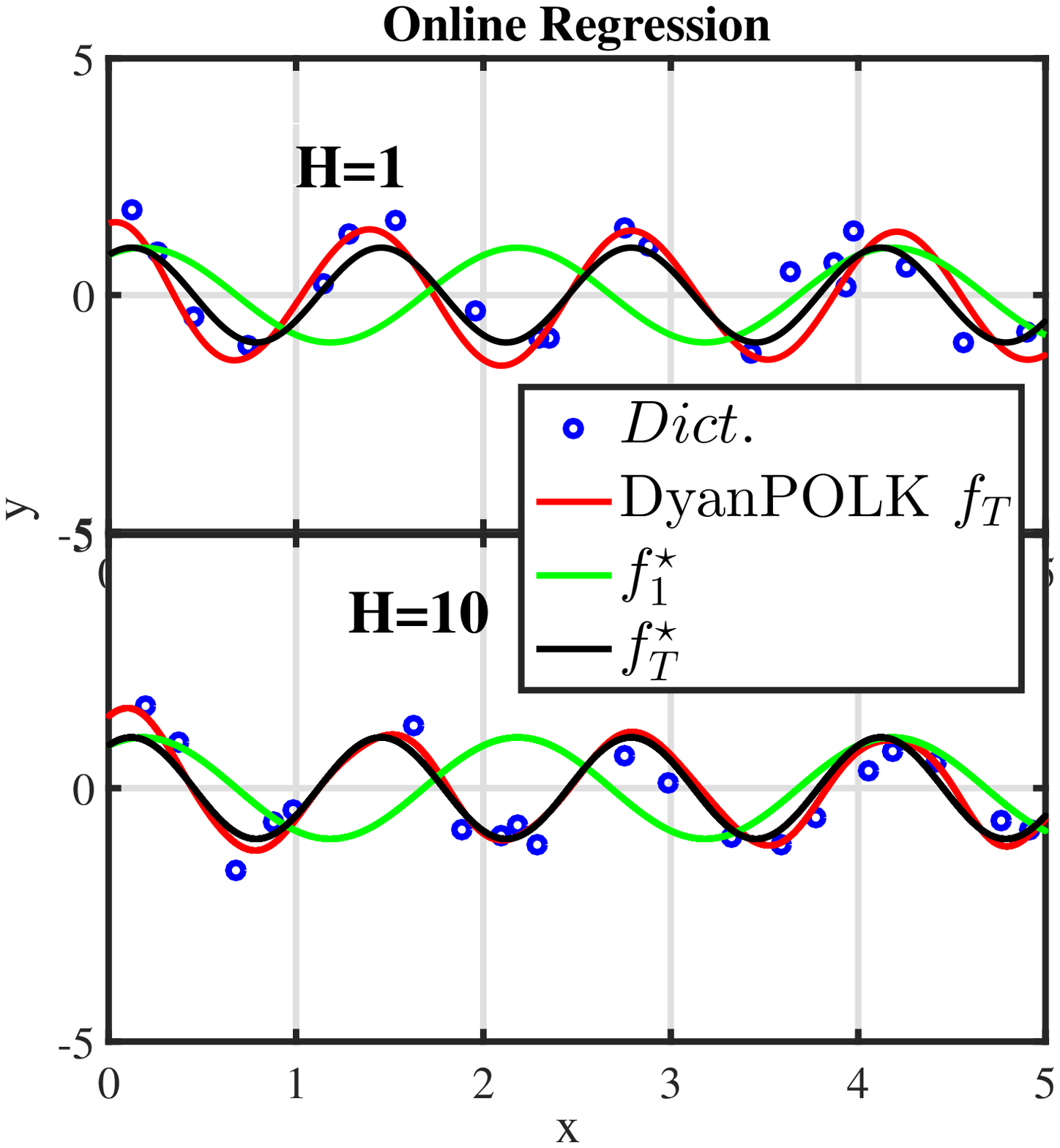}\label{classification111} \hspace{-1.5mm}}
 		\subfigure[Stationary classifier at $t=1000$]{\includegraphics[width=.35\linewidth,height=4cm]{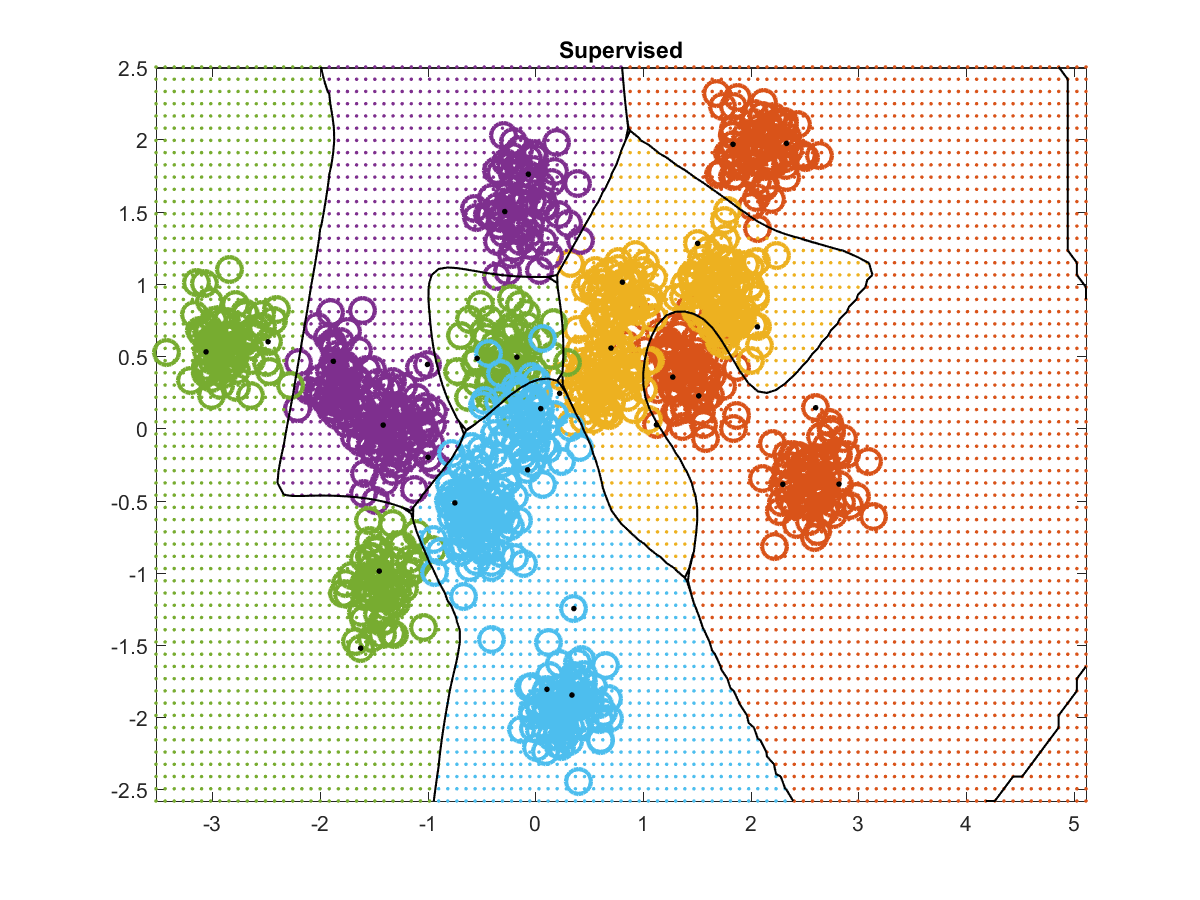}\label{classification112}\hspace{-5mm}}	
 		\subfigure[Drifted classifier at $t=5000$]{\includegraphics[width=.35\linewidth,height=4cm]{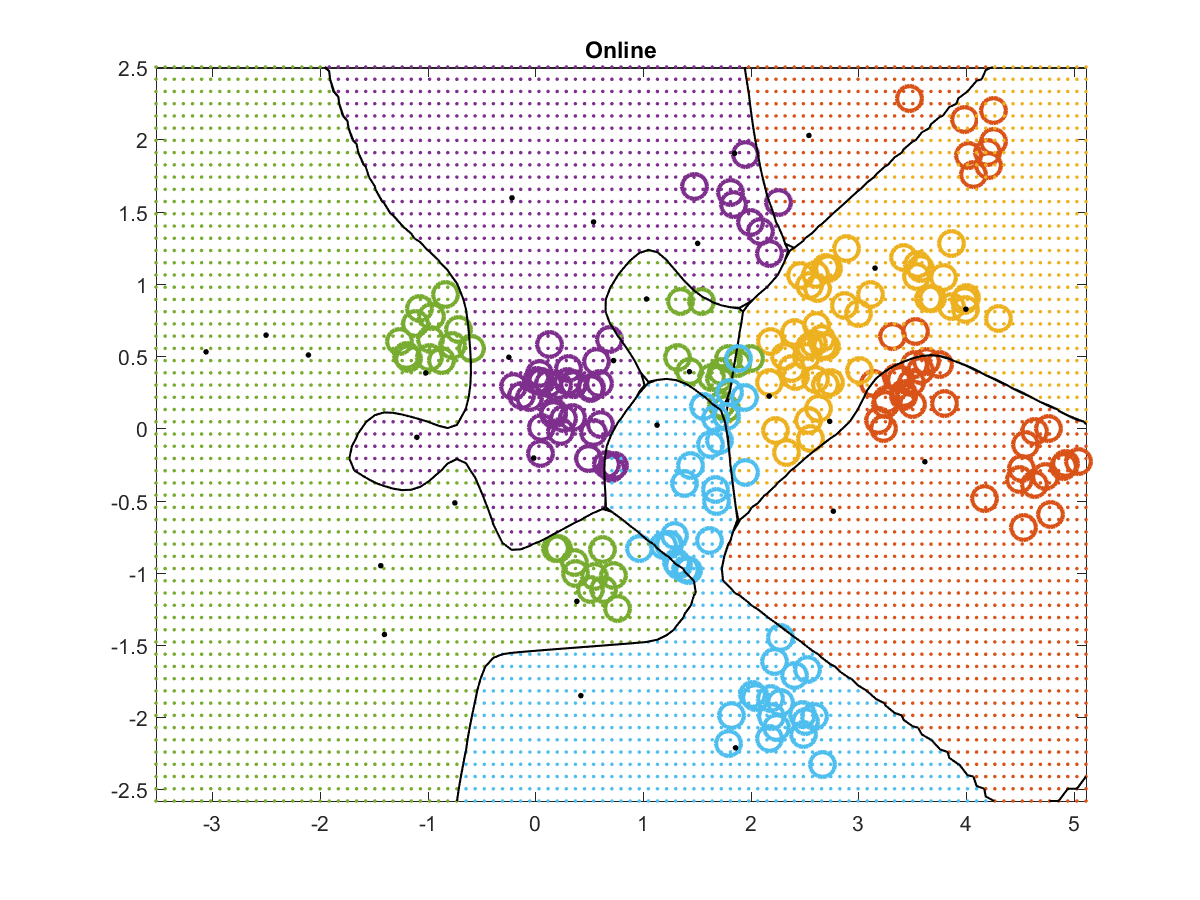}\hspace{-6mm}\label{classification113}}
 		\caption{Left: regression with initial $\&$ final target denoted as $f_1^\star \  \& \ f_T^\star$. DynaPOLK tracks nonlinearity drifting with $(a_t,b_t,c_t)$. Windowing ($H=10$) improves performance. Center: decision surface of DynaPOLK on stationary Gaussian Mixtures \cite{Zhu2005}. Right: classifier adapting to data drift.\vspace{0mm} }
 		\label{classification12}	
 	\end{figure*}
{\bf Online Classification}
Consider the multi-class classification in non-stationary environments, a salient problem in terrain adaption of autonomous systems \cite{sun2010learning}. 
%
%At each time $t$, a robot observes a collection of $P$ images $\{\bbx_\tau\}\subset\ccalX\subset\reals^p$ which could be one of $C$ classes $\ccalY=\{1,\dots,C\}$, where horizon length $P$ depends on the platform's velocity. The robot selects a nonlinear classifier $f_t : \ccalH \rightarrow \ccalY$, after which ground truth is revealed in the form of a convex cost $L_t(f(\bbx_t)$, i.e., the hinge or logistic loss \cite{Murphy2012}. Since the environment is dynamic, the cost of label misclassification changes over time according to, e.g., wind gusts, vehicle slip, or concept drift such as the changing color of grass/pavement. The quantities \eqref{variation} - \eqref{vtdy} for this context are intrinsically tied to the non-stationarity of the domain the robot traverses, and selecting an appropriate horizon $P$ depends on the velocity of the platform.
%
 Motivated by this setting, we experiment on multi-class problems with label drift. Specifically, data is stationary for the first $2500$ iterations during which it reduces to standard supervised learning. After the first $2500$ iterations, the data drifts and we require learning the classifier online.
 We fix the loss as  the multi-class hinge (SVM) loss $\ell_t(f(\bbx))$ as in \cite{Murphy2012}, and generate Gaussian Mixtures data akin to \cite{Zhu2005}.  The synthetic Gaussian Mixtures dataset for classification is generated in a manner similar to \cite{Zhu2005}. It consists of $N=5000$ feature-label pairs out of which last $2500$ are generated with the drift. For the first $2500$ points, we generate $\bbx_n \in \reals^p$ as $\bbx \given y \; \sim \; (1/3) \sum_{j=1}^3 \ccalN(\boldsymbol{\mu}_{y,j}, \sigma^2_{y,j}\bbI)$ where $\sigma^2_{y,j}=0.2$ for all values of $y$ and $j$, where also depends upon the class as $\boldsymbol{\mu}_{y,j} \sim \ccalN( \boldsymbol{\theta}_y, \sigma^2_y\bbI )$. The class mean value $\{ \bbtheta_i\}_{i=1}^C$ is placed around unit circle. We fix $\sigma_y^2=1.0$ and $C=5$. To add drift, after first $2500$ points, we shift each point to the right by $0.1$ at each instant which is clear from the video attached with the submission.
  Moreover, we focus on SVM for ease of interpretation. Its definition the multi-class context is given as
  \begin{align}
  \ell_t(\!f\!, \bbx_t, y_t\!) \!\!=\!\! \max(0, 1+f_r(\bbx_t)-f_{y_t}
  (\bbx_t)) +\lambda\sum\limits_{c'=1}^{C}\|f_{c'}\|_{\mathcal{H}}^2,\nonumber
  \end{align} 
  where $r=\arg\max_{c'\neq y_t}f_{c'}(\bbx)$.
  This definition is taken exactly from \cite{Murphy2012}.
  With dynamic class means during the drift phase: each mean shifts rightward by $0.1$ per step. %More details of data generation may be found in Appendix \ref{apx_experiments}.
 The results are presented in Fig.\ref{classification}: misclassifications over time is shown in  Fig.\ref{classification1}. DynaPOLK yields fewer mistakes in the non-stationary regime. Model complexity (Fig.\ref{classification2}) increases when the data is non-stationary, suggesting that it may be effective for change point detection. Fig. \ref{classification112} displays the learned decision surface on stationary data, and Fig. \ref{classification113} shows evolution to rightward-drifted data. Black dots denote dictionary elements and black lines are decision boundaries -- the supplementary classification video visualizes the classifier evolution.
 %
%  \begin{table}[h!]
% \resizebox{\columnwidth}{!}{ \begin{tabular}{|l|l|l|l|}
%  \hline
%  \textbf{Algorithms/Dataset} & \textbf{Movement} & \textbf{Devices} & \textbf{Activity} \\ \hline
%  Raker                       & 1.10              & 0.28          & 0.24              \\ \hline
%  AdaRaker                      & 1.10              & 0.36          & 0.34                    \\ \hline
%  \textbf{DyanPOLK}           & \textbf{}                 & \textbf{}         &    \textbf{}                 \\ \hline
%  \end{tabular}}
%  \caption{MSE ($10^{-3}$) performance of the different algorithms}
%  \label{comparisos_literature2}
%  \end{table}
%  
  %We display on a larger scale the decision surfaces yielded by DynaPOLK when faced with stationary and drifting data in Figure \ref{classification1222} for ease of interpretation. 
  As all the class means shift rightward, DynaPOLK is able to stably and accurately adapt its model. 
 
%  	%\subsection{Online classification}
%  	
%  	\begin{figure}
%  		\centering
%  		\subfigure[Stationary classifier at $t=1000$]{\includegraphics[width=0.49\linewidth,height=4cm]{image_1000.png}\label{classification11122}}	
%  		\subfigure[Drifted classifier at $t=5000$]{\includegraphics[width=0.49\linewidth,height=4cm]{image_4900.png}\label{classification11222}}
%  		\caption{Online classification}\vspace{-0mm}
%  		\label{classification1222}	
%  	\end{figure}
%  	
  	%
  	%
  	 	Further, we did an additional experiments on time-series classification \cite{street2001streaming} .  This dataset consists of $60000$ examples with $3$ features and $3$ classes. Features take values between $0$ and $10$, and the data is broken up into four blocks, where values of the features shift across the different blocks. See  \cite{street2001streaming}[Table 1] for more specific details. We report the results of comparing DynaPOLK to the alternatives mentioned in Sec. \ref{sec:experiments} in Figure \ref{mse_classification}. Specifically, we display the mean-square error, i.e., for each time, we compute misclassification square error and average it to the previous one. Note that DynaPOLK attains favorable performance.
 		
%  	 	
%  	 	
%   	\begin{figure}[h]
%   		\centering
%   {\includegraphics[scale=0.23]{mse_sea_comparisons.eps}}
%   \caption{MSE for classification problem }\label{mse_classification}
%   	\end{figure}	
  	
 %We further compare DynaPOLK with a DNN \cite{NIPS2007_3313,hinton2012improving} well-trained on MNIST \cite{lecun-mnisthandwrittendigit-2010}, which must be retrained online via backpropagation when it experiences drift. We consider MNIST with drift as in \cite{li2018kalman}: we add $1$ to all labels. 
 %
%Misclassifications over time is shown in Fig.~\ref{comp_dnn}: DynaPOLK makes fewer mistakes, as it adapts its model to changing data, a corollary of experiencing sublinear regret. %We used the MNIST dataset with the conceptual drift proposed in \cite{li2018kalman}. To add a drift, we change the labels of the MNIST dataset by adding $1$ to its labels. The comparison of the performance of a well trained DNN and the proposed algorithm on the drifted MNISt dataset is shown in Fig.  \ref{comp_dnn}.

%!TEX root = Paper.tex
%%%%%%%%%%%%%%%%%%%%%%%%%%%%%%%%%%%%%%%%%%%%%%%%%%%%%%%%%%%%%%%
%%%   S  E  C  T  I  O  N  %%%%%%%%%%%%%%%%%
\section{Conclusion}\label{sec:conclusion}
In this work, we focused on non-stationary learning, for which we proposed an online universal function approximator based on compressed kernel methods. We characterized its dynamic regret as well as its model efficiency, and experimentally observed it yields a favorable  tradeoffs for learning in the presence of non-stationarity. Future questions involve the development of model order as use for change point detection, improving the learning rates through second-derivative information, variance reduction, or strong convexity, and coupling it to the design of  learning control systems.

%acknowledgements
%\acks{We would like to acknowledge the appropriate people.}
%\begin{small}
%%
%!TEX root = Paper.tex
%%%%%%%%%%%%%%%%%%%%%%%%%%%%%%%%%%%%%%%%%%%%%%%%%%%%%%%%%%%%%%%
%%%   S  E  C  T  I  O  N  %%%%%%%%%%%%%%%%%
%%%%%%%%%%%%%%%%%%%%%%%%%%%%%%%%%%%%%%%%%%%%%%%%%%%%%%%%%%%%%%%
%\clearpage\newpage
%\begin{center}{\Large \bf Supplementary Material for \\  \vspace{2mm}
%Nonstationary Nonparametric Online Learning: \\ Balancing Dynamic Regret and Model Parsimony} 
%\end{center}

\appendices

 \section{Proof of Lemma \ref{theorem_model_order}}\label{proof_thm_4}
 
 %\red{In this section, we present the proof for a general batch size $B$ rather than $B=1$ as done for the earlier part of this paper. This is because all the other proofs are straightforward to write for any $B$ except the one we are going to discuss in this paper.}
 
  Before proving Lemma \ref{theorem_model_order}, we present a lemma which allows us to relate the stopping criterion of our sparsification procedure to a Hilbert subspace distance.
 %\red{discussion of lemma}
 %%%%%%%%%%%%%%%%%%%%%%%%%%%%%%%%%%%%%%%%%%%%%%%%%%%%%%%%%%%%%%%
 %%% L  E  M  M  A  %%%%%%%%%%%%%%%%%
 %%%%%%%%%%%%%%%%%%%%%%%%%%%%%%%%%%%%%%%%%%%%%%%%%%%%%%%%%%%%%%%
 %
 \begin{lemma}\label{lemma_subspace_dist}
 Define the distance of an arbitrary feature vector $\bbx$ evaluated by the feature transformation $\phi(\bbx) = \kappa(\bbx, \cdot)$ to $\ccalH_{\bbD}=\text{span}\{\kappa(\bbd_t, \cdot) \}_{t=1}^M$, the subspace of the Hilbert space spanned by a dictionary $\bbD$ of size $M$,  as
 \begin{align}\label{eq:hilbert_subspace_dist}
 \text{dist}( \kappa(\bbx, \cdot) , \ccalH_{\bbD}) 
 = \min_{f\in\ccalH_{\bbD}} \| \kappa(\bbx, \cdot) - \bbv^T \bbkappa_{\bbD}(\cdot) \|_{\ccalH} \; .
 \end{align}
 This set distance simplifies to following least-squares projection when $\bbD \in \reals^{p\times M}$ is fixed
 \begin{align}\label{eq:hilbert_subspace_dist_ls}
 \text{dist}( \kappa(\bbx, \cdot) , \ccalH_{\bbD}) 
 = \Big\|  \kappa(\bbx,\cdot) 
   - [\bbK_{\bbD, \bbD}^{-1} \bbkappa_{\bbD}(\bbx)]^T
    \bbkappa_{\bbD}(\cdot) \Big\|_{\ccalH} \; .
 \end{align}
 \end{lemma}

 %%%%%%%%%%%%%%%%%%%%%%%%%%%%%%%%%%%%%%%%%
 %%%%%%%%%%%%%%%%%%%%%%%%%%%%%%%%%%%%%%%%%
 %%%%%%%%    P   R   O   O   F    %%%%%%%%%%%%%%%%%%%%%%% 
 %%%%%%%%%%%%%%%%%%%%%%%%%%%%%%%%%%%%%%%%%
 %%%%%%%%%%%%%%%%%%%%%%%%%%%%%%%%%%%%%%%%%
 %\begin{proof} See Appendix 
 
 %We now present our main result, which says that the model order of $f_t$ remains bounded for all $t$, and depends on the conditioning of the problem setting as well as the feature space.
 
 %\subsection{Proof of Lemma \ref{lemma_subspace_dist}}\label{apx_lemma_subspace_dist}
 %%%%%%%%%%%%%%%%%%%%%%%%%%%%%%%%%%%%%%%%%
 %%%%%%%%%%%%%%%%%%%%%%%%%%%%%%%%%%%%%%%%%
 %%%%%%%%    P   R   O   O   F    %%%%%%%%%%%%%%%%%%%%%%% 
 %%%%%%%%%%%%%%%%%%%%%%%%%%%%%%%%%%%%%%%%%
 %%%%%%%%%%%%%%%%%%%%%%%%%%%%%%%%%%%%%%%%%
 \begin{proof}
 The distance to the subspace $\ccalH_{\bbD}$ is defined as
 \begin{align}\label{eq:subspace_dist}
 \text{dist}( \kappa(\bbx, \cdot) , \ccalH_{\bbD_t}) 
 =& \min_{f\in\ccalH_{\bbD}} \| \kappa(\bbx, \cdot) - \bbv^T \bbkappa_{\bbD}(\cdot) \|_{\ccalH}\nonumber\\ 
 =& \min_{\bbv\in \reals^{M}} \| \kappa(\bbx, \cdot) - \bbv^T \bbkappa_{\bbD}(\cdot) \|_{\ccalH} \; ,
 \end{align}
 where the first equality comes from the fact that the dictionary $\bbD$ is fixed, so $\bbv\in \reals^M$ is the only free parameter. Now plug in the minimizing weight vector $\tbv^\star=\bbK_{\bbD_t, \bbD_t}^{-1}\bbkappa_{\bbD_t}(\bbx_t)$ into \eqref{eq:subspace_dist} which is obtained in an analogous manner to the logic which yields \eqref{eq:hatparam_update}. Doing so simplifies \eqref{eq:subspace_dist} to the following
 \begin{align}\label{eq:subspace_dist2}
 \text{dist}(\kappa(\bbx_t, \cdot) , \ccalH_{\bbD_t}) \!=\!  \Big\|  \kappa(\bbx_t,\cdot) 
   - [\bbK_{\bbD_t, \bbD_t}^{-1} \bbkappa_{\bbD_t}(\bbx_t)]^T
    \bbkappa_{\bbD_t}(\cdot) \Big\|_{\ccalH} \; .
 \end{align}
 \end{proof}
 %%%%%%%%%%%%%%%%%%%%%%%%%%%%%%%%%%%%%%
 %%%%%%%%%%%%%%%%%%%%%%%%%%%%%%%%%%%%%
 %%%%%%%%%   S  E   C   T   I    O    N    %%%%%%%%%%%%%%%
 %%%%%%%%%%%%%%%%%%%%%%%%%%%%%%%%%%%%%
 %%%%%%%%%%%%%%%%%%%%%%%%%%%%%%%%%%%%%
 \subsection{Proof of Lemma \ref{theorem_model_order}}\label{apx_theorem_model_order}
 %%%%%%%%%%%%%%%%%%%%%%%%%%%%%%%%%%%%%%%%%
 %%%%%%%%%%%%%%%%%%%%%%%%%%%%%%%%%%%%%%%%%
 %%%%%%%%    P   R   O   O   F    %%%%%%%%%%%%%%%%%%%%%%% 
 %%%%%%%%%%%%%%%%%%%%%%%%%%%%%%%%%%%%%%%%%
 %%%%%%%%%%%%%%%%%%%%%%%%%%%%%%%%%%%%%%%%%
 %\begin{proof}
 %
 The proof is similar to that of \cite[Theorem 3]{koppel2019parsimonious} and provided here in detail for completeness. Consider the model order of the function iterates $f_t$ and $f_{t+1}$ generated by Algorithm \ref{alg:soldd} denoted by $M_t$ and $M_{t+1}$, respectively, at two arbitrary subsequent times $t$ and $t+1$. The number of elements in $\bbD_t$ are $M_t=(t-1)$. After performing the algorithm update at $t$, we add a new data points to the dictionary and increase the model order by one, hence $M_{t+1}=M_t+1$.
 %
% Specifically, rather than allowing the dimension of the Hilbert subspace to which $f$ belongs to grow by $1$ at each time as in \eqref{eq:param_update}, we propose to project them onto time-varying subspaces generated by dictionaries $\bbD=\bbD_{t+1}$ extracted from the data points observed thus far.
 
   Begin by assuming function ${f}_{t+1}$ is parameterized by dictionary $\bbD_{t+1}$ and weight vector $\bbw_{t+1}$. Moreover, we denote columns of $\bbD_{t+1}$ as $\bbd_t$ for $t=1,\dots,{M}_{t+1}$.
 %
 % Assume a constant algorithm step-size $\eta$ has been chosen such that $\eta<1/\lambda$ and the approximation budget $\eps$ satisfies $\eps=K \eta^{3/2}$ for some positive scalar $K>0$.
 %
 Suppose the model order of the function $f_{t+1}$ is less than or equal to that of $f_t$, i.e. $M_{t+1} \leq M_t$. This relation holds when the stopping criterion of KOMP, stated as $\min_{\{\{j=1,\dots,{M_t + 1}\}\}} \gamma_j > \eps$, \emph{is not} satisfied for the kernel dictionary matrix with the newest data point $\bbx_t$ appended: $\tbD_{t+1} = [\bbD_t ; \bbx_t ]$ [cf. \eqref{eq:param_tilde}], which is of size $M_t + 1$. 
 Thus, the negation of the termination condition of KOMP holds for this case, stated as
 \begin{align}\label{eq:komp_to_terminate}
 \min_{\{j=1,\dots,{M_t + 1}\}} \gamma_j \leq \eps \; .
 \end{align}
 Observe that the left-hand side of \eqref{eq:komp_to_terminate} lower bounds the approximation error $\gamma_{M_t + 1}$ of removing the recent batch of the feature vectors $\bbS_t$ due to the minimization over $j$, that is, $\min_{\{j=1,\dots,{M_t + 1}\}} \gamma_j \leq\gamma_{M_t + 1} $. Consequently, if $\gamma_{M_t + 1} \leq \eps$, then \eqref{eq:komp_to_terminate} holds and the model order does not grow. Thus it suffices to consider $\gamma_{M_t + 1}$.

 The definition of $\gamma_{M_t + 1}$ with the substitution of $\tilde{f}_{t+1}$ defined by \eqref{eq:param_tilde} allows us to write
 \begin{align}\label{eq:min_gamma_expand}
 \gamma_{M_t+1}
 			&=\!\!\!\min_{\bbu\in\reals^{{M_t}}} \Big\|(1\!-\!\eta\lambda){f}_t - \eta\nabla_f \check{L}_t (f_{t}(\bbS_t)) -\!\!\!\!\!\!\!\!\!\!\!\!\sum_{k \in \ccalI \setminus \{M_t + 1\}} \!\!\!\!\!\!\!\!\!\!\!\!u_k \kappa(\bbd_k, \cdot) \Big\|_{\ccalH} \nonumber
 			\\
 			&=\!\!\!\min_{\bbu\in\reals^{{M_t}}} \Big\|(1 - \eta\lambda) \!\!\!\!\!\! \!\!\!\sum_{k \in \ccalI \setminus \{M_t + 1\}} \!\!  \!\!\!\!\!w_k \kappa(\bbd_k, \cdot)- \eta\nabla_f \check{L}_t (f_{t}(\bbS_t))   \nonumber
 			\\
 			&\quad\quad\quad\quad\quad\quad\quad\quad\quad-\!\!\!\!\!\!\!\sum_{k \in \ccalI \setminus \{M_t + 1\}} \!\!\! \!\!\! u_k \kappa(\bbd_k, \cdot)\Big\|_{\ccalH}. \; 
 \end{align}
  The minimal error is achieved by considering the square of the expression inside the minimization and expand to get 
 \begin{align}\label{eq:error_expansion}
  %&
  \Big\|&(1\!\!- \!\!\eta\lambda) \!\!\! \!\!\!\!\!\!\!\!\sum_{k \in \ccalI \setminus \{M_t + 1\}} \!\!\!\!\!\!\!\!\!\!\!\!  w_k \kappa(\bbd_k, \cdot)- \eta\nabla_f \check{L}_t (f_{t}(\bbS_t)) -\!\!\!\!\!\!\!\!\!\!\!\! \sum_{k \in \ccalI \setminus \{M_t + 1\}} \!\! \!\!\!\!\!\!\!\! u_k \kappa(\bbd_k, \cdot)\Big\|_{\ccalH}^2 \nonumber \\
  &\!\!\!\!\!\!= \!(1\!-\!\eta\lambda\!)^2 \bbw^T \bbK_{\bbD_t, \bbD_t} \bbw 
  \!\!+\!\! \eta^2 (\nabla_f \check{L}_t (f_{t}(\bbS_t)))^2
  \!\!+\!\! \bbu^T \bbK_{\bbD_t, \bbD_t} \bbu  \nonumber 
  \\
  &\quad\!-\! 2(1\!-\!\eta\lambda\!) \eta\nabla_f \check{L}_t (f_{t}(\bbS_t))  \bbw^T\bbkappa_{\bbD_t}\!(\bbx_\tau) 
  \!
  \\
  &\qquad+2\eta\nabla_f \check{L}_t (f_{t}(\bbS_t))  \bbu^T\bbkappa_{\bbD_t}\!(\bbx_\tau)
  - 2 (\!1\!-\!\eta\lambda) \bbw^T\! \bbK_{\bbD_t, \bbD_t}\! \bbu  . \nonumber
  \end{align}
  To obtain the minimum, we compute the stationary solution of \eqref{eq:error_expansion} with respect to $\bbu \in \reals^{M_t}$ and solve for the minimizing $\tbu^\star$, which in a manner similar to the logic in \eqref{eq:hatparam_update}, is given as
  \begin{align}\label{eq:minimal_weights}
   \tbu^\star = (1-\eta\lambda) \bbw - \eta \bbK_{\bbD_t, \bbD_t}^{-1} \nabla_f \check{L}_t (f_{t}(\bbS_t))   \bbkappa_{\bbD_t}(\bbx_\tau) \; .
  \end{align}
 Plug $\tbu^\star$ in \eqref{eq:minimal_weights} into the expression in \eqref{eq:min_gamma_expand} and  using the short-hand notation ${f}_{t}(\cdot)=\bbw^T \bbkappa_{\bbD_t}(\cdot)$ and $\sum_k u_k \kappa(\bbd_k, \cdot)= \bbu^T \bbkappa_{\bbD_t}(\cdot)$. Doing so simplifies \eqref{eq:min_gamma_expand} to
 \begin{align}\label{eq:min_gamma_optimal_weights} 
  &\Big\| (1\!-\!\eta\lambda) \bbw^T\bbkappa_{\bbD_t}(\cdot) - \eta\nabla_f\check L_t'({f}_t(\bbS_t)) -  \bbu^T \bbkappa_{\bbD_t}(\cdot) \Big\|_{\ccalH} \\
  & =\! \Big\|(1\!-\!\eta B\lambda\!)\bbw^T\! \bbkappa_{\bbD_t}(\cdot) \!
   -\! \eta \nabla_f \check{L}_t (f_{t}(\bbS_t)) \kappa(\bbx_\tau,\cdot)\! \nonumber \\
  &\ - \Big[\!(1\!-\!\eta\lambda) \bbw \!-\! \eta \bbK_{\bbD_t, \bbD_t}^{-1} \nabla_f \check{L}_t (f_{t}(\bbS_t))   \bbkappa_{\bbD_t}(\bbx_\tau)\Big]^T 
  \!\!\bbkappa_{\bbD_t}(\cdot)  \Big\|_{\ccalH}  .\nonumber
  \end{align}
 The above expression may be simplified by canceling like terms {$(1-\eta\lambda)\bbw^T\! \bbkappa_{\bbD_t}(\cdot)$ }and collecting the like terms, we get 
 \begin{align}\label{eq:min_gamma_optimal_weights2}
  &\Big\|\!(1\!-\!\eta\lambda) \bbw^T\bbkappa_{\bbD_t}(\cdot) - \eta\nabla_f\check L_t'({f}_t(\bbS_t)) -  \bbu^T \bbkappa_{\bbD_t}(\cdot) \Big\|_{\ccalH} \\
  &=\! \eta\Big\| \nabla_{\!\!f} \!\check{L}_t (f_{t}(\bbS_t)) \Big[\!\kappa(\bbx_\tau,\cdot)\!-\! \eta [\bbK_{\bbD_t, \bbD_t}^{-1}   \bbkappa_{\bbD_t}(\bbx_\tau)]^T 
  \!\!\bbkappa_{\bbD_t}(\cdot)\Big] \! \Big\|_{\ccalH}\nonumber\\
  &\leq\! \eta \Big|\nabla_{\!\!f} \!\check{L}_t (f_{t}(\bbS_t))\Big| \!\cdot\! \Big\|\!\kappa(\bbx_\tau,\cdot)\!-\! \eta [\bbK_{\bbD_t, \bbD_t}^{-1}   \bbkappa_{\bbD_t}(\bbx_\tau)]^T 
  \!\!\bbkappa_{\bbD_t}(\cdot) \! \Big\|_{\ccalH}.\nonumber
  \end{align}
 %
% \red{We need to recheck this proof.}
 The second inequality in \eqref{eq:min_gamma_optimal_weights2} is achieved by the use of triangle  and Cauchy Schwartz inequality. 
 Notice that the right-hand side of \eqref{eq:min_gamma_optimal_weights2} may be identified as the distance to the subspace $\ccalH_{\bbD_t}$ in \eqref{eq:subspace_dist2} defined in Lemma \ref{lemma_subspace_dist} scaled by at most a factor of $P$ times {$\eta |\check\ell_\tau'({f}_\tau(\bbx_\tau)) | $}. We may write the right hand side of \eqref{eq:min_gamma_optimal_weights2} as
 \begin{align}\label{eq:min_gamma_optimal_weights3}
   \!\!\!\eta & \Big|\nabla_f \check{L}_t (f_{t}(\bbS_t))\Big| \cdot \Big\|\kappa(\bbx_\tau,\cdot)- \eta [\bbK_{\bbD_t, \bbD_t}^{-1}   \bbkappa_{\bbD_t}(\bbx_\tau)]^T 
   \!\!\bbkappa_{\bbD_t}(\cdot) \! \Big\|_{\ccalH} \nonumber\\
   &=\eta\Big|\nabla_f \check{L}_t (f_{t}(\bbS_t))\Big| \text{dist}(\kappa(\bbx_\tau,\cdot),\ccalH_{\bbD_t})
  \end{align}
 where we have applied \eqref{eq:hilbert_subspace_dist_ls} regarding the definition of the subspace distance on the right-hand side of \eqref{eq:min_gamma_optimal_weights3} to replace the Hilbert-norm term. Now, when the KOMP stopping criterion is violated, i.e., \eqref{eq:komp_to_terminate} holds, which implies $\gamma_{M_t + 1} \leq \eps$. Therefore, the right-hand side of \eqref{eq:min_gamma_optimal_weights3} is upper-bounded by $\epsilon$, we can write
 \begin{align}\label{eq:min_gamma_optimal_weights5}
   \eta |\nabla_f \check{L}_t (f_{t}(\bbS_t))| \text{dist}(\kappa(\bbx_t,\cdot),\ccalH_{\bbD_t})\leq \epsilon.
 \end{align}
  After rearranging the terms in \eqref{eq:min_gamma_optimal_weights5}, we write
 \begin{align}\label{eq:min_gamma_optimal_weights4}
    \text{dist}(\kappa(\bbx_t,\cdot),\ccalH_{\bbD_t}) \leq \frac{\epsilon}{\eta|\nabla_f \check{L}_t (f_{t}(\bbS_t))|} \;,
  \end{align}
 where we have divided both sides by {$\eta|\nabla_f \check{L}_t (f_{t}(\bbS_t)) |$}. 
 %
 %Using this identification, we transform the sufficient condition for the stopping criterion of KOMP to be violated, stated as $\gamma_{M_t + 1} \leq \eps$, into a criterion on $\text{dist}(\kappa(\bbx_t,\cdot),\ccalH_{\bbD_t})$, the subspace distance of $\kappa(\bbx_t,\cdot)$ to the span of kernel evaluations of the current dictionary $\ccalH_{\bbD_t}$.
 %%
 %%
 %%\red{more explanation here}
 %Substituting the definition \eqref{eq:subspace_dist2} into $\gamma_{M_t + 1} \leq \eps$ and dividing both sides by $\eta |\ell_t'({f}_t(\bbx_t)) |$ yields
 %%
 %\begin{align}\label{eq:newest_gamma_rescaled}
 %\text{dist}(\kappa(\bbx_t,\cdot),\ccalH_{\bbD_t}) \leq \frac{\epsilon}{\eta|\ell'(f_t(\bbx_t),\bby_t)|} \; .
 %\end{align}
 %%
 %%Now use the approximation budget selection in terms of the learning rate $\eta$ as $\epsilon=K\eta^{3/2}$. Furthermore, the $C$-Lipschitz continuity of $\ell$ [cf. \eqref{eq:lipschitz}] in Assumption \ref{as:2} allows us to bound the instantaneous gradient by this same constant. Inverting this expression yields $1/|\ell'(f_t(\bbx_t),\bby_t)| \geq 1/C$. Substituting in this lower bound and selection of $\eps$, we obtain that if
 %Now, divide the above expression by $C \eta$ to write
 %%%
 %\begin{align}\label{eq:newest_gamma_rescaled2}
 %%
 %\text{dist}(\kappa(\bbx_t,\cdot),\ccalH_{\bbD_t}) \leq \frac{K\sqrt{\eta}}{C}
 %%
 %\end{align}
 %%
 Observe that if \eqref{eq:min_gamma_optimal_weights4} holds, then $\gamma_{M_{t+1}} \leq \eps$ holds, but since $\gamma_{M_{t+1}} \geq \min_{j} \gamma_j $, we may conclude that \eqref{eq:komp_to_terminate} is satisfied. Consequently the model order at the subsequent step does not grow $M_{t+1} \leq M_t$ whenever \eqref{eq:min_gamma_optimal_weights4} is valid. 
 
 Now, let's take the contrapositive of the preceding expressions to observe that growth in the model order ($M_{t+1} = M_t + 1$) implies that the condition
 \begin{align}
  \label{eq:min_gamma2}
  \text{dist}(\kappa(\bbx_t,\cdot),\ccalH_{\bbD_t}) > \frac{\epsilon}{\eta|\nabla_f \check{L}_t (f_{t}(\bbS_t))|}
  \end{align} %\vspace{-5cm}
 holds.  Therefore, each time a new point is added to the model, the corresponding kernel function is guaranteed to be at least a distance of {$\frac{\epsilon}{\eta|\check L_t'({f}_t(\bbx_t)) |}$} from every other kernel function in the current model.

 By the $C$-Lipschitz continuity of the instantaneous loss (Assumption \ref{as:2}): specifically {$1/|\nabla_f \check{L}_t (f_{t}(\bbS_t))| \geq 1/HC$}, we can lower-bound the threshold condition in \eqref{eq:min_gamma2} as 
 \begin{align}\label{eq:min_gamma3}
   \frac{\epsilon}{\eta|\check\ell_t'({f}_t(\bbx_t)) |}\geq \frac{\epsilon}{\eta CH}
  \end{align}
 We have
  \begin{align}
  \label{eq:min_gamma321}
  \text{dist}(\kappa(\bbx_t,\cdot),\ccalH_{\bbD_t}) > \frac{\epsilon}{\eta CH}
  \end{align}
 Therefore, For a fixed compression budget $\epsilon$ and step size $\eta$, the KOMP stopping criterion is violated for the newest point whenever distinct dictionary points $\bbd_k$ and $\bbd_j$  for $j,k\in\{1,\dots,M_t\}$, satisfy the condition $\|\phi(\bbd_j) - \phi(\bbd_k) \|_\mathcal{H} > \frac{\epsilon}{\eta CH}$. Next, we follow the similar argument as provided in the proof of Theorem 3.1 in \cite{1315946}.   Since $\ccalX$ is compact and $\kappa$ is continuous, the range $\phi(\ccalX) $ (where $\phi(\bbx)=\kappa(\bbx,\cdot)$ for $\bbx \in \ccalX$) of the kernel transformation of feature space $\ccalX$ is compact.   Therefore, the number minimum of balls (covering number) of radius $\delta$ (here, $\delta = \frac{\epsilon}{\eta CH}$) needed to cover the set $\phi(\ccalX)$ is finite (see, e.g., \cite{anthony2009neural}) for a fixed compression budget $\epsilon$ and step-size $\eta$.
 
 % Further for the non-stationary settings, we select the compression budget $\epsilon$ to be inversely proportional to $T$ in order to make regret sublinear, which implies that the radius of $\delta$ balls is going to zero with $T$. Hence, the covering number approaches infinite. Under this condition, it becomes important to study the rate at which the size of the dictionary increases with respect to decrease in the compression budget. 
  To arrive at the characterization \eqref{model_order_0}, we note that \cite[Proposition 2.2]{1315946} states that for a Lipschitz continuous Mercer kernel $\kappa$ on compact set $\mathcal{X}\subseteq\mathbb{R}^p$, there exists a constant $Y$ such that for any training set $\{\bbx_t\}_{t=1}^\infty$ and any $\nu>0$, and it holds for  the number of elements in dictionary that  %\blue{Really interesting... this result is definitely missing from the literature: exactly relating model complexity to compression budget/step-size selection is really interesting. That we have to better accentuate in the {\bf Contribution} discussion in Sec. \ref{sec:Problem}. }
  \begin{align}\label{model_order_1}
  M\leq Y\left(\frac{1}{\nu}\right)^p.
  \end{align}
 where $Y$ is a constant depends upon $\mathcal{X}$ and the kernel function. By \eqref{eq:min_gamma321}, we have that $\nu=\frac{\epsilon}{\eta CH}$, which we may substitute into \eqref{model_order_1} to obtain
  \begin{align}\label{model_order_2}
   M\leq Y(CH)^p\left(\frac{\eta}{\epsilon}\right)^p.
   \end{align}
   as stated in \eqref{model_order_0}. The lower bound $H$ in \eqref{model_order_0} comes from the fact that to represent the instantaneous gradient of the windowed loss of length $H$, a minimum of $H$ points are required.  \hfill $\qed$
%  Then, the model order of the learned function will depend on the ratio of the step-size to the compression budget, problem constants such as Lipschitz continuity modulus and the parsimony constant.
  % From the above expression, we will get different upper bounds for the different decrement rate of the compression budget  which we will choose in accordance with to minimize the regret. The result in \eqref{model_order_1} holds directly for POLK and related the tradeoff between stepsize and compression budget for POLK. This result was missing from POLK. 
 %
 %\end{proof}

\footnotesize
\bibliographystyle{IEEEtran}
\bibliography{IEEEabrv,bibliography}

\normalsize

%!TEX root = Paper.tex
%%%%%%%%%%%%%%%%%%%%%%%%%%%%%%%%%%%%%%%%%%%%%%%%%%%%%%%%%%%%%%%
%%%     S  E  C  T  I  O  N    %%%%%%%%%%%%%%%%%
%%%%%%%%%%%%%%%%%%%%%%%%%%%%%%%%%%%%%%%%%%%%%%%%%%%%%%%%%%%%%%%
%
\newpage\onecolumn
\section*{Supplementary Material for \\ ``Nonstationary Nonparametric Online Learning: \\ Balancing Dynamic Regret and Model Parsimony"}

% \appendix

%%%%%%%%%%%%%%%%%%%%%%%%%%%%%%%%%%%%%
%%%%%%%%%%%%%%%%%%%%%%%%%%%%%%%%%%%%%
\section{Preliminary Technical Results}\label{apx_assumptions}
%             
%\blue{The appendix needs some work.

%(1) We need to explain what is the purpose of some of the technical lemmas/propositions that are being proved. They're all needed in the proofs of Theorem 1 and 2. We need to explain for which steps specifically they help. 

%(2) Moreover, for the analysis of static regret, we need to clarify at the beginning of Section A that these results are a precursor to analyzing how things look in the non-stationary setting. }
%Also, if there's no difference between the parameter selection strategy in the stationary and dynamic cases, then I think we can remove the one for stationary setting, and simply reference the non-stationary one which we have included in the main body of the paper. Then, we can explain that the strategies are the same. For example, Table IV and Table V being exactly the same is strange and one should be removed.
%\blue{Please go through and make sure we didn't mix up $\check{L}_t$ with $L_t$ anywhere. I think we switched it to $L_t$, right?}
%\red{Alec the analysis for the static regret}
{Next, we establish some technical conditions in terms of Preposition \ref{prop_bounded}, Preposition \ref{prop1}, and Lemma \ref{lemma1}.which are essential to the ensuing proofs of Theorem \ref{theorem:dynamic_cost}  and Theorem \ref{theorem:dynamic_path}. For instance, the result of Proposition \ref{prop_bounded}  is utilized in \eqref{eq:iterate_prop1_subst32} and \eqref{eq:expectation_convexity26}, the result of Proposition \ref{prop1} is used in \eqref{proof_Strong_12}, and the statement of Lemma \ref{lemma1} is utilized in \ref{eq:iterate_prop1_subst0}. } 
%
%In the following subsection, we shift to our main result, which establishes the dynamic regret behavior of Algorithm \ref{alg:soldd} in terms of the cost function variation \eqref{variation} and variable variation \eqref{vtdy} defined in Section \ref{sec:Problem}.
%
%
%
%%%%%%%%%%%%%%%%%%%%%%%%%%%%%%%%%%%%%%%%%%%%%%%%%%%%%%%%%%%%%%%
%%%   P  R  O  P  O  S  I  T  I  O  N   %%%%%%%%%%%%%%%%%
%%%%%%%%%%%%%%%%%%%%%%%%%%%%%%%%%%%%%%%%%%%%%%%%%%%%%%%%%%%%%%%
\begin{proposition}\label{prop_bounded}
Let Assumptions \ref{as:first}-\ref{as:3} hold and denote $\{f_t \}$ as the sequence generated by Algorithm \ref{alg:soldd} with $f_0 = 0$. Further, denote $f^\star$ as the optimum defined by \eqref{static}. Both quantities are bounded by the constant $K:={C X}/{\lambda}$ in Hilbert norm for all $t$ as
\begin{align}\label{eq:prop_bounded}
\| f_t \|_{\ccalH} \leq \frac{C X}{\lambda} \; , \qquad \|f^\star \|_{\ccalH} \leq \frac{C X}{\lambda}
\end{align}
\end{proposition}
%
%\textcolor{blue}{If the proof is same as POLK then we can just cite the JMLR paper, no need to re-prove. The same is true for Proposition \ref{prop1}.} \red{The steps are almost similar but I kept them to emphasize that it doesn't change even for a time varying loss functions.}
The  proof of Proposition \ref{prop_bounded} is similar to in \cite[Proposition 7]{koppel2019parsimonious} but adapted to the distribution-free non-stationary case considered here.% provided in Appendix \ref{proof_prop_1}.
%%%%%%%%%%%%%%%%%%%%%%%%%%%%%%%%%%%%%
%%%%%%%%%   S  E   C   T   I    O    N    %%%%%%%%%%%%%%%
%%%%%%%%%%%%%%%%%%%%%%%%%%%%%%%%%%%%%
%%%%%%%%%%%%%%%%%%%%%%%%%%%%%%%%%%%%%
%\section{Proof of Proposition \ref{prop_bounded}}\label{proof_prop_1}
%
\begin{proof} 
Since we repeatedly use the Cauchy-Schwartz inequality together with the reproducing kernel property in the following analysis, we here note that for all $g\in\ccalH$,  $|g(\bbx_t) | \leq | \langle g, \kappa(\bbx_t, \cdot) \rangle_{\ccalH} | \leq X \| g \|_{\ccalH}$. Now, consider the magnitude of $f_1$ in the Hilbert norm, given $f_0 = 0$
{\begin{align}\label{eq:iterate_opt_bound}
\| f_1 \|_{\ccalH} 
&= \Big\| \ccalP_{ \ccalH_{\bbD_{1}}} \Big[\eta_0
 \nabla_f\check{\ell}(0) \Big] \Big\|_{\ccalH} \nonumber \\
 &\leq \eta_0 \|  \nabla_f\check{\ell}(0) \|_{\ccalH}
  \leq \eta_0 |\check{\ell}'(0)| \|\kappa(\bbx_0, \cdot)\|_{\ccalH} \nonumber \\
&  \leq \eta_0 C X < \frac{C X}{\lambda}.
\end{align}}
The first equality comes from substituting in $f_0=0$ and the second inequality comes from the definition of optimality condition of the projection operator and the homogeneity of the Hilbert norm, and the chain rule applied to definition of the functional stochastic gradient in with the Cauchy-Schwartz inequality. Lastly, we make use of Assumptions \ref{as:first} and \ref{as:2} to bound the scalar derivative $\check\ell'$ using the Lipschitz constant, and the boundedness of the kernel map [cf. \eqref{eq:bounded_kernel}]. The final strict inequality in \eqref{eq:iterate_opt_bound} comes from applying the step-size condition $\eta_0 < 1/\lambda$.

Now we consider the induction step. Given the induction hypothesis $\|f_t \|_{\ccalH} \leq CX/\lambda$, consider the magnitude of the iterate at the time $t+1$ as
{\begin{align}\label{eq:iterate_opt_bound_induction1}
\|f_{t+1} \|_{\ccalH}& = \Big\|\ccalP_{ \ccalH_{\bbD_{t+1}}} \Big[
(1-\eta H\lambda) f_t 
- \eta \nabla_f\check{L}_t(f_{t}(\bbS_t)) \Big] \Big\|_{\ccalH} \nonumber 
\\ 
&\leq \|(1-\eta H\lambda) f_t 
- \eta \nabla_f\check{L}_t(f_{t}(\bbS_t)) \|_{\ccalH} \nonumber 
\\
&\leq (1-\eta H\lambda)\| f_t \|
+ \eta \| \nabla_f\check{L}_t(f_{t}(\bbS_t)) \|_{\ccalH} \; ,
\end{align}}
where we have applied the non-expansion property of the projection operator for the first inequality on the right-hand side of \eqref{eq:iterate_opt_bound_induction1}, and the triangle inequality for the second. Now, apply the induction hypothesis $\|f_t \|_{\ccalH} \leq CX/\lambda$ to the first term on the right-hand side of \eqref{eq:iterate_opt_bound_induction1}, and the chain rule together with the triangle inequality to the second to obtain
\begin{align}\label{eq:iterate_opt_bound_induction2}
\|f_{t+1} \|_{\ccalH}
&\leq (1-\eta H\lambda) \frac{C X}{\lambda}
+ \eta \sum\limits_{\tau=t-H+1}^{t} | \check \ell_t'(f_{t}(\bbx_t)) | \|\kappa(\bbx_t, \cdot)\|_\ccalH
\nonumber \\ 
&\leq (\frac{1}{\lambda} -\eta H) CX 
+ \eta HC X = \frac{C X}{\lambda}
\end{align}
where we have made use of Assumptions \ref{as:first} and \ref{as:2} to bound the scalar derivative $\check\ell'$ using the Lipschitz constant, and the boundedness of the kernel map [cf. \eqref{eq:bounded_kernel}] as in the base case for $f_1$, as well as the fact that $\eta < 1/(H\lambda)$. The same bound holds for $f^\star$ by applying \cite{Kivinen2004}[ Section V-B ] with $m\rightarrow \infty$.
 \end{proof}

Next we introduce a proposition which quantifies the error due to subspace projections in terms of the ratio of the compression budget to the learning rate. %The proof of Proposition \ref{prop1} is detailed in Appendix \ref{proof_prop_2}.
%
%%%%%%%%%%%%%%%%%%%%%%%%%%%%%%%%%%%%%%%%%%%%%%%%%%%%%%%%%%%%%%%
%%%   P  R  O  P  O  S  I  T  I  O  N   %%%%%%%%%%%%%%%%%
%%%%%%%%%%%%%%%%%%%%%%%%%%%%%%%%%%%%%%%%%%%%%%%%%%%%%%%%%%%%%%%
\begin{proposition}\label{prop1}
Fix an independent realization $\bbx_t$ that parameterizes the loss $L_t$ at time $t$. Then the difference between the projected online functional gradient and the un-projected online functional gradient of the regularized loss by \eqref{eq:proj_grad} and \eqref{eq:stoch_reg_grad}, respectively, is bounded for all $t$ as
{\begin{align}\label{eq:prop1}
 \| \tilde{\nabla}_fL_t(f_{t}(\bbS_t)) - {\nabla}_fL_t(f_{t}(\bbS_t)) \|_{\ccalH} \leq \frac{\eps}{\eta}
\end{align}}
where $\eta>0$ denotes the algorithm step-size and $\eps>0$ is the compression parameter of Algorithm \ref{alg:soldd}.
\end{proposition}
%
% \section{Proof of Proposition \ref{prop1}}\label{proof_prop_2}
 \begin{proof}
 Consider the square-Hilbert-norm difference of $ \tilde{\nabla}_fL_t(f_{t}(\bbS_t)) $  and ${\nabla}_fL_t(f_{t}(\bbS_t)) $ defined in \eqref{eq:stoch_reg_grad} and \eqref{eq:proj_grad}, respectively,
\begin{align}\label{eq:norm_stoch_grad_diff}
  \| &\tilde{\nabla}_f L_t(f_{t}(\bbS_t)) - {\nabla}_fL_t(f_{t}(\bbS_t)) \|_{\ccalH}^2  \\
 & = \Big\| \Big( f_{t} -\ccalP_{ \ccalH_{\bbD_{t+1}}} \Big[
   f_t 
 - \eta {\nabla}_fL_t(f_{t}(\bbS_t)) \Big]\Big)/\eta 
  - {\nabla}_fL_t(f_{t}(\bbS_t)) \Big\|_{\ccalH}^2 \nonumber
 \end{align}
 Multiply and divide ${\nabla}_fL_t(f_{t}(\bbS_t))$, the last term, by $\eta$, and reorder terms to write
\begin{align}\label{eq:norm_stoch_grad_expand}
 &\Big\| \Big( f_{t} -\ccalP_{ \ccalH_{\bbD_{t+1}}} \Big[  f_t 
 - \eta {\nabla}_fL_t(f_{t}(\bbS_t)) \Big]\Big)/\eta 
  - {\nabla}_fL_t(f_{t}(\bbS_t)) \Big\|_{\ccalH}^2 \nonumber 
  \\
 & = \Big\| \frac{1}{\eta}\left( f_{t} \!-\!\eta {\nabla}_fL_t(f_{t}(\bbS_t))\right) \!-\!\frac{1}{\eta}\ccalP_{ \ccalH_{\bbD_{t+1}}} \Big[  f_t - \eta {\nabla}_fL_t(f_{t}(\bbS_t)) \Big]  \Big\|_{\ccalH}^2 \nonumber 
 \\
 &=\frac{1}{\eta^2}\| \tilde{f}_{t+1} - f_{t+1} \|_{\ccalH}^2 
   \end{align}
   where we have substituted the definition of $\tilde{f}_{t+1}$ and $f_{t+1}$ \eqref{eq:projection_hat}, and pulled the nonnegative scalar $\eta$ outside the norm. Now, note that the KOMP stopping criterion in Algorithm \ref{alg:soldd} is $\lVert \tilde{f}_{t+1} - f_{t+1} \rVert_{\ccalH} \leq \epsilon$, which we apply to the last term on the right-hand side of \eqref{eq:norm_stoch_grad_expand} to conclude \eqref{eq:prop1}.
 \end{proof}

%
%\blue{Lemma \ref{lemma1} should be different in the static vs. dynamic case, right? } \red{Yes, but I have used the fact that if we replace $f^\star$ by $f^\star_t$ in Lemma 1, we get the same inequality. Or we can do that other way, present the Lemma 1 with $f^\star$ and then use $f6\star$ when required for static analysis.}
Next we establish that Algorithm \ref{alg:soldd} yields a sequence that satisfies a standard descent relation in the statement of Lemma \ref{lemma1}, via convexity and smoothness of the cost functions. %The detailed proof of Lemma \ref{lemma1} is provided in Appendix \ref{proof_lemma_1}.
%%%%%%%%%%%%%%%%%%%%%%%%%%%%%%%%%%%%%%%%%%%%%%%%%%%%%%%%%%%%%%%
%%%  L  E  M  M  A   %%%%%%%%%%%%%%%%%
%%%%%%%%%%%%%%%%%%%%%%%%%%%%%%%%%%%%%%%%%%%%%%%%%%%%%%%%%%%%%%%
\begin{lemma} \label{lemma1}
Consider the sequence generated $\{f_t \}$ by Algorithm \ref{alg:soldd} with $f_0 = 0$. Under Assumptions \ref{as:first}-\ref{as:4}, the following online descent relation holds for a static comparator $f^\star$ as defined in \eqref{static}.
{\begin{align}\label{eq:lemma1}
\| f_{t+1} - f^\star \|_{\ccalH}^2
	 &\leq\| f_t - f^\star \|_{\ccalH}^2
	 - 2 \eta [L_t(f_t(\bbS_t)) -L_t(f^\star(\bbS_t))]  \nonumber \\
	 &\quad+ 2 \eps \| f_t - f^\star \|_{\ccalH} 
	+ \eta^2 \| \tilde{\nabla}_fL_t(f_{t}(\bbS_t)) \|_{\ccalH}^2 \; .
\end{align}}
\end{lemma}
With Propositions \ref{prop_bounded} - \ref{prop1} and Lemma \ref{lemma1} stated, we may now establish some basic results that form the backbone of our dynamic regret analysis to come.
%
 
% \section{Proof of Lemma \ref{lemma1}}\label{proof_lemma_1}
 %%%%%%%%%%%%%%%%%%%%%%%%%%%%%%%%%%%%%%%%%%%%%%%%%%%%%%%%%%%%%%%
 %%%  P  R  O  O F  %%%%%%%%%%%%%%%%%
 %%%%%%%%%%%%%%%%%%%%%%%%%%%%%%%%%%%%%%%%%%%%%%%%%%%%%%%%%%%%%%%
 \begin{proof} 
 Begin by considering the square of the Hilbert-norm difference between $f_{t+1}$ and $f^\star$ defined by \eqref{static}, and expand the square to write
 \begin{align}\label{eq:iterate_square_expand}
  \| f_{t+1} - f^\star \|_{\ccalH}^2 
  	=& \| f_t - \eta \tilde{\nabla}_fL_t(f_{t}(\bbS_t)) - f^\star\|_{\ccalH}^2  \\
  	 =&\|  f_t - f^\star \|_{\ccalH}^2 - 2 \eta \langle f_t - f^\star, \tilde{\nabla}_fL_t(f_{t}(\bbS_t)) \rangle_{\ccalH} \nonumber
  	 \\
  	 &+ \eta^2 \| \tilde{\nabla}_fL_t(f_{t}(\bbS_t)) \|_{\ccalH}^2.\nonumber
  \end{align}
 Add and subtract the gradient of the regularized instantaneous risk ${\nabla}_fL_t(f_{t}(\bbS_t))$ defined in \eqref{eq:stoch_reg_grad} to the second term on the right-hand side of \eqref{eq:iterate_square_expand} to obtain
 \begin{align}\label{eq:iterate_stoch_grad}
 \| f_{t+1} - f^\star \|_{\ccalH}^2 
 	 &=\| f_t - f^\star \|_{\ccalH}^2
 	 - 2 \eta \langle f_t - f^\star, {\nabla}_fL_t(f_{t}(\bbS_t)) \rangle_{\ccalH} \nonumber
 	  \\
 	 &\quad- 2 \eta \langle f_t-f^\star, \tilde{\nabla}_fL_t(f_{t}(\bbS_t)) - {\nabla}_fL_t(f_{t}(\bbS_t)) \rangle_{\ccalH} \nonumber
 	  	  \\
 	  	 & \quad\quad
 	+ \eta^2 \| \tilde{\nabla}_fL_t(f_{t}(\bbS_t)) \|_{\ccalH}^2. 
 \end{align}
 We deal with the third term on the right-hand side of \eqref{eq:iterate_stoch_grad}, which represents the directional error associated with the sparse stochastic projections, by applying the Cauchy-Schwartz inequality together with Proposition \ref{prop1} to obtain
 \begin{align}\label{eq:iterate_prop1_subst}
  \| f_{t+1} - f^\star \|_{\ccalH}^2 
  	 \leq&\| f_t - f^\star \|_{\ccalH}^2
  	 - 2 \eta \langle f_t - f^\star, {\nabla}_fL_t(f_{t}(\bbS_t)) \rangle_{\ccalH} \nonumber
  	  \\
  	 &+ \!2 \eps \| f_t \!-\! f^\star \|_{\ccalH} 
  	\!+\! \eta^2 \| \tilde{\nabla}_fL_t(f_{t}(\bbS_t)) \|_{\ccalH}^2 .
  \end{align}
 From the convexity of loss function at each $t$, it holds that
 {\begin{align}\label{eq:convexity}
 L_t(f_t(\bbS_t)) -L_t(f^\star(\bbS_t)) \leq \langle f_t - f^\star, {\nabla}_f L_t(f_{t}(\bbS_t)) \rangle_{\ccalH}  \;,
 \end{align}}
 which we substitute into the second term on the right-hand side of the relation given in \eqref{eq:iterate_prop1_subst} to obtain
 \begin{align}\label{eq:expectation_convexity}
 \| f_{t+1} \!-\! f^\star \|_{\ccalH}^2
 	 \leq&\| f_t - f^\star \|_{\ccalH}^2
 	 - 2 \eta [L_t(f_t(\bbS_t)) -L_t(f^\star(\bbS_t))]  \nonumber \\
 	  	&\!+ 2 \eps \| f_t\! -\! f^\star \|_{\ccalH} \!+\! \eta^2 \| \tilde{\nabla}_fL_t(f_{t}(\bbS_t)) \|_{\ccalH}^2 \; . 
 \end{align}
 as stated in Lemma \ref{lemma1}.
  \end{proof}
 
 %
%

%\blue{we need to be careful to include the time indexing of $\ell_t$ throughout. Otherwise these definitions don't quite make sense/don't match the functional analogue of \cite{shalev2012online}. Also, note that I've changed the notation to define the regularized $\ell_t$ and un-regularized loss $\check{\ell}_t$, since I think it's clearer to talk about these things from the beginning. I've tried to propagate that change forward but please check if I've missed points.. }
 %
 Here, to establish some baseline results that are helpful in the dynamic setting, we characterize the static regret of Algorithm \ref{alg:soldd}. {We are interested in deriving the dynamic regret bounds for DynaPOLK in terms of the function variations $\mathcal{V}_T$. To achieve that, the following static regret analysis is precursor and presented here in detail for the completeness. }
% We design methods such that $ \textbf{Reg}_T^S$ grows sublinearly in horizon $T$ for a given sequence $\{f_t\}$, i.e., the average regret goes to null with $T$ ( no-regret \cite{Zinkevich2003}). %Such methods are called no-regret algorithms, and a canonical approach to achieving no-regret is online gradient descent \cite{Zinkevich2003}. 
%
%The goal of this work is the development of no-regret algorithms in the case where the optimal action varies with \emph{each time}, motivated by learning with non-stationarity. 
%
%  
\subsection{Static Regret}\label{sec:static}
The classical performance metric for an action sequence $\{f_t\}_{t=1}^{T}$ is its cost accumulation as compared with a best single action in hindsight $f^\star$, defined as the static regret:
\begin{align}\label{static}
\textbf{Reg}^S_T=\sum_{t=1}^{T}\ell_{t}(f_{t}(\bbx_t))-\sum_{t=1}^{T}\ell_{t}(f^\star(\bbx_t)) \; ,  
\end{align}
where $f^\star= \argmin_{f \in \ccalH}\sum_{t=1}^T \ell_t(f(\bbx_t))$. We begin by defining some key quantities to simplify the analysis and clarify the technical setting for which our regret bounds are valid. To be specific, define the regularized online gradient as
%
%\blue{We shouldn't use $\check{a}$ notation since there's no stochastic approximation going on. it looks misleading. I've changed the definition in \eqref{eq:stoch_reg_grad} to reflect the modification of Section 1.}
{\begin{align}\label{eq:stoch_reg_grad}
{\nabla}_fL_t(f_{t}(\bbS_t)) =\nabla_f\check{L}_t(f_{t}(\bbS_t))  + \lambda  f_t
\end{align}}
%
%\blue{please make sure to change the instantaneous loss throughout to $\check{\ell}_t$ so we can distinguish it from the regularized version $\ell_t$. We could also swap the notation such that $\check{\ell}_t$ is regularized and ${\ell}_t$ is not. It doesn't matter, but right now it was ambiguous.}
and its projected variant associated with the step defined in \eqref{eq:projection_hat}:
\begin{align}\label{eq:proj_grad}
\tilde{\nabla}_fL_t(f_{t}(\bbS_t)) =\Big( f_{t} - \ccalP_{ \ccalH_{\bbD_{t+1}}} \Big[
 f_t 
- {\eta}_t {\nabla}_fL_t(f_{t}(\bbS_t)) \Big]\Big)/\eta
\end{align}
such that the Hilbert space update of Algorithm \ref{alg:soldd} [cf. \eqref{eq:projection_hat}] may be expressed as an online projected gradient step 
\begin{align}\label{eq:iterate_tilde}
f_{t+1} = f_t - \eta \tilde{\nabla}_fL_t(f_{t}(\bbS_t)) \; .
\end{align}

The definitions \eqref{eq:proj_grad} - \eqref{eq:stoch_reg_grad} will be used to analyze the convergence behavior of the algorithm.
%
%{\begin{align}\label{eq:unbiased}
%\mathbb{E}[\check{\nabla}_fL_t(f_{t}(\bbS_t)) \given \ccalF_t ] =\nabla_f R_t(f_{t})
%\end{align}}
%

We first establish a foundational result for the subsequent analysis of the non-stationary setting, which is conditions under which Algorithm \ref{alg:soldd} is asymptotically no-regret. This result is stated next.
%
%%%%%%%%%%%%%%%%%%%%%%%%%%%%%%%%%%%%%%
%%%%%%%%%%%%%%%%%%%%%%%%%%%%%%%%%%%%%
%%%%%%%%%   T  H  E  O  R  E  M    %%%%%%%%%%%%%%%
%%%%%%%%%%%%%%%%%%%%%%%%%%%%%%%%%%%%%
%%
\begin{theorem}\label{theorem:static}
Suppose $\{f_t\}\subset \ccalH$ is the function sequence generated by Algorithm \ref{alg:soldd} for $T$ iterations. Then for regularization parameter $\lambda>0$, with step-size $\eta<\min(1/(H\lambda),1/L)$, under Assumptions \ref{as:first}-\ref{as:3}, we have the following regret bound:
\begin{align}\label{eq:static_regret_thm}
	\textbf{Reg}^S_T&\leq\frac{\| f_1 - f^\star \|_{\ccalH}^2}{2 \eta }+  \frac{2\epsilon T CX}{\eta\lambda} 
		  		+ \frac{\epsilon^2T}{\eta}+\frac{\eta Z^2 T}{2} \; \nonumber\\
				&= \mathcal{O}\left(\frac{1+( \epsilon +\epsilon^2) T}{\eta}+\eta T\right).  
\end{align}
{where $Z:=CH(1+X)$}. Then, the static regret grows sublinearly $\textbf{Reg}^S_T\leq \ccalO(\sqrt{T})$ in $T$ for step-size selection $\eta=\ccalO(T^{-1/2})$ and compression budget $\epsilon > \ccalO(T^{-1/2})$.
\end{theorem}
%
%
%\blue{Did we define $Z$ somewhere prior to the expression \eqref{eq:static_regret_thm}? I couldn't find the definition anywhere in the draft thus far.}
%\end{theorem}

%
%The proof of Theorem \ref{theorem:static} is provided in Appendix \ref{proof_thm_1}. 
%
 %
% \section{Proof of Theorem \ref{theorem:static}}\label{proof_thm_1}
 %
 %%%%%%%%%%%%%%%%%%%%%%%%%%%%%%%%%%%%%%%%%%%%%%%%%%%%%%%%%%%%%%%
 %%%  P  R  O  O F  %%%%%%%%%%%%%%%%%
 %%%%%%%%%%%%%%%%%%%%%%%%%%%%%%%%%%%%%%%%%%%%%%%%%%%%%%%%%%%%%%%
 \begin{proof}
 
 Begin by noting that
 \begin{align}\label{eq:iterate_prop1_subst011}
 \| &\tilde{\nabla}_f L_t(f_{t}(\bbS_t)) \|_{\ccalH}
 \\
 &\quad=\| \tilde{\nabla}_f L_t(f_{t}(\bbS_t))-{\nabla}_f L_t(f_{t}(\bbS_t))+{\nabla}_fL_t(f_{t}(\bbS_t))\|_{\ccalH}\nonumber
 \end{align}
 where we add subtract the term { $ {\nabla}_fL_t(f_{t}(\bbS_t))$}. Using Cauchy-Schwartz inequality and the result of Proposition \ref{prop1}, we get
 \begin{align}\label{proof_Strong_111}
 \| &\tilde{\nabla}_f L_t(f_{t}(\bbS_t)) \|_{\ccalH}^2\nonumber
 \\
 & \leq  \left(\| \tilde{\nabla}_f L_t(f_{t}(\bbS_t))-{\nabla}_f L_t(f_{t}(\bbS_t))\|_{\ccalH}+\|
 {\nabla}_f L_t(f_{t}(\bbS_t)) \|_{\ccalH}\right)^2\nonumber \\
 & \leq  2\| \tilde{\nabla}_f L_t(f_{t}(\bbS_t))-{\nabla}_f L_t(f_{t}(\bbS_t))\|_{\ccalH}^2+2\|
 {\nabla}_f L_t(f_{t}(\bbS_t)) \|_{\ccalH}^2
 \nonumber\\ 
 &\leq \frac{2\epsilon^2}{\eta^2}+2\|{\nabla}_f L_t(f_{t}(\bbS_t))\|^2.
 \end{align}
 %
 %For a Lipschitz continuous gradient function Assumption \ref{as:4}, we have %\textcolor{blue}{We'll need to potentially address the case where the Lipschitz constant of both the function and its gradient are not a fixed constant. This is because the functions could change so much across time...Is this done anywhere in the literature on dynamic regret? Otherwise, I think we'll have to think a bit more about how to justify the uniform Lipschitz constants over time.} \red{In literature, mainly a constant parameter is always used at all the places as far as I konw. We may look into this in further detail.}
 %\begin{align}
 %\|{\nabla}_fL_t(f_{t}(\bbS_t))\|^2\leq 2L[L_t(f_t(\bbS_t)) -L_t(f_t^\star(\bbS_t))],
 %\end{align}
 %which implies that
 %\begin{align}\label{proof_Strong_1121}
 %\| \tilde{\nabla}_fL_t(f_{t}(\bbS_t)) \|_{\ccalH}^2\leq&\frac{2\epsilon^2}{\eta^2}+2L[L_t(f_t(\bbS_t)) -L_t(f_t^\star(\bbS_t))] .
 %\end{align}
 %
 Substitute the upper bound obtained in \eqref{proof_Strong_111} into the descent property stated in Lemma\ref{lemma1} [c.f. \eqref{eq:expectation_convexity}] to obtain
 \begin{align}\label{eq:expectation_convexity2222}
 \|f_{t+1} \!\!-\!\! f^\star \|_{\ccalH}^2
 	 \!\leq&\| f_t \!-\! f^\star \|_{\ccalH}^2
 	 \!-\! 2 \eta [L_t(f_t(\bbS_t)) \!-\!L_t(f^\star(\bbS_t))]  
 	  	 \\
 	  	 &+ 2 \eps \| f_t - f^\star \|_{\ccalH}+{2\epsilon^2}+2\eta^2\|{\nabla}_f L_t(f_{t}(\bbS_t))\|^2.   \nonumber
 	% =&\| f_t - f^\star \|_{\ccalH}^2
 	% 	 - 2 \eta (1-\eta L)[L_t(f_t(\bbS_t)) -\ell_t(f^\star(\bbx_t))]  + 2 \eps \| f_t - f^\star \|_{\ccalH}+{2\epsilon^2}	  
 \end{align}
 From the definition of loss function $\ell_t(\cdot)$, we have
 \begin{align}\label{gradient_0}
 \|{\nabla}_f L_t(f_{t}(\bbS_t))\|=&\|{\nabla}_f\check L_t(f_{t}(\bbS_t))+\lambda H f_{t}(\bbS_t)\|\nonumber 
 \\
 =&\|{\nabla}_f\check L_t(f_{t}(\bbS_t))\|+\lambda H\| f_{t}(\bbS_t)\|.
 \end{align}
 The Assumption \ref{as:2} implies that $\|{\nabla}_f\check L_t(f_{t}(\bbS_t))\|\leq CH$ and from the result of Proposition \ref{prop_bounded}, we get
 \begin{align}\label{bound_gradient}
 \|{\nabla}_f L_t(f_{t}(\bbS_t))\|\leq CH(1+X).
 \end{align}
 Let us define $Z=CH(1+X)$ and utilize the upper bound in \eqref{bound_gradient} into \eqref{eq:expectation_convexity2222}, we get 
 \begin{align}\label{eq:expectation_convexity222222}
 \| f_{t+1} - f^\star \|_{\ccalH}^2
 	 \leq&\| f_t - f^\star \|_{\ccalH}^2
 	 - 2 \eta [L_t(f_t(\bbS_t)) -L_t(f^\star(\bbS_t))]  \nonumber 
 	 \\ 
 	 &+ 2 \eps \| f_t - f^\star \|_{\ccalH}+{2\epsilon^2}+2\eta^2Z^2 .
 	% =&\| f_t - f^\star \|_{\ccalH}^2
 	% 	 - 2 \eta (1-\eta L)[L_t(f_t(\bbS_t)) -\ell_t(f^\star(\bbx_t))]  + 2 \eps \| f_t - f^\star \|_{\ccalH}+{2\epsilon^2}	  
 \end{align}
  After rearranging the terms in \eqref{eq:expectation_convexity222222}, we get 
 \begin{align}\label{eq:expectation_convexity22222}
 \!\!\!\!\!2 \eta [L_t(f_t(\bbS_t)) \!-\!L_t(f^\star(\bbS_t))]
 	 \!\leq&\| f_t \!-\! f^\star \|_{\ccalH}^2
 	 	 - \| f_{t+1} - f^\star \|_{\ccalH}^2   
 	 	  	 \\ 
 	 	  	 &+ 2 \eps \| f_t - f^\star \|_{\ccalH}+{2\epsilon^2}+\eta^2Z^2.\nonumber	  
 \end{align}
 Next, divide both sides of \eqref{eq:expectation_convexity22222} by $2\eta$, we get
 \begin{align}\label{eq:expectation_convexity2}
 	  [& L_t(f_t(\bbx_t)) -L_t(f^\star(\bbx_t))]
 	  \\
 	  &\leq\frac{\| f_t \!-\! f^\star \|_{\ccalH}^2}{2 \eta }-\frac{\| f_{t+1} \!-\! f^\star \|_{\ccalH}^2}{2 \eta }
 	   +  \frac{\eps}{\eta } \| f_t \!-\! f^\star \|_{\ccalH} 
 	+ \frac{\epsilon^2}{\eta }\!+\!\frac{\eta Z^2}{2}\; \nonumber\\
 	&\leq\frac{\| f_t - f^\star \|_{\ccalH}^2}{2 \eta}-\frac{\| f_{t+1} - f^\star \|_{\ccalH}^2}{2 \eta }
 		   +  \frac{\eps}{\eta} \cdot\frac{2CX}{\lambda} 
 		+\frac{\epsilon^2}{\eta}+ \frac{\eta Z^2}{2} \;\nonumber . 
 \end{align}
 We apply the upper bound obtained in Proposition \ref{prop_bounded} to get the second inequality in \eqref{eq:expectation_convexity2}. 
 {Taking the sum from $t=1$ to $T$ in \eqref{eq:expectation_convexity2} and dropping the negative terms from the right hand side, we get}
 \begin{align}\label{eq:expectation_convexity22}
 	  \sum_{t=1}^{T}[L_t(f_t(\bbx_t)) -L_t(f^\star(\bbx_t))]
 	  \leq&\frac{\| f_1 - f^\star \|_{\ccalH}^2}{2 \eta }+  \frac{2\epsilon T CX}{\eta\lambda} \nonumber
 	  \\
& 	  
 	  		+ \frac{\epsilon^2T}{\eta}+\frac{\eta Z^2 T}{2} \; . 
 \end{align}
 For a fixed step size $\eta$ such that $\eta<\frac{1}{\lambda}$,  we obtain the static regret as
 \begin{align}\label{static_Regret}
  \textbf{Reg}_T^S=\mathcal{O}\left(\frac{1+( \epsilon +\epsilon^2) T}{\eta}+\eta T\right)
 % %
  \end{align}
 %   %
which grows sublinearly in $T$ for $\eta = \ccalO(T^{-1/2})$ and $\epsilon\in[0,T^{-1/2})$. For instance, $\textbf{Reg}_T^S\leq\mathcal{O}(\sqrt{T})$ if we select $\eta=\frac{1}{\sqrt{T}}$ and $\epsilon=\frac{1}{T}$. However, if instead we select $\eta$ and $\epsilon$ as follows
 \begin{align}\label{selection0}
  &\eta=\mathcal{O}(T^{-a}) \ \ \ \text{and}\ \ \  \eps=\mathcal{O}(T^{-b}).
 \end{align}
then the right-hand side of \eqref{static_Regret} becomes % The values of the positive constants $a$ and $b$ will be decided later.  We discuss the $\eta$ and $\epsilon$ selection for each of the case separately and corresponding upper bound on the number of elements in the dictionary.
 %
% \begin{itemize}
% \item [{Case 1}] \textbf{Static regret convex case:} The expression for static regret in terms of $V_T$ and convex loss function is given by 
% \begin{align}\label{last_11}
%  \textbf{Reg}_T^S=&\mathcal{O}\left(\frac{1+\epsilon T}{\eta}+\eta T\right)
% \end{align}
% Using selection in \eqref{selection}, it holds that 
 \begin{align}\label{last_111}
  \textbf{Reg}_T^S=&\mathcal{O}\left(T^b+T^{(1-(a-b))}+T^{1-b}\right) \; .
 \end{align}
 When we substitute \eqref{selection0} into Lemma \ref{theorem_model_order}, we obtain
  \begin{align}\label{eq:model_order_substitution_stepsizes}
                 M= \mathcal{O}(T^{p(a-b)}).
                 \end{align}
                 For the static regret to be sublinear, we need $b\in(0,1)$ and $a\in(b,b+\frac{1}{p})$. As long as the dimension  $p$ is not too large, we always have a range for $a$. This implies that $p(a-b)\in(0,1)$ to ensure sublinear model order $M$ as well.% We summarize these observations in the following table.
 \end{proof}

Theorem \ref{theorem:static} establishes that Algorithm \ref{alg:soldd} exhibits the asymptotic no-regret property under appropriate step-size and compression budget selections, and that there is a tunable tradeoff between regret and memory: for tighter regret, one should set $\epsilon \in [0,T^{-1/2})$ to be closer to null. However, making $\epsilon$ too small causes the model order to grow as the time horizon $T$. This result exactly matches regret growth for parametric settings with $\epsilon=0$. The largest compression budget allowable that still yields asymptotic no-regret is $\epsilon=T^{-1/2} - \gamma$ for $\gamma >0$. Observe that by setting $\epsilon=\mathcal{O}(T^{-a})$ and $\eta=\mathcal{O}(T^{-b})$, we have that sublinear regret for $b\in(0,1)$ and $a\in(b,b+\frac{1}{p})$. To have the model complexity under control, via \eqref{eq:model_order_substitution_stepsizes}, this imposes constraints $p(a-b)\in(0,1) $. These observations are summarized in Table \ref{first_table}.
%\red{It in interesting to see the tradeoff between the regret rate and the model order $M$ (cf. Lemma \ref{theorem_model_order}. For instance, let $\epsilon=\mathcal{O}(T^{-a})$ and $\eta=\mathcal{O}(T^{-b})$ with $b\in(0,1)$ and $a\in(b,b+\frac{1}{p})$, we obtain the different values mentioned in Table \ref{first_table}.  }
%
\begin{table}[]
             	\centering
                {\begin{tabular}{|c|c|c|c|}
                		\hline
                		$(a,b)$& Regret & Model Order $M$ & Comments	 \\
                		\hline
                		 $a=b$& $\mathcal{O}(T)$ & $\mathcal{O}(1)$                  	& \text{Linear regret} \\
                		\hline
                		$a-b=1/p$& $\mathcal{O}(T^{(p-1)/p})$ & $\mathcal{O}(T)$                  	& \text{Linear Model Order} \\
                		                		\hline
                		                		$a-b={1}/{(1+p)}$& $\mathcal{O}(T^{p/(1+p)})$ & $\mathcal{O}(T^{p/(1+p)})$                  	& \text{Sublinear Order Model \emph{and} Regret} \\
                		                		                		                		\hline
                	\end{tabular}}             	
                	\caption{Summary of convergence rates for different parameter selections.}%\blue{How does this relate to usual $T^{-1/2}$ selection? That we need to clarify, to help people understand.} \red{I am not sure how to do it.}}
                	\label{first_table}
                \end{table}

 %%%%%%%%%%%%%%%%%%%%%%%%%%%%%%%%%%%%%
%%%%%%%%%   S  E   C   T   I    O    N    %%%%%%%%%%%%%%%
%%%%%%%%%%%%%%%%%%%%%%%%%%%%%%%%%%%%%
%%%%%%%%%%%%%%%%%%%%%%%%%%%%%%%%%%%%%

 \section{Selection of the step size $\eta$ and $\epsilon$ for the Dynamic Case}\label{seelctoion_tradeoff}
In this section, we provide the detailed analysis of the parameters $(\eta,\epsilon)$ selection and the relative effect on the regret performance and model order.   First of all, we collect all the conditions for $\eta$ and $\epsilon$ imposed by the analysis in the paper. We have 
 \begin{align}
 \eta<\frac{1}{\lambda} \ \  \ \  \ \text{and}\ \ \ \ \eta <\frac{1}{L}.
 \end{align}
  Next, we combine the conditions and get
 \begin{align}
 \eta<\min\Bigg\{\frac{1}{\lambda},\frac{1}{L}\Bigg\}
 \end{align}
 which implies that there is an upper bound on the value of the step size. {Similarily, for the strongly convex case, we have  
  $\eta<\min\{\frac{1}{\lambda},\frac{\mu}{L^2}\}$.}
 
 Before proceeding with the analysis, let us make a common selection for $\epsilon$ as follows
 \begin{align}\label{selection}
   \eps=\mathcal{O}(T^{-\alpha}).
 \end{align}
 The values of the positive constant $\alpha$ will be decided later.  We discuss the $\epsilon$ selection for each of the case separately and corresponding upper bound on the number of elements in the dictionary.

 \textbf{(1) Dynamic regret convex case in terms of $V_T$}
 
 \begin{align}\label{dynamic_regret_vT21}
 \textbf{Reg}^D_T
 \leq& \ceil{\frac{T}{\triangle_T}}\mathcal{O}\left(\frac{1+\epsilon \triangle_T}{\eta}+\eta \triangle_T\right)+2\triangle_TV_T.
 \end{align}
 We can select it in a similar way to Case I with $\epsilon=\mathcal{O}(\triangle_T^{-\alpha})$ and $\eta=\mathcal{O}(\triangle_T^{-\beta})$.
 
 \textbf{(2) Dynamic regret, convex case, in terms of $W_T$:} The expression for dynamic regret in terms of $W_T$ and convex loss function is given by
        \begin{align}\label{eq:expectation_convexity321}
             	   \textbf{Reg}_T^D\leq&\left( {1 + \red{T\sqrt{\epsilon} } + W_T }  \right).                	   	 	   	 	 	 
                 \end{align}
                 %\red{Note that for the dynamic regret case with the convex loss function, we select \red{$\eta$} which is independent of $T$.} 
                 Using the $\epsilon$ selection in \eqref{selection}, we can write 
                 \begin{align}\label{last_1111}
                  \textbf{Reg}_T^D=&\mathcal{O}\left(1+T^{\red{(1-\frac{a}{2})}}+W_T\right).
                 \end{align}
                 and \begin{align}
                                 M= & \mathcal{O}(T^{\alpha p}).
                                 \end{align}
                                 \red{For the regret and the model order to be sublinear up to the variations $W_T$, we need $\alpha\in(0,\frac{1}{p}]$. As long as the dimension  $p$ is not too large, we always have a range for $a$. This implies that $\red{\alpha p\in(0,1)}$ and hence $M$ is sublinear. }

                                   \begin{table}[!h]
                                                 	\centering
                                                 	\begin{tabular}{|c|c|c|c|}
                                                 		\hline
                                                 		$\alpha$& Regret & M & Comments	 \\
                                                 		\hline
                                                 		 \red{$\alpha=0$}& $\mathcal{O}(T)+W_T$ & $\mathcal{O}(1)$                  	& \text{Linear regret} \\
                                                 		\hline
                                                 		$\red{{\alpha}=\frac{1}{p}}$& $\mathcal{O}\left(T^{\frac{(2p-1)}{2p}}+W_T\right)$ & $\mathcal{O}(T)$                  	& \text{Linear } $M$ \\
                                                 		       		\hline
                                                		\red{$\alpha=\frac{1}{p+1}$}& $\mathcal{O}\left(T^{\frac{2p+1}{2p+2}}+W_T\right)$ & $\mathcal{O}(T^{p/(1+p)})$                  	& \text{Sublinear } $M$ \\
                                                 		                		                		                		\hline
                                                 	\end{tabular}\vspace{2mm}
                                                 	
                                                 	\caption{Summary of dynamic regret rates for convex loss function. }
                                                 	\label{tab:my_label2}
                                                 \end{table}
                 
 \textbf{(4) Dynamic regret, strongly convex case, in terms of $W_T$:} The expression for dynamic regret in terms of $W_T$ and strongly convex loss function is given by
                 \begin{align}\label{eq:iterate_prop1_subst1341}
                  \textbf{Reg}_T^D\leq&\mathcal{O}\left( \frac{1 + \red{T\sqrt{\epsilon}} + W_T }{1-\rho}  \right)\nonumber 
                  \\
                  \leq & o\left( {1 + \red{T\sqrt{\epsilon}} + W_T }\right).
                 \end{align}
                 where $\rho=\sqrt{(1-2\eta(\mu-\eta L^2))}$. The expression in \eqref{eq:iterate_prop1_subst1341} is similar to the one on \eqref{eq:expectation_convexity321} except for the term $(1-\rho)$ in the denominator. If we choose $\eta$ such that $(1-\rho)>\eta$, the results for strongly convex functions is improved. Rearrange this expression to obtain
                 $$ (1 - \eta)^2 >\rho^2 = 1 - 2 \eta (\mu - \eta L^2)$$
                 which, upon solving for a condition on $\eta$, simplifies to 
                 $$\eta < \frac{2 (\mu - 1)}{2 L^2 -1}.$$
\red{                 Next, we summarize the dynamic regret rates achieved for a constant $\eta$ and $\epsilon=\mathcal{O}(T^{-\alpha})$ with different $\alpha$ in Table \ref{tab:my_label222}. }
                 %Of course, any $\eta=\ccalO(T^{-a})$ satisfies the preceding expression for sufficiently large $T$.
                 %
                 %
                                                    \begin{table}[!h]
                                                                  	\centering
                                                                  	\begin{tabular}{|c|c|c|c|}
                                                                  		\hline
                                                                  		$\alpha$& Regret & M & Comments	 \\
                                                                  		\hline
                                                                  		 \red{$\alpha=0$}& $o(T)+W_T$ & $o(1)$                  	& \text{Linear regret} \\
                                                                  		\hline
                                                                  		$\red{{\alpha}=\frac{1}{p}}$& ${o}\left(T^{\frac{(2p-1)}{2p}}+W_T\right)$ & ${o}(T)$                  	& \text{Linear } $M$ \\
                                                                  		       		\hline
                                                                 		\red{$\alpha=\frac{1}{p+1}$}& ${o}\left(T^{\frac{2p+1}{2p+2}}+W_T\right)$ & ${o}(T^{p/(1+p)})$                  	& \text{Sublinear } $M$ \\
                                                                  		                		                		                		\hline
                                                                  	\end{tabular}\vspace{2mm}
                                                                  	
                                                                  	\caption{Summary of dynamic regret rates for strongly convex loss function. }
                                                                  	\label{tab:my_label222}
                                                                  \end{table}

\end{document}